 \documentclass{amsart}

\usepackage{amsmath, mathabx}
\usepackage{amssymb,multicol}
\usepackage{amsthm, mathrsfs}
\usepackage[all]{xypic}
\usepackage{amsfonts}
\usepackage{fullpage}
\usepackage{verbatim}
\usepackage{url}
\usepackage{hyperref}
\usepackage[mathscr]{euscript}

\numberwithin{equation}{section} 
\theoremstyle{plain}
\newtheorem{theorem}[equation]{Theorem}
\newtheorem{proposition}[equation]{Proposition}
\newtheorem{corollary}[equation]{Corollary}
\newtheorem{lemma}[equation]{Lemma}

\newtheorem{conjecture}[equation]{Conjecture}

\newtheorem{sublemma}[equation]{Sublemma}

 \newtheorem{claim}[equation]{Claim}

\theoremstyle{definition}
\newtheorem{definition}[equation]{Definition}
\newtheorem{defn}[equation]{Definition}
\newtheorem{notation}[equation]{Notation}
\newtheorem{hypothesis}[equation]{Hypothesis}

\newtheorem{example}[equation]{Example}

\theoremstyle{remark}
\newtheorem{remark}[equation]{Remark}

\newcommand{\DMO}{\DeclareMathOperator}

\newcommand{\beq}{\begin{equation}}
\newcommand{\eeq}{\end{equation}}

\newcommand{\ang}[1]{\langle #1 \rangle}

\newcommand{\bbar}[1]{\overline{#1}}

\providecommand{\abs}[1]{\lvert#1\rvert}

\DeclareMathOperator{\id}{{injdim}}
\DeclareMathOperator{\Hom}{{Hom}}

\DeclareMathOperator{\End}{{End}}
\DeclareMathOperator{\Ext}{{Ext}}

\DeclareMathOperator{\Tor}{Tor}

\DeclareMathOperator{\Proj}{Proj}

\DeclareMathOperator{\codim}{codim}

\DeclareMathOperator{\pd}{pd}
\DeclareMathOperator{\Ker}{Ker}
\DeclareMathOperator{\coker}{Coker}
\DeclareMathOperator{\im}{Im}

\DeclareMathOperator{\chrr}{char}
\newcommand{\Skl}{{S\hskip -.5 pt kl}}
\newcommand{\QVB}{{Q_{V\hskip -1.5pt dB}}}

\DeclareMathOperator{\GK}{GKdim}
\DeclareMathOperator{\gr}{gr}
\DeclareMathOperator{\Div}{Div}

\DeclareMathOperator{\Ann}{Ann}
\newcommand{\lann}{\ell\text{-}\Ann}

\newcommand{\mc}{\mathcal}

 \DeclareMathOperator{\Fr}{Fr}
 
 \newcommand{\scr}{\mathscr}
 \newcommand{\Kk}{\scr{K}}
 \newcommand{\Ss}{\scr{S}}
 \newcommand{\Xx}{\scr{X}}
 \newcommand{\Yy}{\scr{Yy}}
 \newcommand{\Cc}{\scr{C}}
 \newcommand{\Dd}{\scr{D}}
 \newcommand{\Mm}{\scr{M}}
  \newcommand{\Ll}{\scr{L}}
    \newcommand{\Ii}{\scr{I}}
   \newcommand{\Aa}{\scr{A}}
      \newcommand{\Bb}{\scr{B}}
      \newcommand{\Ee}{\scr{E}}
          
         \newcommand{\Nn}{\scr{N}}
   \newcommand{\Pp}{\scr{P}}
   \newcommand{\Zz}{\scr{Z}}

\newcommand{\kk}{{\Bbbk}}

\newcommand{\ZZ}{{\mathbb Z}}

\newcommand{\PP}{{\mathbb P}}
\newcommand{\NN}{{\mathbb N}}

\newcommand{\mb}{\mathbb}
\newcommand{\wt}{\widetilde}

\newcommand{\sB}{\mc{B}}

\newcommand{\sL}{\mc{L}}

\newcommand{\sO}{\mc{O}}

\DeclareMathOperator{\rgr}{gr-\!}

\DeclareMathOperator{\rmod}{mod-\!}

\DeclareMathOperator{\rqgr}{qgr-\!}

\DeclareMathOperator{\GKdim}{GKdim}

\DeclareMathOperator{\SK}{Skl_4}

\newcommand{\Tinf}{T_{\infty}}

\newcommand{\ssm}{\smallsetminus}

\newcommand{\uHom}{\underline{\Hom}}
\newcommand{\uEnd}{\underline{\End}}
\DMO{\hilb}{hilb}
\DMO{\gldim}{gldim}
\DMO{\Tot}{Tot}
\DeclareMathOperator{\Sing}{Sing}

 \newcommand{\wc}{\widecheck}
\newcommand{\wh}{\widehat}
 \newcommand{\too}{\longrightarrow}

\DMO{\grk}{grk}
\DMO{\pdim}{pdim}
\DMO{\Soc}{Soc}
\newcommand{\injdim}{\id}

\title{Some noncommutative minimal surfaces}

\author{D. Rogalski,  S. J. Sierra, and J. T. Stafford}
\address{(Rogalski)
Department of Mathematics, UCSD, La Jolla, CA 92093-0112, USA. }
\email{drogalski@ucsd.edu}
 \address{(Sierra) School of Mathematics, 
University of Edinburgh, Edinburgh EH9 3FD, Scotland.}
\email{s.sierra@ed.ac.uk}
\address{(Stafford) School of Mathematics,  The University of Manchester,   Manchester M13 9PL,
England.}
\email{Toby.Stafford@manchester.ac.uk}

\thanks{The first author was partially supported by the NSF grant DMS-1201572 and the NSA grant H98230-15-1-0317.}

 \thanks{The second author was  partially   supported by EPSRC grant EP/M008460/1.}

\date{\today}

\subjclass[2010]{Primary: 14A22,  16P40,    16S38, 16W50; Secondary:  14H52, 14E30.}
\keywords{Noncommutative projective geometry,   Sklyanin algebras,
noncommutative   minimal models}

\begin{document}

 \begin{abstract}  In the ongoing programme to classify noncommutative projective surfaces 
 (connected graded noetherian domains of Gelfand-Kirillov 
  dimension three) a natural question is to determine the minimal models within any birational class.

%\smallskip
 \noindent 
  In this paper we show that the generic noncommutative projective plane   (corresponding to the three 
  dimensional Sklyanin algebra)  as well as  noncommutative analogues of  $\mathbb{P}^1\times\mathbb{P}^1$ and the 
  more general  Van den Bergh quadrics
   satisfy very strong minimality conditions. Translated 
  into an algebraic question, where one is interested in a maximality condition, we prove the following result.
  
  \smallskip
{\bf Theorem A:} {\it Let $R$ be  a   Sklyanin algebra or  a Van den Bergh quadric  that is infinite dimensional over its centre
 and let  $A\supseteq R$ be any 
connected graded noetherian maximal order, with the same graded quotient 
ring as $R$. Then, up to taking Veronese rings, $A$ is isomorphic to $R$.} 

\smallskip
\noindent
Let $T$ be an elliptic algebra (that is, the coordinate ring of a noncommutative surface containing an elliptic curve).  Then,   
under an  appropriate  homological  condition,   we prove that  \emph{every connected graded noetherian overring  of $T$ is obtained by blowing 
down finitely many lines (line modules) of self-intersection~$(-1)$.}
 \end{abstract} 

  \maketitle

  \tableofcontents

  %%%%%%%%%%%%%%%%%
 \section{Introduction}\label{section-intro}

 The classification  of noetherian, connected graded domains $R$ of Gelfand-Kirillov dimension 3 (or the corresponding noncommutative  
 surfaces, written $\rqgr R$) is one of the major open problems in noncommutative algebraic geometry.  The classification
 has been solved in many particular cases and those solutions have led to some fundamental advances in the subject; 
 see, for example, \cite{ATV1990,  RSSlong, KRS, SV, VdB-blowups,VdB3}. In \cite{Ar}, 
 Artin conjectured that, birationally at least, there is a short list of such surfaces. More precisely, the corresponding division rings of rational functions are either: (1)  finite over their centre; (2) Ore extensions of $\kk(C)$ for a curve $C$; or (3) the 
 division ring associated to      a Sklyanin algebra $\Skl$, as defined  in Example~\ref{eg:sklyanin}.
 Artin's conjecture is completely open, but  this then leaves the question of classifying the algebras in each birational class. Case (1) was largely resolved in \cite{KRS, RS, S-surfclass},   while it is expected that case (2) will be answered by an amalgam of the methods developed  cases (1) and (3).  Thus a fundamental part of the classification problem, and the motivating question for this paper  is:  
 
 \centerline{\em \quad What are the connected graded noetherian algebras $R$ that are birational to a Sklyanin algebra $\Skl$?}    
 
 \noindent In this paper, such algebras are always assumed to be infinite dimensional over their centres.

 A natural approach to this problem is to follow the commutative classification of rational surfaces, which we briefly review.
  Here, one first classifies the {\em minimal models}:  smooth projective surfaces $X$ with the property that  any birational morphism from $X$ to a smooth projective surface $Y$ is   an isomorphism \cite[p.~175]{Shafarevich}.
It is a consequence of Zariski's Main Theorem and Castelnuovo's contraction criterion that:
\begin{theorem}\label{ithm:A}
A smooth projective surface $X$ is a minimal model if and only if $X$ contains no lines of self-intersection $(-1)$.
\end{theorem}

In fact much more is true.   
\begin{theorem}\label{ithm:B}
{\rm (\cite[Corollary~V.5.4]{Ha})}
 If $X$ and $Y$ are smooth projective surfaces, then any birational morphism $X \to Y$ factors as a composition of finitely many monoidal transformations {\rm (}contractions of  lines of self-intersection $(-1)${\rm)}.
\end{theorem}
\begin{theorem}\label{ithm:C}
{\rm (\cite[Theorem~V.5.8, Remark~V.5.8.4]{Ha})}
\begin{enumerate}
\item Any smooth projective surface $X$ has a birational morphism to a minimal model.
\item The minimal rational surfaces are known:  they are $\PP^2$ and Hirzebruch surfaces $\mb F_n$ for $n \neq 1$. 
\end{enumerate}
\end{theorem}

In this paper we move the classification programme for noncommutative surfaces forward by giving noncommutative analogues of Theorems~\ref{ithm:A} and Theorem~\ref{ithm:B} for algebras birational to $\Skl$.

 We note that noncommutative  analogues of blowing up points   are understood (see  \cite{VdB-blowups, R-Sklyanin, RSSlong}) and this has, for example, been used to classify the noetherian subalgebras of $\Skl$ that are birational to that algebra \cite{RSSlong, Hipwood-thesis}. Moreover, a noncommutative analogue of contraction or blowing down  has been  developed in     \cite{ RSSblowdown},  showing in particular that one really can contract  lines (or more formally line modules) of self-intersection $(-1)$.  
Beyond that, what one might call the birational geometry of noncommutative projective surfaces is wide open.  
In particular,  it is important to understand  whether noncommutative  minimal models exist  and, indeed,  how they should be defined.   Whatever the definition, one would expect that the Sklyanin algebra $\Skl$ and Van den Bergh's quadrics $\QVB$ are indeed examples   (see Example~\ref{eg:quadric} and  \cite{VdB3} for the definition).

 The main aim of this paper is to show these algebras are indeed   minimal models; indeed they satisfy a minimality property that is far stronger than is possible in the commutative situation. This is provided by Theorem~A from the abstract.   
The reason why this is described in terms of overrings will be explained below but we will continue to use the geometric notation since we feel this gives the better intuition for the results proved here.

 To formalise these statements we need some definitions.
In the paper, all rings will be algebras over a fixed algebraically closed field $\kk$, while in 
  the introduction we also assume that $\kk$ has characteristic zero. 
A $\kk$-algebra $R$ is \emph{connected graded}\label{cg-defn} or \emph{cg}
 if $R=\bigoplus_{n\geq 0}R_n$ is a finitely generated 
$\mathbb{N}$-graded algebra  with $R_0=\kk$.
For such a ring $R$, the category of graded noetherian right 
$R$-modules will be denoted $\rgr R$ with quotient category $\rqgr R$ \label{rqgr-defn}
 obtained by quotienting out the Serre subcategory 
of finite dimensional modules. An effective intuition    is  to regard  $\rqgr R$ as the category of coherent sheaves on
 the (nonexistent) space $\Proj(R)$. 
The graded quotient ring $Q_{gr}(R)$ of $R$ is obtained by inverting the 
 non-zero homogeneous elements and  two such  domains  $R$ and $ S$ are \emph{birational} if  
 $Q_{gr}(R)_0 \cong Q_{gr}(S)_0$.  
 
 As is explained in \cite{R-Sklyanin, RSSlong},  when one  works with algebras related to $\Skl$ it is  
convenient to work with the following   general class of algebras.  An \emph{elliptic algebra}  \label{ellipticalg-defn}  is a   
 cg domain   
  $T$ containing a distinguished 
central element $g \in T_1$ so that $T/gT$ is isomorphic to a \emph{twisted homogeneous coordinate ring} 
$B=B(E,\mathcal{N},\tau)$,
 where $E$ is  an elliptic curve equipped  with an ample invertible sheaf $\mathcal{N}$ (which we assume in the introduction to have  degree $\geq 3$) 
 and an  infinite order automorphism  $\tau$.  
  See   Section~\ref{GENERALITIES} for more details.  
 For example,  the   Veronese rings  $T = \Skl^{(3)}$  and $T=\QVB^{(2)}$ are  elliptic; the  Veronese ring  is needed to 
 ensure that the central element  has degree one, but this is a fairly harmless change
  since   $\rqgr S\cong\rqgr T$
 for $S= \Skl$, respectively $ \QVB$.
By Lemma~\ref{lem:dP}, a  ring that is commutative    but satisfies all other properties of an elliptic algebra is the anticanonical coordinate ring of a del Pezzo surface, so this is a natural class of algebras to consider.

The ring-theoretic notion of blowing up a ring $R$ in \cite{R-Sklyanin} produces
  a subring $\widetilde{R}\subseteq R$, although as its construction    is not used in this paper  we will not repeat the definition. 
As noted in \cite[Introduction]{R-Sklyanin},  this notion of blowing up is equivalent to Van den Bergh's 
categorical notion \cite{VdB-blowups}.  Subalgebras described in terms of blowups also appear naturally in the commutative case (see Remark~\ref{rem:dP3} for an example).  

The appropriate notion of blowing down  for an elliptic algebra $T$  has been studied in  
\cite{RSSblowdown}. In brief, define a cyclic graded right $T$-module $L$ to be 
a \emph{line module} if has the Hilbert series $(1-t)^{-2}$ of $\kk[x,y]$. Then   one can    
contract (blow down) any line module $L$ satisfying the 
noncommutative notion of self-intersection $(-1)$; that is, 
  $(L\cdot L) = \sum_{n=0}^\infty (-1)^{n+1} \dim \Ext_{\rqgr T}^n(L,L) = -1$. 
  This gives an elliptic algebra $T' \supsetneqq T$  defined by the property that $T'/T$ is a suitable 
  direct sum of   shifts of copies of $L$. Crucially,   $T'$ has a point module can be blown up to recover $T$.  
Conversely, blowing up a point on $T$ constructs a subalgebra $\wt{T}$
with    a   line module $L$ that can be contracted to recover $T$.    
See \cite[Introduction]{RSSblowdown}  for the  results mentioned here.

So, in the noncommutative context,  
 what is the appropriate definition of a minimal model?
Since  contracting a curve corresponds to taking a certain  overring within the graded quotient ring,  
any noncommutative analogue of a 
minimal surface should be an algebra with few overrings.  However, beyond that, it is unclear how best  to translate the commutative definition. Fortunately, it does not matter since the  elliptic algebras of interest  satisfy a very 
  strong maximality property, which we can take as our definition. 

 \begin{defn}\label{min-model defn}  Given an elliptic algebra $T$, with its central element $g\in T_1$, write 
 $T_{(g)}$ for the localisation obtained by inverting the homogeneous elements from $T\smallsetminus gT$. 
 Then a {\em minimal (noncommutative) elliptic surface} is an elliptic algebra $T$ which has the property that, if $R$ is a cg noetherian algebra  with $T \subseteq R \subset T_{(g)}$, then $T=R$.     \end{defn}
 
   As will be seen in Theorem~\ref{ithm:arbitraryoverring},  the restriction to subrings of $T_{(g)}$ is also unnecessary. 
   
Definition~\ref{min-model defn} is so strong that one may wonder if there are any minimal elliptic surfaces at all. 
After all,   every cg commutative noetherian domain  $A$ of Krull dimension $\geq 2$  has  a proper cg noetherian overring---simply adjoin an element of positive degree from its graded quotient ring.
 In contrast, in  the noncommutative setting  minimal elliptic surfaces  do exist,  as our main theorem shows.

 \begin{theorem} {\rm (See Theorem~\ref{thm:3.6}.)}
\label{ithm:3.6} Let $T=\QVB^{(2)}$ be the  second Veronese of a  Van den Bergh quadric   or 
$T=\Skl^{(3)}$, the third Veronese of a quadratic Sklyanin algebra $\Skl$. Then $T$ is a minimal  elliptic surface.
\end{theorem}

Theorem~\ref{ithm:3.6} is a consequence of   the following noncommutative version of Theorem~\ref{ithm:A}, which   further  justifies our definition of a minimal model.

\begin{theorem} {\rm (See Theorems~\ref{thm:3.5} and \ref{thm111}.)} \label{ithm1}
Let $T$ be an elliptic algebra with   no line modules of self-intersection $(-1)$.   Assume that   
$(T[g^{-1}])_0$ is hereditary.  
Then $T$ is a minimal elliptic surface. \end{theorem}

The conditions of this theorem always hold for  $T = \Skl^{(3)}$ and hold generically for  $T=\QVB^{(2)}$.

We note that Theorem~\ref{ithm:3.6} easily lifts to give the analogous result for the original rings $\Skl$ and $\QVB$.

\begin{corollary} {\rm  (See Corollary~\ref{cor:3.56}.) } \label{icor:3.56} 
 Set   $S=\QVB$  or $S=\Skl$, with corresponding central element $g$.   
    If $S\subseteq U\subset S_{(g)}$ for some cg noetherian ring  $U$ then $S=U$. 
 \end{corollary} 
  
The focus on the localised ring $T_{(g)}$ in Definition~\ref{min-model defn} may seem mysterious.
However, as we next discuss, minimal models still have few overrings without this restriction.
Let $T$ be one of the algebras from Theorem~\ref{ithm:3.6}, or indeed any minimal elliptic surface, with  its central element $g\in T_1$.
  Then there always exist 
 cg noetherian overrings of $T$ with the same graded quotient ring.
  Indeed, for any integer $n\geq 2$ one has
 $T\subsetneqq \kk\ang{T_n g^{1-n}}$. This is of course a rather ``cheap'' counterexample since
 after a change of grading,  $
\kk\ang{T_n g^{1-n}}\cong T^{(n)}$ under the homomorphism $xg^{1-n}\mapsto x$ for $x\in T_n$. As we will see,
 rings like this are essentially the only other cg noetherian overrings of $T$.

In the next result, a $\ZZ$-graded Goldie domain $U$ is a \emph{maximal order}\label{maxorder-defn} if it satisfies the 
 following condition: If $V$   is  a $\ZZ$-graded  
  $\kk$-algebra with $U\subseteq V\subset Q_{gr}(U)$ that satisfies $aVb\subseteq U$ for some 
  $a,b\in U\smallsetminus\{0\}$, then $V=U$. This condition is equivalent to its ungraded analogue and
   is the appropriate noncommutative analogue of an integrally closed domain. We note that elliptic algebras
    are maximal orders (see Remark~\ref{mo-remark}).   
   
\begin{theorem}\label{ithm:arbitraryoverring}
 {\rm (See Theorem~\ref{thm:arbitraryoverring}.)}
  Let $T$ be any minimal elliptic surface, with  central element $g\in T_1$,  and let  $R$
  be a cg noetherian $\kk$-algebra with $T \subseteq R \subset Q_{gr}(T)$.   If $R$
   is a maximal order, 
then  $R^{(m)} \cong   T^{(n)}$ for some integers $m, n\geq 1$.   
   The analogous theorem also holds for the rings $\Skl$ and $\QVB$. 
   \end{theorem}

In the commutative case, Theorem~\ref{ithm:B} follows from the N\'eron-Severi theorem and Castelnuovo's contraction criterion.
 In Section~\ref{OVERRING}  we use the techniques from the proof of Theorem~\ref{ithm:3.6}, together with the noncommutative version of Castelnuovo's contraction criterion from \cite{RSSblowdown}, to give a noncommutative analogue of this result. 
 Given a graded $T$-module $M$ over an elliptic algebra $T$, we define $T^\circ =T[g^{-1}]_0$ and $M^\circ = M\otimes_TT[g^{-1}]_0$.
 \label{Tcirc-defn}     
 
  \begin{theorem}\label{ithm:manchester}  {\rm (See Corollary~\ref{cor:manchester}.)}
   Let $T$ be an elliptic algebra for which 
   $T^\circ$  is  hereditary and let $T\subseteq R\subset T_{(g)}$ be any noetherian cg overring. Then there is a  module-finite ring 
   extension 
   $R\subseteq R'$ such that   $R'$ is obtained from $T$ by   contracting finitely many line modules $L$ of self-intersection $-1$.  
   \end{theorem}

We remark that   finitely generated but non-noetherian cg overrings of  minimal elliptic surfaces $T$ obviously exist (just take the ring generated by $T$ and  a 
homogeneous element of positive degree 
from $T_{(g)}$). However, the resulting ring 
has many unpleasant properties (see, for example,  Proposition~\ref{prop:2.3} and Corollary~\ref{cor:weird}). 
Moreover (except for  trivial overrings like $R = \kk\ang{T_ng^{1-n}}$), the Gelfand-Kirillov dimension $\GKdim R$ must jump;  for example,  if $R$ is a cg
noetherian  ring with $T\subsetneqq R\subset T_{(g)}$, then 
$\GKdim R\geq 4>3=\GKdim T$ (see Section~\ref{GKDIM} and in particular Corollary~\ref{cor:GK5}).

The idea of the proof of  Theorem~\ref{ithm:3.6}   is as follows.  The starting point  is
  that, by Lemma~\ref{lem:nolines}, $T$ has no line modules   to contract. But suppose that $T$ does have 
  a noetherian cg overring $T\subsetneqq R\subset T_{(g)}$. 
Form the localisation 
$R^\circ:=R[g^{-1}]_0\supsetneqq T^\circ:=T[g^{-1}]_0$.    
In the cases of interest, one can always reduce to the case when 
$T^\circ$ is simple and is either hereditary or has a mild singularity. 
For simplicity assume that $T^\circ$ is hereditary; the argument is easiest to explain here but does still give the general idea.    
Also, after possibly replacing $R$ by a harmless, finitely generated extension one can assume that 
$R= \Phi ( R^\circ) =\bigoplus_{n\in \mathbb{Z}} \{a\in (T_{(g)})_n : ag^{-n}\in R^\circ\}$
  (see Lemma~\ref{lem:1.7prime}).   
  
 Now pick  a simple submodule   $\Nn=\mathscr{M}/T^\circ \subseteq R^\circ /T^\circ$.  One can show that 
 $\mathscr{M}=M^\circ$ for some module $M\supseteq T$ such that  $N=M/T$ is a  2-critical $T$-module with $N^\circ \cong \Nn$. 
 Since  $N$ cannot be a line module, one proves   that there are many copies of submodules
  of $N$ inside $Q_{gr}(T)/T$, one of  which, say $N'=M'/T$, is not contained in $R/T$.  
 However, one still has $(N')^\circ\cong \Nn  \cong N^\circ$. Now, as $T^\circ$ is hereditary,  a classic result of Goodearl \cite{Goodearl}
 shows that $R^\circ$ will be a torsion-theoretic localisation of $T^\circ$. This means that  $R^\circ/T^\circ $ will 
 contain all possible copies of $N^\circ$ and in particular $(N')^\circ$.  In other words, $(M')^\circ   \subseteq R^\circ $ 
 and hence $M'   \subseteq \Phi(R^\circ ) = R$, giving the required contradiction. See Section~\ref{MINIMAL} for the details.  This argument  is  modified to work for algebras with $A_1$ singularities in Section~\ref{A1SINGULARITIES} and, as is shown in Section~\ref{ALGEBRAS}, this in turn is enough to prove  the theorem.

 Finally, we note that noncommutative versions of   Theorem~\ref{ithm:C} have yet to be established and are the subject of ongoing work.
 We  believe that the full list of minimal models for the quantum rational case will consist of $\rqgr R$ for 
$R=\Skl$ and $R=\QVB$, as above, together with  their Veronese rings and certain endomorphism algebras; 
  see Remark~\ref{rem:further rings}.     We conjecture  that Theorems~\ref{ithm:3.6} and \ref{ithm:manchester} can be extended to prove the ultimate goal that, given an elliptic algebra $R$, then one can obtain a minimal model $\widehat{R}$ from $R$ by contracting finitely many line modules.   This in turn will prove much of Artin's programme for algebras birational to the Sklyanin algebra $\Skl$. See Conjecture~\ref{final-conj} and Remark~\ref{rem:further rings} for more  details.
    
   %%%%%%%%%%%%%%%%%
\section{Generalities}\label{GENERALITIES}

  In this short section, we set up some of the basic notation and results used later in the paper.
  Fix an algebraically closed base field $\kk$ of arbitrary characteristic.  
   Let $T$ be an elliptic algebra with graded quotient ring $Q_{gr}(T) \cong D[g, g^{-1}]$, for the appropriate division ring $D$.  Thus 
$T/gT \cong B :=B(E, \mc{M}, \tau) $  
 is a \emph{twisted homogeneous coordinate ring}, or TCR, over the elliptic curve $E$\label{TCR-defn}
and $\tau$ is an automorphism of $E$ of infinite order (and so  $\tau$ is given by translation by a point of $E$ under the group law).  
Here $B=\bigoplus B_n$, where $B_n=H^0(E, \mc{M}_n)$ and $\mc{M}_n = \mc{M}\otimes\cdots\otimes \mc{M}^{\tau^{n-1}}$, with the natural multiplication.
We say that $E$ is the {\em elliptic curve  associated} to $T$ (or $B$)  and define the  \emph{degree} of $T$  \label{degree-defn}
 to be the degree of the line bundle $\mathcal{M}$.  Unless   stated otherwise, we assume in the body of the paper 
  that $\deg \mc{M} \geq 2$ so that $B$ and $T$ 
are generated in degree~$1$ (see, for example, \cite[Lemma~3.1(2)]{R-Sklyanin}).

 \begin{definition}\label{rem:2.331} First, it is convenient to weaken the concept of a cg algebra.   Define a  $\ZZ$-graded $\kk$-algebra $R$ 
 to be   \emph{finitely graded}   if     $\dim_{\kk}R_n<\infty$ for all $n$ and $R_n\not=0$ for some $n>0$ (the final condition is included since
  it is convenient to exclude rings graded by  $-\NN$). Obviously, apart from $R=k$ itself, 
   cg algebras, as defined in the introduction, are finitely graded. 
\end{definition}

\begin{remark}\label{rem:2.33}  The following observations  will be used several times, usually without further comment. 

(1) Suppose  that  $R$ is a $\ZZ$-graded domain  
with $\dim_{\kk}R_n\geq 2$ for some $n\geq 0$ 
(as is always the case in this paper). If $R$ is not $\NN$-graded, say with $R_{-a}\not=0$, then $(R_{-a})^n(R_n)^a$ contains an 
element $\alpha\in R_0\smallsetminus \kk$ and so $R_0$ contains  the polynomial ring $\kk [\alpha]$. Thus $R$ is not finitely graded. 
Equivalently, if $R$ is finitely graded then  $R$ is necessarily $\NN$-graded with $R_0=\kk$.

(2) In a similar vein,   if $R$ is   a   noetherian $\mathbb{N}$-graded
  $\kk$-algebra with $R_0=\kk$,   then  generators of the $R$-module  $R_{\geq 1}$ also  generate  $R$ as an algebra. Thus $R$ is  cg. \end{remark}

We next review some important homological conditions.  Throughout $\Hom_R(M,N)$ and $\Ext^i_R(M,N)$ will denote 
 the given  groups in the   category of $R$-modules. When $M$ and $N$ are finitely generated $\mathbb{Z}$-graded modules over 
 an $\mathbb{N}$-graded ring $R$,   these carry a natural  $\mathbb{Z}$-gradation.

\begin{definition}\label{def:gor}
Let $A$ be a ring with  $\id(A)<\infty$, in the sense that $A$ has finite injective dimension on both left and right. 
For  a  finitely generated right $A$-module 
 $M$, write   $j(M)= \min\{ r : \Ext^r_{A}(M,A)\not= 0\}$ for the \emph{homological grade} of $M$. \label{grade-defn}
 Then  $M$   is then called 
 \emph{Cohen-Macaulay} (or CM), if 
$\Ext_{A}^j(M,A) = 0$ for all $j\not=j(M)$.  The module $M$ is {\em maximal Cohen-Macaulay} (or MCM) if  it is  CM with $j(M)=0$.  

A ring $A$ with $\id(A)<\infty$   is called  \emph{Auslander-Gorenstein} if 
 the following holds: if $0 \leq p<q$ and  $M$ is a finitely generated $A$-module, then
$\Ext_{A}^p(N,\,A)=0$ for every submodule $N$ of 
$\Ext_{A}^q(M,\,A)$. 
Write $\GKdim(M)$ for the Gelfand-Kirillov dimension of an $A$-module $M$, as in \cite{KL}. \label{GK-defn}
Let $A$ be an noetherian Auslander-Gorenstein $\kk$-algebra with $\GKdim(A)<\infty$. 
The algebra $A$ is called \emph{Cohen-Macaulay} (or CM), provided   \label{CM-defn}
 that  $j(M)+\GKdim(M)=\GKdim(A)$ holds for every finitely generated
 $A$-module $M$.    
Note that  an elliptic algebra is always Auslander-Gorenstein and CM by \cite[Proposition 2.4]{RSSshort}.
  \end{definition}

Let  $M=\bigoplus_{n\in \mathbb{Z}} M_n$  be a  non-zero $\ZZ$-graded module over a   cg ring $R$ of finite Gelfand-Kirillov dimension such that 
$\dim_{\kk}M_n<\infty$ for all $n$. 
Then $M$  is called \emph{$d$-critical} if $\GKdim M = d$ but $\GKdim(M/N)<d$  \label{critical-defn}
  for all non-zero submodules $N\subseteq M$.  In order to avoid repetition, \emph{we always assume that critical $T$-modules are both finitely generated and graded.}   
  Similarly, $M$ is 
 \emph{$d$-pure} if $\GKdim M = \GKdim N=d$ for all non-zero submodules $N\subseteq M$. Note that, by 
  \cite[Theorem~6.14]{KL}, a  $d$-critical $\ZZ$-graded module is automatically $d$-pure.  
  
The   \emph{Hilbert series} of   $M$  is defined to be  $\hilb M=\bigoplus_{n\in \mathbb{Z}}(\dim_{\kk}M_n)t^n.$\label{Hilbert-defn}
 Finally,   $M$ is a  {\em linear module} if $M=M_0R$,   with   $\hilb(M) =   (1-t)^{-p}$ for some $p>0$.
 When $p=1$, respectively $2$,  the module $M$ is   called {\em a point module},  respectively {\em a line module}.\label{linear-defn}
If $R$ is also an elliptic algebra, then $M$ is called  \emph{torsion} or  \emph{Goldie torsion} if, for all $m\in M$ there exists 
$t\in R\smallsetminus\{0\}$ with $mt=0$. We say that $M$ is \emph{$g$-torsion} if for all $m\in M$  one has $m g^n=0$ for some 
$n\geq 0$.  The terms \emph{Goldie torsionfree} and \emph{$g$-torsionfree} then have their expected meaning.
  If $M$ is a graded $R$-module and $n\in \mathbb{Z}$, define $M[n] = \bigoplus_{ j \in \ZZ} M[n]_j$, where $M[n]_j=M_{j+n}$. 

\begin{lemma}\label{lem:Bdual}
 Let $B = B(E, \sL, \tau)$ be a TCR  over  an elliptic curve $E$ with $\deg \sL \geq 2$ and $| \tau | = \infty$.
 \begin{enumerate}
\item $B$ is Auslander-Gorenstein, CM and generated  by $B_1$ as an algebra.
\item We have $\Ext^i_B(\kk, B) = \delta_{i,2}\kk$.
\item Let $M$ be a right (left) point module over $B$.  Then $M$ is CM  and $\Ext^1_B(M, B)[1]$ is a left (right) point module.
\end{enumerate}
\end{lemma}

\begin{proof}  For Part (1), see 
  \cite[Lemma~2.2]{RSSblowdown}, for Part (2)  see  \cite[Theorem~6.3]{Lev1992}
and for Part (3) see \cite[Lemma~3.3]{RSSblowdown} (this final result assumes that $\deg\sL\geq 3$ but the proof also works 
when $\deg\sL=2$).\end{proof}

\begin{lemma} \label{lem:exths}  
Let $T$ be an elliptic algebra. 
Let $M$ be a $2$-pure finitely generated graded $g$-torsionfree $T$-module such that $M/Mg$ is $1$-pure.  
Then  $M/Mg$ has a filtration with shifted  point module sub-factors $\{ P(p_i)[m_i] : 1 \leq i \leq d \}$.  
  Moreover, 
$\Ext^1_T(M, T)$ has Hilbert series $\sum_{i=1}^d s^{m_i + 1}/(1-s)^2$.
\end{lemma}

\begin{remark}\label{div-defn} 
In the lemma above, we write $P(p)$ for the $T/Tg$-point module parameterised by $p \in E$; we shall not need the details of this parameterisation.  
The number $d$ is written $d=d(M).$
\end{remark}

\begin{proof}   Since $M/Mg$ is  $1$-pure, it has a   composition series 
$0 = N_0 \subseteq N_1 \subseteq \dots \subseteq N_d = M$ with  $1$-critical factors  $N_i/N_{i-1}$. 
By  \cite[Lemma 2.8]{RSSshort}, the $1$-critical modules over $B = T/gT$ are just shifted point modules, which proves the first statement.
 The second statement then follows from  \cite[Lemma~5.4(2)]{RSSblowdown} (once again, that result assumes that $\deg\sL\geq 3$, 
  but the proof also works when $\deg\sL=2$).
 \end{proof}

\begin{notation}\label{E11-notation} Given a right  module $M$ over an elliptic algebra $T$, set $E^{11}(M) :=\Ext^1_T(\Ext^1_T(M, T), T)$. 
Recall that  $T^{\circ} := T[g^{-1}]_0 = \bigcup_{n \geq 0} T_n g^{-n}$ and that 
$M^{\circ} := M[g^{-1}]_0\cong  (M \otimes_T T[g^{-1}])_0$.
\end{notation}

\begin{lemma}\label{lem:M-modify}
  Let $M$ be   a $2$-critical, $g$-torsionfree  right  module over   an elliptic algebra $T$.  Then 
  \begin{enumerate}
  \item  $E^{11}(M)$ is the maximal essential extension of $M$ by finite-dimensional modules. Hence
   $E^{11}(M)$ is still $g$-torsionfree and $2$-critical, with $(E^{11}(M))^{\circ} = M^{\circ}$ and    $\Ext^1_T(\kk, E^{11}(M)) = 0$;
\item    $E^{11}(M)$ is CM.  Moreover  $E^{11}(M)/ E^{11}(M)g$ is $1$-pure.
\end{enumerate}
\end{lemma}
  
  \begin{proof} (1)  The first statement follows, for example, from  \cite[(4.6.6) and Remark~5.8(4)]{Lev1992}.
   As such, $E^{11}(M)/M$ is annihilated by a power of $g$ 
   and so   $E^{11}(M)^\circ=M^\circ$.  The other two assertions then follow easily.  
  
  (2)   If $E^{11}(M)/E^{11}(M)g$ has a non-zero,  finite dimensional submodule, say $W/E^{11}(M)g$,
then $E^{11}(M) \cong E^{11}(M)g[1] $ has a nontrivial extension by that finite-dimensional module, contradicting Part (1).
 Thus  $E^{11}(M)/ E^{11}(M)g$ is $1$-pure. The  CM condition then  follows from  Lemma~\ref{lem:exths} and
   \cite[Lemma~5.4]{RSSblowdown}. 
  \end{proof}

%%%%%%%%%%%%%%%  
   \section{Some key  lemmas}\label{keylemmas}
   
   In this section we provide three technical lemmas   that, nevertheless, lie at the heart of the proofs of the main theorems.
   
     The first  lemma is at the heart of the proof of Theorem~\ref{ithm:arbitraryoverring}; it shows that the structure of certain
      modules over an elliptic algebra $T$ can be perturbed without affecting their image in the localised category
       $\rmod T^\circ$.  This is used in Proposition~\ref{prop:linesonly}    to get useful pertubations of $T$ modules. 
       The final result of the section  investigates the relationship between a cg algebra and its $g$-divisible hull, as defined below, which will be important in understanding general rings.

\begin{lemma}\label{lem:1.4}
Let $T$ be an elliptic algebra with a 2-critical,   $g$-torsionfree   $T$-module $M$  such that $M= E^{11}(M)$  
and   $\min \{ i : M_i \neq 0 \} = 0$.  Suppose  that $M/Mg$ has a filtration  
$$0=N(0)/Mg \subsetneqq N(1)/Mg \subsetneqq \cdots \subsetneqq N(d)/Mg =M/Mg$$ with factors being   (unshifted) point modules 
 $\{ P(p_i) :=N(i)/N(i-1) : 1 \leq i \leq d \}$.
  Then the following hold.
 \begin{enumerate}
 \item  $N:=N(1)$ is again $2$-critical CM 
$g$-torsionfree,  with $N^{\circ} = M^\circ$ and   $\min \{ i : N_i \neq 0 \} = 0$. 
\item However,
 now $N/Ng$ has a filtration by   shifted point modules 
$\{ P(q_i)[n_i] : 1 \leq i \leq d \}$ where $n_d = 0$ but $n_i = -1$ for $1 \leq i \leq d-1$.
\end{enumerate}
\end{lemma}

\begin{proof} (1) 
 We may assume that $d>1$, as the result is trivial otherwise. The module $N$ is trivially $2$-critical and $g$-torsionfree and hence also $2$-pure.   
  By Lemma~\ref{lem:M-modify},
 $E^{11}(N) $ is the largest essential extension of $N$  by finite-dimensional modules, while   $M$ has 
 no nontrivial such extensions. Thus,  using the inclusion  $E^{11}(N)  \subseteq I(N)\subseteq I(M)$,  inside the injective hull $I(M)$, we 
 conclude  that $E^{11}(N)  \subseteq M$. But, by construction,  $M/N$ is filtered by shifted point modules 
and hence is $1$-pure. Thus $E^{11}(N)  = N$. Therefore, by Lemma~\ref{lem:M-modify},   $N$   is    CM.
Since $M \supseteq N \supseteq Mg$, certainly  $M^{\circ} = N^{\circ}$. Since $(N/Mg)_0=P(p_1)_0\not=0$, certainly $N_0\not=0$  and hence $\min \{ i : N_i \neq 0 \} = 0$.

(2)   It remains to understand the shifted point module  filtration of $N/Ng$.   
Recall that $\Tor_1^T(L, T/Tg)[1]\cong \{ x \in L : xg = 0 \}$   for any graded $T$-module $L$ (see, for example,
    \cite[Equation (8.1)]{R-Sklyanin}).  In particular,  $\Tor_1^T(M, T/Tg) = 0$. Thus   the short exact sequence $0\to N\to M\to M/N \to 0$ induces an     exact sequence
\[ \xymatrix{
0 \ar[r] & \Tor_1^T(M/N, T/Tg)  \ar[r] & N/Ng \ar[r] &  M/Mg \ar[r] & M/(N + Mg) \ar[r] & 0. }
\]
By  \cite[Equation (8.1)]{R-Sklyanin}, again, $X :=\Tor_1^T(M/N, T/Tg) \cong (M/N)[-1]$.
Thus, by the definition of $N$, $X$ is filtered  by the point modules $\{ P(p_i)[-1] : 2 \leq i \leq d \}$,
while $(N/Ng)/X \cong (N+Mg)/Mg=N/Mg$ is the point module $P(p_1)$. This defines  the desired filtration of $N/Ng$.
\end{proof}

We now provide a useful application of Lemma~\ref{lem:1.4}.

\begin{proposition} \label{prop:linesonly} 
 Let  $T $ be  an  elliptic algebra, with a finitely graded overring  $T\subsetneqq  R\subset T_{(g)}$.  
Assume that  $T^\circ $ is  simple and fix 
 a simple $T^\circ$-submodule $\Ss$ of $R^\circ/T^\circ$.
Then   there is  
 a $2$-critical $g$-torsionfree $T$-module $M$ with  $\min \{ i : M_i \neq 0 \} = 0$ such that $M^{\circ} = \Ss$. 
 Moreover, in the
  notation of Lemma~\ref{lem:exths} either
   \begin{enumerate}
   \item  $d(M)=1$ and $M$ is a   line module, or 
   \item   $d(M)>1$ in which  case there exists a   module extension $T\subsetneqq L \subset T_{(g)}$ such that 
   $L^\circ/T^\circ \cong \Ss$ but $L_r\supsetneqq T_r$ for some $r\leq 0$.   
   \end{enumerate}
\end{proposition}

\begin{remark} The significance of this result is that, for the algebras $T$  of interest, we will show that 
  $L$  can also be embedded into the given overring $R$. This   contradicts
   Remark~\ref{rem:2.33}(1). Thus $T$ must have line modules which,  in turn, proves much  of 
Theorem~\ref{ithm1}.\end{remark}

\begin{comment} 
 \begin{remark} The significance of this result is that, for the algebras $T$  of interest, one can prove that  the
module $L$ will also be contained in the given overring $R$. By Remark~\ref{rem:2.33}, this 
contradicts the fact that $R$ is finitely graded 
and proves that $T$ has no  proper overrings that are also elliptic.\end{remark}
\end{comment}

\begin{proof}   
Write  $\Ss \cong T^{\circ}/\Ii$ for some right ideal $\Ii$ and set  $I = \bigoplus_{n\in \mathbb{N}}\{ x \in T_n : x g^{-n} \in \Ii\ \}$.  Thus  $I$ is a graded right ideal of $T$ such that $I^{\circ} = \Ii$ 
and so $(T/I)^{\circ} \cong \Ss$.    By construction, $M=T/I$ is   $g$-torsionfree.  If $\GKdim(M)=1$ then, by  the proof of \cite[Proposition~7.5]{ATV2},  $\Ss = M^{\circ}$ would be   finite-dimensional, contradicting the simplicity of $T^\circ$. Thus $\GKdim(M)=2$.  Indeed, we claim
 that $M$ is 2-critical. To see this, suppose that  $M$ has a proper factor $T/J$ with $\GKdim(T/J)=2$; we may assume that $T/J$ is 2-critical.
 Then, by the definition of $I$, this forces  $J^\circ=T^\circ$ and hence $J\supseteq g^nT$ for some $n$. By \cite[Proposition~2.36(vi)]{ATV2},
  $T/J$ has a prime annihilator and hence $J\supseteq gT$.  Since $B=T/gT$ is  $2$-critical, this implies that $gT= J\supset I$. This is
   impossible by the definition of $I$ and implies that   $M$ is 2-critical.

 By  Lemma~\ref{lem:M-modify},  it is harmless to replace $M$ by $E^{11}(M)$ and so $M$ now has the properties described by that lemma.  Note that $Mg\cong M[-1]$ and so $M^\circ\cong M[n]^\circ$ for all $n\in \mathbb{Z}$.  Thus we may also replace $M$ by some shift $M[n]$ and  
  assume that $\min \{ i : M_i \neq 0 \} = 0$.
  
    By Lemma~\ref{lem:exths} $M/Mg$ has a filtration with $d=d(M)$ shifted point module subfactors. If $d=1$,  then $M/Mg$ is a   shifted point module, and so  $M$ has the Hilbert series $s^m/(1-s)^2$ of a shifted line module.  Since $\min \{ i : M_i \neq 0 \} = 0$,  actually $M$  has  Hilbert series $1/(1-s)^2$. Since $M/Mg$ is cyclic, so is $M$ and hence $M$  is a line module.

 So suppose that   $d\geq 2$.  
Let the point modules in the filtration of $M/Mg$ be $\{ P(p_i)[m_i] : 1 \leq i \leq d \}$.   Since $\min \{ i : (M/Mg)_i \neq 0 \} = 0$, necessarily 
$m_i = 0$ for some $i$.  It may be that   $m_j < 0$ for some $j$, which is fine.  If this is not the case, 
then $m_i = 0$ for all $1 \leq i \leq d$.  In this case,  we 
can, by Lemma~\ref{lem:1.4},  replace $M$ by a second module with all of the   properties of $M$ except that now   $m_d = 0$ and 
$m_i < 0$ for $1 \leq i \leq d-1$.

By Lemma~\ref{lem:exths},  $\Ext^1_T(M, T)$ has Hilbert series  $\sum_{i=1}^d s^{m_i+1}/(1-s)^2$.  
By construction, some $m_i < 0$ and so writing this Hilbert series as $\sum_{n \in \mb{Z}} c_n s^n$, we have $c_n > 0$
for some $n \leq 0$.  Fixing such an $n$, there therefore exists  $0 \neq \theta \in \Ext^1_T(M, T)_n = 
\Ext^1_T(M, T[n])_0 = \Ext^1_T(M[-n], T)_0$.  Then $\theta$ corresponds to a (necessarily nonsplit) graded exact sequence 
$0 \to T \to X \to M[-n]\to 0$ for some module $X$.    Now   $T$ is
Auslander-Gorenstein and CM
by \cite[Theorem~6.3]{R-Sklyanin} and so $\Ext^1_T(N, T) = 0$ for any module $N$ with $\GKdim(N) \leq 1$.  
Since $M[-n]$ is 2-critical, it  follows easily  that the 
extension $T \hookrightarrow X$ is  essential. Thus we may embed  $T\subsetneqq X\subsetneqq Q_{gr}(T)$, the graded quotient 
ring of $T$.   Since $M[-n]$ is $g$-torsionfree, but  (Goldie) torsion,  it follows that, for any $x\in X$ there exists $t\in T\smallsetminus gT$ 
such that $xt\in T$. In other words,   $ X\subseteq T_{(g)}$.  

Finally, $X/T \cong M[-n]$, and since $M^{\circ} \cong \Ss$, one also has  $M[-n]^\circ\cong \Ss$.  
  However, $(X/T)_n =M[-n]_n=M_0\not=0$ with $n \leq 0$.  Thus $L=X$ satisfies the conclusions of the proposition.
  \end{proof}

 The main results of this paper will  also cover  non-elliptic algebras and we end the section with some technical results needed for this more general case.

  \begin{lemma}\label{lem:1.6}
Suppose that $A$ is a   cg $\kk$-algebra  
 that is a domain with $\GKdim A=2$  and graded quotient ring $ Q_{gr}(A)$. Assume that  
$Q_{gr}(A) \hookrightarrow G := k(E)[t, t^{-1}; \tau]$ with $|\tau|=\infty. $  

  Let  $\{Q(i) : i\geq 0\} $ be an ascending chain  of graded $A$-sub-bimodules of $G$ that are finitely generated as both left and right $A$-modules. Then the chain is eventually stationary.
\end{lemma}

\begin{proof}    It does no harm to replace $A$ by some Veronese ring $A^{(n)}$, and thereby assume that 
$A=A^{(n)}$, with $\dim_\kk A_n\geq 2$. We  emphasise that, here, we do not change the grading on $A$, since we cannot replace   $Q(j)$ by its Veronese.  We now claim:

\begin{claim}\label{1.6-claim}
$A\subseteq Z :=B(E',\mc{L},\tau'),$ for some  invertible sheaf $\mc{L}$ over an elliptic curve $E'$ with $\deg \mc{L} \geq 2$ and $|\tau'|=\infty$. 
 This embedding may be chosen so that  $Z$ is a noetherian left $A$-module. 
  \end{claim}

\noindent
 {\it Proof of the claim:} 
  As was true for  $A$, 
  our convention here is that  
 $Z=\bigoplus_{j\geq 0} Z_{nj}$ with $$Z_{nj}=H^0(E'\, , \,  \mc{L} \otimes\mc{L}^{\tau'}\otimes \cdots\otimes \mc{L}^{(\tau')^{j-1}}).$$
 The  hypotheses of   \cite[Theorem~5.9]{AS} are satisfied for $A$, and there is an embedding $A \subseteq Z :=B(E',\mc{L},\tau')\subset Q_{gr}(A)$ for some ample invertible sheaf $\mc{L}$ over some smooth curve $E'$ for which $Z$ is a finitely generated left $A$-module.  Our choice of grading implies  that 
  $Z=Z^{(n)} $ with $\dim_{\kk}Z_n\geq 2$.  Thus $\deg \mc L \geq 2$.

By \cite[Theorem~34]{Sco}, the full division ring of fractions $F :=\Fr(G)$ is  finitely generated as both a left and right module over $F'  :=\Fr(A)$. As in the proof of \cite[Proposition~2.6]{RSS},  it follows that $E'$ must be elliptic with $|\tau'|=\infty$.  
 Thus the claim is proven.\qed
 
 \medskip
We return to the proof of the lemma.  Since the Goldie rank of the $Q(i)$ as    left $A$-modules is bounded above by
 the finite Goldie rank of $F$ as a left $F'$-module, we can remove a finite number of terms and assume that the Goldie rank of  the ${}_AQ(i)$ is
  constant. In particular,  each $X(i) :=Q(i)Z/Q(0)Z$ is   torsion as a left $A$-module. Note that, as ${}_AZ$ is finitely generated, so is 
  each $Q(i)Z$ as a left $A$-module. Similarly, since the $Q(i)$ are finitely generated as right $A$-modules, each $Q(i)Z$ is a finitely 
  generated right $Z$-module. Write  $X(i)=\sum_{j=1}^m x_jZ$ for some $x_j$; thus  $K(i) :=\lann_A(X(i)) =\bigcap_j\lann_A(x_j)\not=0$. Moreover,
 \cite[Proposition 6.5(2)]{AS} implies that   $\dim_{\kk} A/K(i)< \infty$. In particular, as the $X(i)$ are finitely generated left $A/K(i)$-modules
  it follows that each $X(i)$ is finite dimensional. 
 
 Now consider the $Q(i)Z$ as right $Z$-modules.   Since $\deg \mc{L} \geq 2$,  
 Lemma~\ref{lem:Bdual} implies that $Z$
  is Auslander-Gorenstein and CM. By construction, each   $Q(i)Z$ is also   torsionfree as a right $Z$-module. Thus, by \cite[(4.6.6) and Remark~5.8(4)]{Lev1992}, there is a unique largest essential extension  $Y$ of $Q(0)Z$ by  finite dimensional right $Z$-modules. By its construction in \cite{Lev1992}, $Y$ is  finitely generated and hence noetherian as a right $Z$-module.
  As such, $Y/Q(0)Z$ is finite dimensional, say with  right annihilator $L$. Since $Q(0)Z$ is a left $A$-module it follows that,  for any $a\in A$,
   the right $Z$-module  $\bigl(aY+Q(0)Z\bigr)/Q(0)Z$ is also killed by $L$ and hence is also finite dimensional.  Hence $aY\subseteq Y$ 
   and so $Y=AY$ is actually a left $A$-module.    
 In particular, as  $Y/Q(0)Z$ is finite dimensional,   $Y$ is finitely generated as both a right $Z$-module and a left $A$-module.

 Finally, as the $X(i)$ are finite dimensional,     each $Q(i)Z$ and hence each $Q(i)$ lies in $Y$. Since $A$ is noetherian by
 \cite[Theorem~0.4]{AS}, it follows that  the union $ \bigcup_n Q(n)$    is also a noetherian  left $A$-module. Thus, 
the chain  $\{Q(n)\}$ must be eventually stationary.
\end{proof}

\begin{definition}\label{g-divisible-defn}
For any graded vector subspace $X \subseteq T_{(g)}$,  the \emph{$g$-divisible hull of $X$} is defined to be  
\begin{equation}\label{hull-defn}
\widehat{X}=\{t\in T_{(g)} | tg^n\in X \text{ for some } n\in\mathbb{N}\}.\end{equation}
We say that $X$ is {\em $g$-divisible}\label{g-div}
 if $X \cap gT_{(g)} = gX$.   
It is immediate that $\wh{X}$ is $g$-divisible, and if $X$ is $g$-divisible then $\wh{X} =X$. However,  even if (say) $X$ is a   cg $\kk$-algebra,  there is no reason for $\wh{X}$ to be cg or even finitely graded.  

For any elliptic algebra $T$, the fact that $gT$ is completely prime quickly implies that $T$ is $g$-divisible.
\end{definition}

We next prove an important technical result, Lemma~\ref{lem:1.7prime}, which  should be compared with  \cite[Proposition~8.7(2)]{RSSlong}. 
The latter proposition  gives a similar result in the case when $R\subseteq T$.
  We begin with a preliminary result.

\begin{lemma}\label{lem:paleale}
Let $E$ be an elliptic curve, with an infinite order  automorphism  $\sigma$.  
Let $y \in \kk(E)$ and let $V$ be  a $\kk$-subspace of $ \kk(E) $ with $2 \leq \dim V < \infty$.
Then $\dim(V + y V^\sigma) > \dim V$.
\end{lemma}

\begin{proof}
Write $V \cdot \sO_E$ for the (invertible) subsheaf of the constant sheaf $\kk(E)$ generated by $V$.  
Concretely, $V\cdot \sO_E = \sO(D_V)$, where $D_V = \min \{ D \in \Div(E) : D + (f) \geq 0 \text{ for all } f \in V\}$.
We have $(y V^\sigma) \cdot \sO_E  \cong \sO( D_V)^{\sigma}$.  
As $\deg D_V \geq \dim V \geq 2$ and $\sigma$ is an infinite-order translation, 
$\sO(D_V)^{\sigma} \not \cong \sO(D_V)$, and thus $y V^{\sigma} \neq V$.
The result follows.
\end{proof}

 Two domains $A$ and $A'$ with the same division ring of fractions $\Fr(A)=\Fr(A')$ are \emph{equivalent orders} if there exist 
 nonzero elements $a,b,c,d\in \Fr(A)$ such that  $aAb\subseteq A'$ and $cA'd\subseteq A$.

\begin{lemma}\label{lem:1.7prime}
Suppose that $T$ is an elliptic algebra and let $R$ be a   noetherian cg algebra with 
$R \subset (T_{(g)})^{(n)}$ for some $n \geq 1$. Assume that  $g^n\in R$ but that 
$R \not\subseteq \kk + gT_{(g)}$.   
 Then the following hold.
 \begin{enumerate}
 \item  Both $\wh{R}$ and 
$R':=(\wh{R})^{(n)}$ are noetherian as  left and right $R$-modules, with $g^{mn}R'\subseteq R \subseteq R'$ for some~$m$. 
\item 
In particular, $\wh{R}$ and $R'$ are noetherian   cg $\kk$-algebras while $R'$ is a equivalent order to $R$.  
\end{enumerate}
\end{lemma}
 
 \begin{proof}  Note that Part (2) is an immediate consequence of Part (1) combined with the observation from Remark~\ref{rem:2.33}(2)
  and so only Part (1) needs proof.  
 
 In this result, we regard $(T_{(g)})^{(n)}$ as a subalgebra of $T_{(g)}$ and do not change the grading. 
We start by simplifying the problem.  Since $R \not\subseteq \kk + gT_{(g)}$, there exists some $x\in R_{an}\smallsetminus gT$. 
 For the moment consider $U :=R^{(an)}$, which is also noetherian  by \cite[Lemma~4.10(2)]{AS}. 
 Suppose that $\wh{U}$ is  noetherian as both a left and a right $U$-module. Note that if $r\in R_{ns}$ then 
 $rg^{n\ell}\in U$ for the appropriate $\ell$ and so $R\subseteq \wh{U}$ and hence $\wh{R}= \wh{U}$. Thus $\wh{R}$
  is a noetherian module over both $U$ and $R$.
Therefore, we may replace $R$ by $U$ and $n$ by $an$  and assume that there exists   $x\in R_{n}\smallsetminus gT$.

 Next, set $C :=R[g]$. Since $g^n\in R$, this is certainly a noetherian  $R$-module with $\wh{C}=\wh{R}$.
  In other words, $R'=  \wh{C}^{(n)}$.  Moreover, since $R\subset (T_{(g)})^{(n)}$,  clearly $C^{(n)} = R[g^n]=R$.  
Note, also,   that    
 \begin{equation}\label{prime1}
  \wh{C}^{(n)} = \{x\in (T_{(g)})^{(n)} : g^{n\ell}x\in R \ \text{for some } \ell\geq 1\}.
 \end{equation}
 
 For $x \in T_{(g)}$, let $\overline{x} = x+g T_{(g)} \in T_{(g)}/g T_{(g)}$.
 Set $B = \overline{R} =(R+gT_{(g)})/gT_{(g)}$.
For $i \geq 0$, define 
$$Q(i) := \frac{(g^{-in}C \cap T_{(g)})+ gT_{(g)}}{gT_{(g)}}\ \subseteq \ \frac{T_{(g)}}{T_{(g)}g} .$$
 As $g^n\in R$, clearly  $\overline{C} =   Q(0)$ with    $ Q(j) \subseteq Q(j+1)\subseteq \bigcup_{i\geq 0}Q(i)=\overline{\wh{C}}$ for  $j\geq 0$.

 The next sublemma provides the strategy for the proof of the lemma.
 
 \begin{sublemma}\label{sublem:prime11} In the above notation, suppose that $Q(r)=Q(r+1)$ for $r\geq r_0$. Then
 $\wh{R}=\wh{C}$ is an equivalent order to $C$ and is noetherian as a  one-sided  $C$-module and hence as a one-sided $R$-module. 
   
Moreover,  $R' =\wh{C}^{(n)} $  is an equivalent order to $R$, with   $g^{\ell n}R'\subseteq R\subseteq R'$   
 for some $\ell \geq 1$.
 As such, $R'$ is a noetherian cg $\kk$-algebra.
  \end{sublemma}

\medskip\noindent
{\it Proof of the Sublemma.} The proof is essentially the same as   the second paragraph of the proof of \cite[Proposition~8.7(1)]{RSSlong}, although for the 
reader's convenience we include a proof here.

 To begin with, we  claim that $\wh{C} \cap g^mT_{(g)} = C \cap g^mT_{(g)}$ for all $m\geq \rho= r_0n$.  
  If not, 
there exists    $y := g^mx \in (\wh{C} \cap g^mT_{(g)} ) \smallsetminus C$ for some such $m$.  
Choose  $(y,x)$ with this property for which $x$ has minimal degree. This  ensures that $y\not\in g^{m+1}T_{(g)}$, since 
  otherwise one could write $y=g^{m+1}x'$ with $\deg(x')=\deg(x)-1$. 
 Note that, as $g^mx\in \wh{C}$, certainly   $g^{m+\ell}x\in C$ for some $\ell$, and so $x\in \wh{C}$.
  Therefore,   we can write  $\overline{x} =[x+gT_{(g)}]$ as $\overline{x}=\overline{w}$ for some $w\in g^{-\rho}C\cap T_{(g)}$. 
 Thus $g^mw=g^{m-\rho}(g^{\rho}w)\in C$.  Moreover,  $x-w=gt$ for some $t\in T_{(g)}$  and so 
 $g^{m+1}t=g^mx-g^mw\in g^mT_{(g)} \cap \wh{C}$. Here, $\deg t =\deg x-1$ and so, by the inductive hypothesis for $(y,x)$, we obtain $g^{m+1}t\in C$. In other words, $g^mx = g^mw+g^{m+1}t\in C$; a contradiction. Thus the claim is proven.
 
 By the claim, $I :=\wh{C} \cap g^{\rho}T_{(g)} $ is a non-zero ideal of both $C$ and $\wh{C}$, and so certainly $C$ and $\wh{C}$ are equivalent orders.  
 As $g\wh{C}=\wh{C}\cap gT_{(g)}$, an easy induction shows that 
 $g^{\rho}\wh{C}= \wh{C} \cap g^{\rho}T_{(g)}  = C\cap g^{\rho}T_{(g)}$.  
  Thus,  $g^{\rho}\wh{C}\subseteq I\subset C$ and so $\wh{R}=\wh{C}$ is a noetherian $C$-module and hence a noetherian $R$-module on both sides, while it is cg by Remark~\ref{rem:2.33}. 
  Moreover,  $g^{\rho}R'= g^{\rho}\wh{C}^{(n)} \subseteq C^{(n)}=R\subseteq R'$ and so $R'$ and $R$ are indeed equivalent orders.  \qed

\smallskip

 Returning to the proof of the lemma note that, 
 by the  second paragraph of the proof,   there  exists   $0\not=x \in B_n$. Suppose first that $\dim_\kk Q(i)_j\leq 1$ for all $i,j$. As $B$ and the $Q(i)$ are 
 contained in the domain $T_{(g)}/gT_{(g)}\cong \kk(E)[z,z^{-1};\tau]$, this implies that $B$ contains the domain $\kk[x]$. Moreover,  
 $W :=\bigcup Q(i)$  
  satisfies $\dim_\kk W_n\leq 1$ for all $n\geq 0$. As such,  $W$ is finitely generated  
 as a $k[z]$-module and hence as a $B$-module.  In particular, $Q(r)=Q(r+1)$ for all  $r\gg 0$ and so the lemma follows from
 Sublemma~\ref{sublem:prime11}. (In fact a little more work   shows that this case cannot   happen.)

We may therefore assume that          there exists  $i,j\geq 1$ such that $V := Q(i)_j$ has   $\dim V \geq 2$.
We next show that this implies that  $\GKdim B = \GKdim \overline{C}=2$.  Since $C^{(n)}=R$, clearly 
$(\overline{C})^{(n)} = \overline{R}^{(n)}=B$,
and so it suffices to prove that $\GKdim \overline{C}=2$.   Note that, as $C$ is a noetherian $R$-module, each $Q(i)$ is finitely generated as a
$C$-module and hence as both an $R$-module and a $B$-module on either side. 
Now  $Q(i)_{k+n} \supseteq xQ(i)_k  + Q(i)_k x$ for all $k \in \ZZ$.  
Since each $Q(i)\subseteq  \kk(E)[z,z^{-1};\tau^n]$, we may apply  Lemma~\ref{lem:paleale} with $\sigma=\tau^n$, to show that  $\dim Q(i)_{j+na} \geq \dim V + a$ for all $a\geq 1$. Hence $\GKdim Q(i) \geq 2$.  
As $Q(i)$ is a finitely generated left and right $\overline{C}$-module it follows that $\GKdim B = \GKdim \overline{C}\geq 2$.  

 Conversely, we know that $\overline{C}$ is a noetherian, cg subalgebra of   $T_{(g)}/gT_{(g)}\cong \kk(E)[z,z^{-1};\tau]$, and hence of 
  $ \kk(E)[z;\tau]$.   It is therefore   a finitely generated  algebra by the graded Nakayama's Lemma. Thus, by \cite[Theorem~0.1]{AS},  $\GKdim \overline{C}\leq 2 $. Combined with the last paragraph, this implies that  $\GKdim B = \GKdim \overline{C}=~2$. 
 
We can now apply Lemma~\ref{lem:1.6} to   $A =B$ with $B\subseteq Q(i)\subseteq Q(i+1) \subseteq T_{(g)}/gT_{(g)}\cong k(E)[z,z^{-1}; \tau]$. 
Thus, by that result, $Q(r)=Q(r+1)$ for all  $r\gg 0$ and so we can apply the sublemma. 
Since $R\subseteq R'$, the  lemma follows.   \end{proof}

%%%%%%%%%%%%%%

\section{Overrings of locally simple elliptic algebras}\label{MINIMAL}

In this section we begin to work towards Theorem~\ref{ithm:3.6}.
We fix an elliptic algebra $T\subset T_{(g)}$ with factor $B=T/gT$ and localisation  $T^\circ=T[g^{-1}]_0$ and  then  isolate   conditions that preclude the existence of proper noetherian cg overrings of $T$.   These involve  slightly awkward 
conditions (see Hypothesis~\ref{hyp:2.0}) that will be refined in later sections.

We first comment on the title to this section:  we   define     $T$, or any other elliptic algebra, 
 to be \emph{locally simple}\label{locally simple defn} if the localisation $T^\circ$ is a simple ring.  The intuition behind this term is that 
 if one regards $\rqgr T$ as the category of coherent sheaves on the noncommutative (and non-existent) projective variety $\Proj T$ then 
$\rmod T^\circ$ corresponds to  the analogous category of  sheaves on the   noncommutative affine space $\Proj T\smallsetminus E$. This  is, in turn,  the smallest nonempty open subset of  $\Proj T $.   

\begin{hypothesis}\label{hyp:2.0}  Assume that:
\begin{enumerate}
 \item The ring  $T^\circ$ is a simple domain with division ring of fractions $\Fr(T^\circ)$.
 \item  There is  a fixed $\ZZ$-graded  overring $T\subsetneqq R\subset T_{(g)}$
 and a 
fixed  simple right
$T^\circ$-module   $ \Ss \subseteq  R^\circ/T^\circ$. 
\item  There exists an extension $0\to \Ss\to \Xx(\Ss)\to \Yy\to 0$   of  right $T^\circ$-modules of finite length such that:
  \begin{enumerate}
  \item $\Xx :=\Xx(\Ss)$ has finite projective dimension; 
  \item $\Ss = \Soc(\Xx).$
  \end{enumerate}
    \item $\Ss$ cannot be written as the localisation $\Ss=L^{\circ}$ of a $T$-line module $L$.
\end{enumerate}
\medskip

 We make a few comments about these hypotheses. First, we  do not assume here  that $\Xx(\Ss)$ is unique. However, 
if $\pd(\Ss)<\infty$,  in which case (3)  holds automatically,    we   always  set $\Xx(\Ss)=\Ss$.  
Finally, note  that  the  hypotheses imply that $\Ss$ is essential in $\Xx$.   
 \end{hypothesis}
 
We first note some elementary properties of modules over rings of injective dimension one.

\begin{lemma}\label{lem:2.1new}
Let $A$ be a noetherian (or Goldie) prime ring of injective dimension one.
Then 
\begin{enumerate}
\item Every   torsionfree, finitely generated right  $A$-module $P$ of   finite projective dimension is projective.
\item Every finitely generated    torsionfree   and every  finitely generated     torsion   $A$-module  is CM. 
 \end{enumerate} 
\end{lemma}

\begin{proof}
(1)  It does no harm to replace $P$ by some direct sum $P^r$ so that $P$ has Goldie rank equal to an integer   $t$ times  the Goldie rank of $A$.   If $F=\Fr(A)$   we can then  identify $P\subseteq PF= F^t$.
 Clearing denominators on the left gives an embedding of $P$ into the finitely generated, free right $A$-module $ A^t$. Set $M=A^t/P$, which is therefore a torsion module.  Then $n :=\pd(M)<\infty$, say with $\Ext^n_{A}(M,N)\not=0 $ for some finitely generated module $N$. If $A^{m}\twoheadrightarrow N$ then, from the usual long exact sequence in cohomology, 
  $\Ext^n_A(M,A^m)\not=0 $, as well. Thus  $n\leq 1$ by hypothesis and so  $P$  is projective.
  
  (2) If $M$ is a finitely generated   torsion right $A$-module then $\Ext^n_{A}(M,A)\not=0$ if and 
    only if $n=1$.
   If $P$ is a   torsionfree right $A$-module, then it is again harmless to replace $P$ by some $P^r$. Then,  as in the proof of (1)  we can embed $P$ into a free right $A$-module $A^t$ so  that $M=A^t/P$ is torsion.  It follows that $\Ext^n_{A}(P,A)\not=0$ if and 
    only if $n=0$.
  \end{proof}
  
We next apply these properties to localisations of elliptic algebras.

\begin{lemma}\label{lem:2.1} 
Let $T$ be an elliptic algebra.
Then $T^\circ$ is Auslander-Gorenstein and Cohen-Macaulay with $\injdim (T^\circ) \leq 2$.
If in addition $T^\circ$ is simple, then 
\begin{enumerate}
\item  $T^\circ$ has  injective dimension  $\id(T^\circ)=1$.  
\item Every   torsionfree, finitely generated right  $T^\circ$-module $\Pp$ of   finite projective dimension is projective.  
\item  Every finitely generated    torsionfree   and every  finitely generated     torsion   $T^\circ$-module  is CM. 
 
\item In particular, if $T^\circ$ satisfies the conditions of  Hypothesis~\ref{hyp:2.0}, then 
$\pd(\Xx)=1$.
 \end{enumerate}\end{lemma}

 \begin{proof}   By \cite[Lemmas~2.1 and~2.2]{RSS},   $T^\circ\cong T/(g-1)T$, which  
 has a natural filtration $\Lambda$ induced from the graded structure of $T$. As   such,
 the associated graded ring $\gr_\Lambda T\cong T/gT=B$.  By \cite[Theorem~6.6]{Lev1992}, $B$ is 
 Auslander-Gorenstein and  CM of injective dimension 2. Thus, by 
 \cite[Theorem~4.1 and its proof]{Bj},  $T^\circ$ is also 
 Auslander-Gorenstein and  CM with $\id(T^\circ)\leq 2$. 
 
 Now for (1), if  $T^\circ $  is simple, it has no finite dimensional modules, so the CM condition easily implies that
 $ \Ext^2_{T^\circ}(M,T^\circ)=0 $ for all finitely generated right $T^\circ$-modules $M$, and hence that $\id(T^\circ)<2$, as required.
  Parts (2) and (3) follow from (1)  combined with Lemma~\ref{lem:2.1new}, while (4) is immediate from (2).
  \end{proof}
   
   As a partial converse to Lemma~\ref{lem:2.1} we have the following well-known result, for which we could not find an appropriate reference.  
   
   \begin{lemma}\label{lem:hereditary-simple}
   Let $A$ be a finitely generated $\kk$-algebra that is a prime noetherian, Auslander-Gorenstein and CM ring, with $\GKdim(A)=2$. 
  If $A$  is also hereditary, then $A$ is simple.
   \end{lemma}
   
   \begin{proof}  If $A$ is not simple, pick a prime ideal $P\not=0$ of $A$. Then $\GKdim(A/P)\leq 1$ by \cite[Proposition~3.5]{KL}. Thus, $A/P$ satisfies a polynomial identity by \cite{SW}. As such, $A/P$ and hence $A$ has 
   a finite dimensional factor ring $\overline{A}$ \cite[Corollary~10.9]{KL}.  But now the CM condition implies that  $j_A(\overline{A})=2$, whence
   $\Ext^2_A(\overline{A},A) \not= 0$. This contradicts the hereditary assumption.
   \end{proof}

   We next give a general lemma on torsionfree extensions, which we will use several times below.
   
 \begin{lemma}\label{general lemma}
 Let $A$ be a prime right noetherian ring and let $L$ be an essential submodule of the finitely generated right  $A$-module $M$.
 Further suppose that $\Ext^2_A( M/L, A) = 0$.
 Then every torsionfree extension $0 \to A \to X \to L \to 0$ lifts to an extension $0 \to A \to Y \to M \to 0$ and every such $Y$ is torsionfree.
 \end{lemma}
 
 \begin{proof}  Let $\alpha:  \Ext^1_A(M, A) \to \Ext^1_A(L, A)$ be the map induced from the inclusion $L \subseteq M$.
As $\alpha $ is surjective by  the assumption on $M/L$, if we regard $X$ as an element of $\Ext^1_A(L, A)$, then there exists $Y \in \alpha^{-1}(X)$.  
We may assume that   $A \subseteq X \subseteq Y$.
Let $Z$ be the torsion submodule of $Y$, and suppose $Z \neq 0$.
Because $X$ is torsionfree, $Z \cap X   = 0$.
Thus $(Z+A)\cap X = (Z \cap X)+A = A$.
Now $Y/A \supseteq (Z+A)/A \neq 0$ and   $X/A \cong L$ is essential in $Y/A \cong M$, from which  
it follows that  $\bigl((Z+A)/A\bigr) \cap X/A \neq 0$. Thus $(Z+A) \cap X \supsetneqq A$, a contradiction.
\end{proof}

 \begin{lemma}\label{lem:2.2}  Assume that $T\subsetneqq R$ satisfy  
 Hypothesis~\ref{hyp:2.0}.   
 Let $R^\circ \supseteq \Cc\supset T^\circ$ be a  (torsionfree) $T^\circ$-module 
 such that $\Cc/T^\circ\cong \Ss$.    
 Then   there exists a projective $T^\circ$-module  $\wc{\Cc} $
 such that $\Fr(T^\circ)\supset \wc{\Cc} \supseteq \Cc$ 
  and  $\wc{\Cc}/T^\circ\cong \Xx(\Ss) = \Xx$.   \end{lemma}
 
 \begin{proof} If $\Cc$ is projective then $\pd(\Ss)<\infty$ and $\Xx = \Ss$.  Thus   $\wc{\Cc}=\Cc$, which  is automatically 
 projective, as required. So assume that $\Cc$ is not projective. 
  
   Take the short exact sequence 
 $0\too \Ss\  \buildrel{a}\over\too \   \Xx\  \buildrel{b}\over\too \   \Yy\ \too  0
 $ given by  Hypothesis~\ref{hyp:2.0}, noting that $\Ss$ is essential in $\Xx$. By  Lemma~\ref{lem:2.1} ,
  $A=T^\circ$  has injective dimension 1 and so, by Lemma~\ref{general lemma},
   $\Cc$ lifts to an extension $0 \to T^\circ \to \Mm \to \Xx \to 0$, in which  $\Mm$ is torsionfree.
   This also  implies that there is a natural embedding $\Mm\hookrightarrow  \Fr(T^\circ)$ extending the 
  embedding of $\Cc$.   
 By  Lemma~\ref{lem:2.1},  $\wc{\Cc}=\Mm$    is projective.    
 \end{proof}

\begin{notation}\label{not:2.31}    Let $\Mm$ be a $T^\circ$-submodule of $\Fr(T^\circ)$. 
Following \cite[Section~7, p.2099]{RSSlong} we define
$ \Phi  \Mm :=\bigoplus_{n\in \mathbb{Z}} (\Phi \Mm)_n$, where $(\Phi \Mm)_n
:= \{a\in (T_{(g)})_n : ag^{-n}\in \Mm\}$.   As $T$ is   $g$-divisible, it is immediate that $T=\Phi(T^\circ)$ 
and hence that $\Phi \Mm$ is a $T$-module. We remark that \cite{RSSlong} used $\Omega$ in place of 
$\Phi$, but in this paper $\Omega$ will be reserved for another more longstanding construction. 

Let $X$ be a $T$-submodule of $T_{(g)}$ and recall the definition of the $g$-divisible hull $\wh{X}$ from Definition~\ref{g-divisible-defn}.  We note that $\wh{X} = \Phi(X^\circ)$. 
\end{notation}

We are now ready to prove the  first main   result on the non-existence of finitely graded  
 overrings; indeed the   main theorems from the introduction will ultimately reduce to this case.  We note that the idea of the proof originates in Goodearl's result  \cite[Theorem~5]{Goodearl} that overrings of HNP rings  are localisations.

\begin{proposition}\label{prop:2.3} 
Assume that $T\subsetneqq R$ satisfies the conditions of Hypothesis~\ref{hyp:2.0}. 
\begin{enumerate}
\item If $R $ is $g$-divisible, then $\dim R_0 = \infty$.
\item If $R$ is finitely graded, then $R$ is not noetherian.  
\end{enumerate}
\end{proposition}

\begin{proof}    
   (1)
  It is immediate that $\Phi (R^\circ) = \widehat{R}=R$.
Now let $T^\circ \subset \ \Cc\subset R^\circ$ be a $T^\circ$-module with 
$\Cc/T^\circ\cong \Ss$. By Hypothesis~\ref{hyp:2.0}(4),   Proposition~\ref{prop:linesonly}(2) applies and
 provides a  $T$-module extension 
$T\subsetneqq D\subset T_{(g)} $
with   $D^\circ/T^\circ \cong \Ss$,
and so that $D_r\supsetneqq T_r$ for some $r\leq 0$. 
Let $\Dd := D^\circ$.

Apply Lemma~\ref{lem:2.2}. This provides extensions 
$$ T^\circ \subsetneqq \Cc \subset \wc{\Cc} \subset \Fr(T^\circ)
\qquad \text{and}\qquad
T^\circ \subset \Dd \subset \wc{\Dd} \subset \Fr(T^\circ),$$
such that $\Cc /T^\circ = \Soc(\wc{\Cc}/T^\circ)$ 
and  $\Dd /T^\circ = \Soc(\wc{\Dd}/T^\circ)$ and  there is an isomorphism $$\chi:\wc{\Cc}/T^\circ \ \buildrel{\cong}\over{\too}\  \Xx\ \buildrel{\cong}\over{\too}\   \wc{\Dd} /T^\circ.$$
(If $\Ss$ and hence $\Cc$ have finite projective dimension then  
$\Xx=\Ss$ whence
  $\wc{\Cc}=\Cc$ and $\wc{\Dd}=\Dd$.  In this case the desired properties hold tautologically.)

Let $\pi_1: \wc{\Cc} \to \wc{\Cc}/T^\circ$ and $\pi_2: \wc{\Dd} \to \wc{\Dd}/T^\circ$ be the natural projections. 
 By  Lemma~\ref{lem:2.2}, $\wc{\Cc}$ is projective, and so   the isomorphism $\chi$ lifts  to a $T^\circ$-module homomorphism   $\xi: \wc{\Cc} \to 
\wc{\Dd}$  making the following diagram commute.
\begin{equation}\label{2.35}
 \xymatrix{  \wc{\Cc}/T^\circ  \ar[r]^\chi & \wc{\Dd}/T^\circ   \\
 \wc{\Cc}   \ar@{.>}[r]^\xi\ar[u]^{\pi_1} & \wc{\Dd} \ar[u]^{\pi_2}
 } 
\end{equation}
As $\wc{\Cc} $ and $\wc{\Dd}$ are submodules of $\Fr(T^\circ)$, 
$\xi$ is given by left multiplication  by an  element $x\in \Fr(T^\circ).$  
Since  $\chi\pi_1(T^\circ)=0$  the  commutativity of \eqref{2.35}    
implies that  $\pi_2\xi(T^\circ)=0$ and
 $\xi(T^\circ)\subseteq T^\circ$. In other words, $x\in T^\circ$.
 
 Next, since $\Cc/T^\circ = \Soc(\wc{\Cc}/T^\circ)$, and $\pi_2\xi = \chi \pi_1$ is surjective, 
   the map $\xi$ must map $\Cc$ to 
 the preimage of the simple socle  $ \Soc(\wc{\Dd}/T^\circ)=\Dd/T^\circ$.  Consequently,
  $\Dd/T^\circ= (x\Cc+T^\circ)/T^\circ.$
 In particular, since $R^\circ $ is a ring,
  $\Dd = (x\Cc+T^\circ) \subseteq xR^\circ + T^\circ \subseteq R^\circ.$ 
  Since $\Dd = D^\circ$, it follows that  $D \subseteq \Phi(R^\circ) = R$, and so, by Remark~\ref{rem:2.33}(1), 
   $\dim R_0 = \infty$.   
   
   (2) 
Suppose that $R$ is finitely graded  and    noetherian. 
Then, by  Lemma~\ref{lem:1.7prime}, $\widehat{R}$ is a finitely generated right (and left) 
$R$-module and so   $\widehat{R}$ is also finitely graded. Moreover,  by construction, $\wh{R}^\circ = R^\circ$ and so Hypothesis~\ref{hyp:2.0} holds for $\wh{R}$.   Thus we can  apply Part (1) to $\wh{R}$  to  show that $\wh{R}_0$ is infinite dimensional;  a contradiction.
    \end{proof}
       
  We remark that in the penultimate sentence of  the proof of Part (1) of Proposition~\ref{prop:2.3} 
  one needs only that $R^\circ$ is a $T^\circ$-bimodule.  
  In other words, the argument above proves the following  statement. Since the result will not be used in the paper, the details are left to the interested reader. 
  
  \begin{corollary}\label{lem:bimodule}
   Assume that $T\subsetneqq R$ satisfies the conditions of Hypothesis~\ref{hyp:2.0}.
Let $\Ee \subseteq \Fr(T^\circ)$ be the maximal extension of $T^\circ$ by a direct sum of copies of $\Ss$, and note that $\Ee$ is a $T^\circ$-bimodule.
Then $\Ee/T^\circ$ is simple as a $T^\circ$-bimodule. \qed
\end{corollary}

  %%%%%%%%%%%%%%%%%
      
\section{Algebras with  \texorpdfstring{$A_1$}{LG}  singularities}\label{A1SINGULARITIES}

In the next  section we will show that Proposition~\ref{prop:2.3}  can be applied to show that  the quadric algebras  $Q$ defined 
by Van den Bergh in \cite{VdB3}  are indeed minimal elliptic surfaces.   The corresponding algebra $Q^\circ$ will either have finite 
global dimension or be simple with a mild singularity. In this section we prepare for the latter  case by studying arbitrary elliptic algebras 
$T$ for which $T^\circ$ has such a singularity. The relevant  definitions are as follows. As usual, given a right (left) module $M$ over a ring $A$ we write $M^*$ for the left (right) $A$-module $\Hom_A(M,A)$.

\begin{definition}\label{sing-defn}
  Given a noetherian ring $A$ of finite injective dimension 
  define the \emph{singularity category $\Sing(A)$} as in \cite[Section~I.5, p.46]{AB} (where it is called the Stabilised Category)   and write  Hom   groups  in $\Sing(A)$
by $\uHom_{A}(M,N)$, for   right $A$-modules $M$ and $N$.   
By \cite[Theorem~4.4.1]{Buchweitz},
 $\Sing(A)$ can be identified with the
 quotient category           
 \[\Sing(A) \simeq \{ \mbox{maximal Cohen-Macaulay (MCM) modules} \} / \{\mbox{projective modules}\}.\]
For  right   MCM   modules $M, N$ we have, by \cite[2.1]{Buchweitz},
\[\uHom_A(M,N) = \Hom_A(M, N)/ \{\mbox{maps   factoring through a projective module}\} =  \Hom_A(M, N)/NM^*,\]
where  $NM^*$ is the natural image of $N\otimes M^*$ inside $\Hom_A(M,N)$. 
 By \cite[Theorem~4.4.1]{Buchweitz},  $\Sing(A)$ is triangulated with the inverse of the homological shift $\Sigma$ given by the syzygy functor
  $\Omega$.   

We say that $A$ has an {\em $A_1$ singularity} if $\Sing(A) $
is triangle equivalent to the category  of $\kk$-vector spaces $\operatorname{Vect}$ (necessarily with trivial homological shift).  
If $\chrr \kk \neq 2$, then   the  Kleinian singularity $\kk[[x,y]]^{C_2}$ has an $A_1$ singularity, whence the name.
 \end{definition}

In the next few results, we give some basic properties of rings with $A_1$ singularities.  

\begin{remark}\label{rem:3.1}
Observe that if an algebra $A$ 
 has an $A_1$ singularity, then there exists 
  an MCM right $A$-module $\Mm$  (equivalently, by Lemma~\ref{lem:2.1new},   a finitely generated torsionfree module  $\Mm$),
   so that
\begin{enumerate}
\item[(a)] $\Mm$ corresponds to $\kk$ under the equivalence $\Sing(A) \simeq \operatorname{Vect}$,  
\item[(b)]  $\uEnd_{A}(\Mm) = \End_{A}(\Mm)/\Mm \Mm^* \cong \kk$,  and 
\item[(c)]  the first syzygy  $\Omega \Mm \cong \Mm$ in $\Sing(A)$. 
\end{enumerate}
We refer to $\Mm$ as the {\em generator of $\Sing(A)$}.
\end{remark}

\begin{lemma}\label{lem:3.3}
Let $A$ be a prime noetherian ring of injective dimension 1 that is either hereditary or has an $A_1$ singularity.
Let $\Nn$ be a  non-zero, finitely generated,  torsionfree right $A$-module.
Let $\Mm$ generate $\Sing(A)$ if $\Sing(A)$ is nontrivial, or else let $\Mm = A$.
Then  the following hold.
\begin{enumerate}
\item  There exist an integer $s\geq 0$ and finitely generated  projective right $A$-modules $\Aa$ and $\Bb$ such that 
$\Nn\oplus \Aa\cong  \Bb\oplus \Mm^{(s)}.$
\item    There exist a projective right $A$-module $\Ll$ and a short exact sequence 
$0\to \Nn\to \Ll\to \Nn\to 0$.   In particular, this is true for $\Nn=\Mm$.
\end{enumerate}
\end{lemma}
 
\begin{proof}
We may assume that $\Sing(A)$ is nontrivial.

 (1)  By Remark~\ref{rem:3.1}(a)  there exists $s\geq 0$ such that 
 $\Nn\cong \Mm^{(s)}$ in $\Sing(A)$.  Now apply
  \cite[Proposition~1.44(4)]{AB}.
  
  (2)   The result is trivial if $\Nn$ is projective, so assume not. Pick an isomorphism 
  $\Nn\cong \Mm^{(s)}$ in $\Sing(A)$ by Part (1) and note that $s\geq 1$ as $\Nn$ is not projective.  
 Choose a surjection $\chi:\Pp \to \Nn$ where $\Pp$ is a finitely generated projective module  and set $\Nn'=\Ker(\chi)$.
 By   Remark~\ref{rem:3.1}(c), $\Nn' \cong (\Omega \Mm)^{(s)} \cong \Mm^{(s)} \cong \Nn$ in $\Sing(A)$.
 Thus there are projective modules $\Aa$ and $ \Bb$ so that $\Nn' \oplus \Aa \cong \Nn \oplus \Bb$.
 
From the short exact sequence $0 \to \Nn' \to \Pp \to \Nn \to 0$ we obtain a short exact sequence
\[ 0 \too \Nn' \oplus \Aa \too \Pp \oplus \Aa \too \Nn \too 0.\]
By the previous paragraph, this induces
an exact sequence 
 \[ 0\too \Nn \oplus \Bb \ \buildrel{\alpha}\over\too\  \Pp \oplus \Aa \ \buildrel{\beta}\over\too\  \Nn  \to 0.\]
Set  
$\Ll   :=(\Pp \oplus \Aa)/\alpha(\Bb)$.
Then  
 \[ \Ll\ \supseteq \Zz\ := \  \frac{   \alpha(\Nn) \oplus \alpha(\Bb)  }{\alpha(B) } \ \cong \ \Nn
\qquad \text{with}\qquad 
\Ll/\Zz \ \cong\  \frac{ \Pp\oplus \Aa }{ \alpha(\Nn\oplus \Bb) }
\ \cong\   \Nn.\]
Finally, by its initial construction $\pd(\Ll)\leq 1$ and from the last  displayed equation it follows  that $\Ll$ is torsionfree. Thus, 
by Lemma~\ref{lem:2.1new},  $\Ll$ is   projective, as required.
\end{proof}

\begin{lemma}\label{lem:3.4} 
Let $A$ be a prime noetherian ring of injective dimension 1 that is either hereditary or has an $A_1$ singularity. Let $\Cc$ be a  right $A $-module  with $A \subset \Cc \subset \Fr(A)$ and  such that $\Cc/A \cong \Ss$ is a
simple right $A$-module.

Then either $\pd(\Ss )<\infty$ or there exists a finitely generated right $A $-module $\Xx=\Xx(\Ss)$ of finite projective dimension for which there exists an
 extension \[0\too \Ss\too \Xx\too \Ss\too 0.\]
   \end{lemma}

\begin{proof} We may assume that $\pd(\Ss)=\infty$. 
By Lemma~\ref{lem:3.3}(2) we have a short exact sequence 
\[ 0\too \Cc_1 \ \buildrel{\alpha}\over\too\    \Ll \ \buildrel{\beta}\over\too\    \Cc_2\too 0\] with $\Ll$ projective and the $\Cc_i$ being
 the two copies of $\Cc$. The natural inclusion of $A$ into   $\Cc_2$ lifts to a monomorphism $\theta: A \to \Ll$. Thus if $\iota(A )$ is
  the given copy of $A $ inside $\Cc_1$ then $\iota(A)\oplus \theta(A)\subset \Ll$. 
Clearly,
\[ \frac{\Ll}{ \bigl(\iota(A)\oplus \theta(A)\bigr) +\Cc_1 } \ \cong\ \frac{\Cc_2}{A } \ = \ \Ss,\]
while \[\Zz\ := \ \frac{  \bigl(\iota(A)\oplus \theta(A)\bigr) +\Cc_1}{\iota(A)\oplus \theta(A)}\]
is a homomorphic image of $\Cc_1/\iota(A)\cong \Ss.$  Since $\Ss$ is simple and $\Zz\not=0$,
 this implies that $\Zz\cong \Ss$. 

Finally, consider $\Xx  := \Ll/\left(\iota(A)\oplus \theta(A)\right)$. Then, clearly $\pd(\Xx) \leq 1$, while the last paragraph ensures that there is the required  short exact sequence $0\to \Ss\to \Xx\to \Ss\to 0$.
\end{proof}

 Applying the results of this section to elliptic algebras we obtain our first result on minimal elliptic surfaces:

\begin{theorem}\label{thm:3.5}  
   Let $T$    be an elliptic algebra with   localisation  $T^\circ$.      
      Assume either  that $T^\circ$ is hereditary or that   $T^\circ$ is simple with an  $A_1$ singularity.  
        Further assume that $T$ has no line modules.  Then $T$ is a minimal elliptic surface.    
  
    More generally, let $T \subseteq R \subset  T_{(g)}$ be a graded overring. Then:
    \begin{enumerate}
\item  if  either $\wh{R}$ is finitely graded or $R$ is both finitely graded and    noetherian, then $R = T$; 
\item  if $ T \neq R$ and $R = \wh{R}$ is $g$-divisible, then $\dim R_0 = \infty$.\end{enumerate}
\end{theorem}

\begin{remark}\label{rem:111}  If $T^\circ$ is hereditary  and $\deg T\geq 3$, then  the conclusion of  theorem still holds provided $T$ has no line modules of self-intersection $(-1)$. See Theorem~\ref{thm111} for the details.
\end{remark}
 
\begin{proof} Note that $T^\circ$ is a finitely generated $\kk$-algebra by \cite[Lemma~2.1]{RSS}.
Thus,  if $T^\circ$ is hereditary, then $T^\circ$ is also simple by Lemma~\ref{lem:hereditary-simple}.

Pick  a simple $T^\circ$-submodule  $\Ss 
\subseteq  R^\circ/T^\circ$. If $\Ss$ has infinite projective dimension then, by Lemma~\ref{lem:3.4}, the conditions of
Hypothesis~\ref{hyp:2.0} are  satisfied.   If $\pd(\Ss)<\infty$, then  
Hypothesis~\ref{hyp:2.0} is immediate. In either case, the result  follows from Proposition~\ref{prop:2.3}. 
 \end{proof}

We note the following curious consequence of the theorem, which seems to us a strikingly weird result: 

\begin{corollary}\label{cor:weird}
Let $T$ be as in Theorem~\ref{thm:3.5} and let $R $ be the subalgebra of $T_{(g)}$ generated by $T$ and $x$,  for any homogeneous element $x\in T_{(g)} \ssm gT_{(g)}$ of positive degree. 
Then:
\begin{enumerate}\item  R is not noetherian;
\item $\wh{R}_0$ is infinite dimensional.\qed  
\end{enumerate}
\end{corollary}

Finally, we note for future reference the following sufficient condition, due to Simon Crawford, for when a ring has an $A_1$ singularity.  
Note that the hypothesis that $M$ is a generator is automatic if $A$ is simple.

\begin{proposition}\label{prop:Simon}
{\rm (\cite[Theorem~4.5.7]{Simonthesis})}
 Let $A$ be  a left and right noetherian $\kk$-algebra of injective dimension at most 2. 
Suppose that  $A$ has an MCM right module $M$ which is a generator  such that $\gldim \End_A (M) \leq 2$ and $\uEnd_A(M) = \End_A(M)/MM^* \cong \kk$.
 Then $\gldim A=\infty$  and $A$  has an $A_1$ singularity.    \end{proposition}

  %%%%%%%%%%%%%%%%%
  
\section{Sklyanin algebras and Van den Bergh quadrics}\label{ALGEBRAS}

In this section we apply  Theorem~\ref{thm:3.5} to  some well-known elliptic algebras $T$: Sklyanin algebras and
 the quadric algebras  
  defined  by Van den Bergh in \cite{VdB3}.  The algebras break into two cases, the first of which is when $T$  satisfies the hypotheses of Theorem~\ref{thm:3.5}. The other case is when $\gldim T^\circ=2$. In the latter case we show that there is a Morita context between
   $T$ and a second elliptic algebra $T'$ to which Theorem~\ref{thm:3.5} can be applied. This is enough to prove that $T$ is also a minimal    elliptic surface.

  The formal definitions are as follows.

\begin{example}\label{eg:sklyanin} The (quadratic) Sklyanin algebra is defined to be 
$$S :=\Skl(a,b,c):=\kk\{x_1,x_2,x_3\}/(ax_ix_{i+1}+bx_{i+1}x_i+cx_{i+2}^2 : i\in \ZZ_3),$$
where  $[a\hskip-2pt : \hskip-2pt b\hskip-2pt  :\hskip-2pt c] \in \mathbb{P}^2\smallsetminus \Sigma$ for a (known)   set $\Sigma$. 
Here, $S$ contains a canonical central element $g\in S_3$ such that $S/gS\cong B(E,\mathcal{L},\sigma)$ for an elliptic curve $E$
with a line bundle $\mc{L}$ of degree 3. In this paper we assume that $|\sigma|=\infty$, from which 
it follows  that  the 3-Veronese ring $T=S^{(3)}$ is an elliptic algebra, which we sometimes call \emph{the Sklyanin elliptic algebra}. See, for example, 
\cite[Theorem~6.8(1)]{ATV1990} for the details.
\end{example}

\begin{example} \label{eg:cubic-sklyanin}
The cubic Sklyanin algebra is defined as
\[ S' := \Skl'(a,b,c) = \kk\{x_1, x_2\}/(a(x_{i+1}^2x_i+x_i^2x_{i+1})+bx_{i+1} x_ix_{i+1}+ cx_i^3 : i \in \ZZ_2),\]
for  $[a\hskip-2pt : \hskip-2pt b\hskip-2pt  :\hskip-2pt c] \in \PP^2 \ssm \Sigma$ for  a known   set $\Sigma$.
Then $S'$ contains a central element $g \in S'_4$ so that $S'/gS' \cong B(E, \sL', \sigma)$ for an elliptic curve $E$ with a line bundle $\sL'$ of degree 2.  As above, we assume that $|\sigma| = \infty$, and so $(S')^{(4)}$ is an elliptic algebra.
Details are in \cite{ATV1990} as above. We will not need to treat this case separately, since the Veronese ring $(S')^{(2)}$ is a noncommutative quadric in the sense of the next example (this is implicit in \cite{VdB3} and proven explicitly in \cite[Theorem~7.1]{DL}). 
\end{example}

 \begin{example}\label{eg:quadric}
 Assume now that $\chrr \kk  =0 $.
  For the purposes of this paper we define noncommutative quadrics as follows. First, 
 let $\SK$ denote a 4-dimensional Sklyanin algebra; thus $\SK$ is the $\kk$-algebra
with  $4$ generators $x_0,\dots ,x_3$ and the $6$ relations
$$x_0x_i-x_ix_0=\alpha_i(x_{i+1}x_{i+2}+x_{i+2}x_{i+1}),\quad 
x_0x_i+x_ix_0=
x_{i+1}x_{i+2}-x_{i+2}x_{i+1},$$ where the subscripts are taken to be
$\{1,2,3\}$ mod $3$ and the $\alpha_i$ satisfy $\alpha_1\alpha_2\alpha_3
+ 
\alpha_1+\alpha_2+\alpha_3=0$ and $\{\alpha_i\}\cap\{0,\pm
1\}=\emptyset$.
Significant properties from \cite{SSf} are  that $S=\SK$ is a noetherian Artin-Schelter 
regular domain of dimension $4$, as defined in \cite{SSf},   with the Hilbert series $(1-t)^{-4}$ of a polynomial ring in 4 variables. Moreover,   $S$
has a two-dimensional space of  central  homogeneous
elements
$V\subset S_2$. The factor ring 
$S/SV$ is isomorphic to a twisted homogeneous
coordinate
ring $B(E,\mathcal{L},\sigma)$, where $E$ is a
 smooth elliptic curve and
$\sigma$ is an automorphism given by 
translation by a point also denoted $\sigma\in E$. However,  $\mathcal{L}$ is now a line
bundle of degree $4$.  For the purposes of this paper we again assume that $|\sigma|=\infty$.

 Given a non-zero element $\Omega\in V$ we define \emph{the Van den Bergh quadric}
  $\QVB:=\QVB(\Omega):=S/S\Omega$.  \label{VDB-defn} 
  As explained,  for example in \cite{SV} this is a domain which is Artin-Schelter Gorenstein in the sense of \cite{ASc}.  
 Fixing  a basis element $g$ for the image of $V$ in $\QVB$ ensures 
  that   $T :=\QVB^{(2)}$
 is an elliptic algebra, which we term  \emph{a quadric elliptic algebra}.

 The structure of $\QVB$ depends upon the choice of  $\Omega$  and so  we refine our notation as follows.
 We use the notation from  \cite[(10.3)]{SV},  where the reader is referred for more details. First identify  $E\subset \mathbb{P}(S_1^*)$. 
 Given two points $p,q\in E$ we can define a left $S$-line module $L(\overline{pq}):=S/SW$, where $W\subset S_1$ consists of the forms
  vanishing on the line $\overline{pq}  \subset  \mathbb{P}(S_1^*)$. It is known that, up to scalar multiple,   there exists    a unique   
   $\Omega\in V$ vanishing on $L(\overline{pq})$.
In fact, $\Omega$ depends only on the sum $r=p+q$ under the group law on $E$, so we write $\Omega = \Omega(r)$.
 Finally, we define $\QVB(r) := S/S\Omega(r)$.   As noted in \cite[(10.3)]{SV}, $\QVB(r)=\QVB(-r-2\sigma)$. \end{example} 
     
       \begin{remark}\label{rem:char0} 
       The characteristic zero hypothesis is  only  needed  in Example~\ref{eg:quadric} because this   hypothesis is assumed in the  results 
       we cite from  \cite{LS,VdB1}, and we conjecture that all the results hold over arbitrary fields. Certainly  our results on  quadratic 
       Sklyanin algebras are characteristic-free, and one can give characteristic-free proofs for the results on cubic Sklyanin algebras.               \end{remark}

 We remark that in \cite{VdB3},  Van den Bergh actually defines his quadrics as certain categories deforming the category of 
 quasicoherent sheaves on a commutative quadric surface. It is only after his classification that one can identify his ``elliptic'' quadrics 
 as $\rqgr R$ for factors $R$ of $\SK$. Since this relationship is not relevant to the present paper, we refer the interested  reader to \cite{VdB3} and \cite[Sections~11 and 12]{StV} for the details.  These quadrics  include, but are much more general than,  the noncommutative analogue of $\mathbb{P}^1\times\mathbb{P}^1$ (see \cite[Introduction]{VdB3})  as well as a sort of noncommutative analogue of a Hirzebruch surface $F_2$ (see \cite[Section~10.5]{SV}).

   Much is known about the Van den Bergh quadrics, and we begin by collecting some of those known facts.
    We start with another definition.
   
  \begin{definition}\label{defn:smooth} If $T$ is a cg noetherian  algebra, then $\rqgr T$ is   \emph{smooth} if it has finite homological dimension. By \cite[Lemma~6.8]{RSSblowdown},  if $T$ is an elliptic algebra, this is equivalent to asserting that $\gldim T^\circ < \infty$.
  \end{definition}

    \begin{lemma}\label{lem:Smith-VdB}  Let $Q  :=\QVB(r) = S/\Omega(r)$   for some $r\in E$, and let $E_2$ denote the points of order 2 in $E$. Then:
    \begin{enumerate} 
    \item $\rqgr Q$ is smooth if and only if $r+\sigma\not\in E_2$. 
    \item Assume that  $r+\sigma\not\in E_2$. Then either $Q^\circ$ is simple hereditary or $\gldim Q^\circ=2$. In the latter case $Q^\circ$  has, up to isomorphism, a unique finite dimensional simple module, say of dimension $d$. 
    This occurs if and only if $r=\omega+n\sigma$ for some $\omega\in E_2$, $n \in \ZZ \ssm \{-1\}$, with $d = \abs{n+1}$. 
    
    \item    If  $r+\sigma\in E_2$, then $Q^\circ$ is simple of infinite global homological dimension. 

\item All the above results also hold for $Q=\QVB(r)^{(2)}$. \end{enumerate}
 \end{lemma}
 
 \begin{proof}  (1)  This  is \cite[Theorem~10.2]{SV}.  
 
 (2)  By Lemma~\ref{lem:2.1}, $Q^\circ$ is CM of finite injective dimension and Gelfand-Kirillov dimension 2, for any choice of $r$. As such,
  if $\rqgr Q$ is smooth then, necessarily,  either $Q^\circ $ is  simple or $\gldim Q^\circ=2$, in which case $Q^\circ$ has at least one finite dimensional simple module. So we need to identify when the latter happens.

 Recall that a \emph{fat point} over $Q$ (or $S$)
 is a 1-critical graded $Q$-module $M$. As such, there exists $d$ such that 
 $\dim_{\kk}M_n=d$, for $n\gg 0$, called the \emph{multiplicity of $M$}. \label{multiplicity-defn}
  Note that if $M$ is a 1-critical $S$-module, then either $M$ is a $g$-torsion point module (and hence a $B$-module) or $M$ is killed by at most one $\Omega\in V$, up to scalar multiples. 
 
   If $M$ is a $g$-torsionfree point module, then  \cite[Proposition~2.4]{SSf} implies that $M$ is   one of 4 exceptional  point modules   and, by  \cite[Theorem~5.7]{LS} this implies that $r\in \{\omega, \omega-2\sigma\}$ for some $\omega\in E_2$. On the other
    hand, if $M$ is a fat point of multiplicity $d>1$, then $r\in \{\omega -1 \pm d\sigma\}$ by
  \cite[Proposition~4.4]{SS}.   Moreover, by  \cite[Remark~p.84]{SS} if  $r\in \{\omega -1 \pm d\sigma\}$ for $d\geq 1$,   there is only one fat 
  point in $\rqgr Q$. 
  By \cite[Lemma~2.1]{RSS}, $Q^\circ \cong Q^{(2)}/(g-1)$. Thus  $Q^\circ$ then  has exactly  one finite dimensional, simple module 
  $M^\circ$ and, moreover, $\dim_{\kk} M^\circ =d$.

  (3)  The proof is similar to that of Lemma~\ref{lem:hereditary-simple}. 
   Assume that $r+\sigma \in E_2$.  By (1) $\gldim Q^\circ=\infty$ and, by (2), it cannot have any finite dimensional modules.
  Suppose, however, that $Q^\circ$ is not a simple ring; thus it has a proper, nonzero prime factor ring $A=Q^\circ/I$. 
  By \cite[Proposition~3.15]{KL}, $\GKdim A\leq 1$ and so,  by \cite{SW} and  \cite[Corollary~10.9]{KL},
    $A$ has a nonzero finite dimensional module, a contradiction. 
 Thus $Q^\circ$ is simple. 
  
 (4) This follows immediately from the fact that $\rqgr Q=\rqgr Q^{(2)}$ (see for example \cite[Proposition~6.2]{AS})
 and $Q^\circ=(Q^{(2)})^\circ$.
 \end{proof}

If $r' = r+n\sigma$ for some $n \in \ZZ$ then Van den Bergh shows there is a Morita context between $\QVB(r)$ and $\QVB(r')$.
In the next  result, we show that if $n=1$, this Morita context is given by a \emph{line ideal}:  a right ideal $J$ of $Q= \QVB(r)$ so that  
 $Q/J $ is a line module.
 
  As a matter of notation, suppose that $A$ is a noetherian  domain with  division ring of fractions $\Fr(A)$ and a finitely generated 
   fractional right ideal $J$. Then we will identify $J^* =\Hom_A(J,A) $ with $\{\theta\in \Fr(A) : \theta J\subseteq A\}$ and 
   $\End_A(J) $ with $\{\theta\in \Fr(A) : \theta J\subseteq J\}$, with similar conventions for left modules.   The following well-known lemma will be used frequently.

\begin{lemma}\label{mo-lemma} Let $A$ and $A'$ be noetherian maximal orders with the same division ring of fractions $\Fr(A)$ and suppose that 
there is an $(A,A')$-bimodule $M$ contained in $\Fr(A)$ and finitely generated on both sides. Then;
\begin{enumerate} 
\item $A'=\End_{A}(M)=\End_{A}(M^{**})$;
\item $\Hom_A(M,A)=\Hom_{A'}(M,A')$, so the notation $M^*$ is unambiguous.
\end{enumerate}
\end{lemma}

\begin{remark}\label{mo-remark}  All the rings that concern us, in particular Sklyanin   elliptic algebras,   quadric elliptic algebras as well as 
the Sklyanin and quadric algebras themselves are maximal orders by \cite[Theorem~6.7]{R-Sklyanin}  and so the 
 lemma applies to them.
\end{remark} 

\begin{proof} (1) By hypothesis $A$ and $A'$ are equivalent orders in the sense that there exists non-zero elements $a,b,c,d\in \Fr(A)$
such that $aA'b\subseteq A$ and $cAd\subseteq A'$. Similarly, $A$ and $\End_A(M)$ are equivalent orders. Moreover $A'\subseteq \End_{A}(M)$. Now apply  \cite[Lemma~6.2]{RSSlong}.

(2) Note that $M^*=\Hom_A(M,A)$ satisfies $MM^*\subseteq A'=\End_A(M)$ and so $M^*\subseteq \Hom_{A'}(M,A')$. So 
$M^*=\Hom_{A'}(M,A')$ by symmetry.
\end{proof}

\begin{lemma}\label{lem:translation}
Let $Q := \QVB(r)$ be a Van den Bergh quadric, and let $r' = r+\sigma$.
Then there is a right ideal $J$ of $Q$ so that $\End_Q(J) \cong \QVB(r') $ and $\End_{\QVB(r')}(J) \cong Q$.  
Thus  $J$  is a $(\QVB(r'),\, Q)$-bimodule.   
 Further, we may choose $J$ to be a   line ideal that is MCM.

\end{lemma}

\begin{proof}
Set $Q' := \QVB(r')$ and $A :=\SK$.
By \cite[(7.12) and Lemma~7.4.3]{VdB1}, there is a $(Q', Q)$-bimodule $J$ with 
\[ \hilb J \ =\ \hilb Q - 1/(1-t)^2 + H(t) \quad \mbox{ for some some Laurent polynomial  $H(t)$}\]
so that $\End_Q(J) \cong Q'$. 
Technically, $J$ is only defined as an element of $\rqgr Q$, so we take $J$ to be some representative in $\rgr Q$ of this equivalence class. 
In particular, from the proof and notation of  \cite[Lemma~7.4.3]{VdB1}  we can take 
\[ J \ :=\ \bigoplus_m \Gamma(\PP^1, \sO_{\PP^1}(-1) \otimes_{\PP^1} (\sB_{z+2\sigma})_m)\ 
 \subseteq\ \bigoplus_m \Gamma(\PP^1, (\sB_{z+2\sigma})_m) \ =\ Q,\]
so $J$ is a right ideal of $Q$. Next, as $Q'$ is a maximal order  by Remark~\ref{mo-remark}, 
 it follows from Lemma~\ref{mo-lemma} 
that $Q'=\End_Q(J^{**})$, and so we can replace $J$ by $J^{**}$.   Necessarily,  
$ \GKdim(Q/J)=2$ still holds and $L=Q/J$ still has multiplicity one. However, by the CM condition, $L$ now has no submodules of GKdim $\leq 1$.

Now  consider $L$ as an $A$-module.  By the CM condition for $A$, we have  $j_A(L)=\GKdim A-\GKdim L=2$. Thus, by   \cite[Proposition~2.1(f)]{LS},  both 
$L^{\hskip -2pt\vee}  := \Ext^2_A(L, A) $ and  $L^{\hskip -2pt\vee\vee}  :=\Ext^2_A(\Ext^2_A(L, A), A)$ 
are CM over $A$  with GK-dimension 2. Moreover,    \cite[Proposition~2.1(d)]{LS} implies that 
  the natural map $\psi: L\to L^{\hskip -2pt\vee\vee} $ has $\GKdim(\Ker(\psi))\leq 1$ and 
  $\GKdim(\coker(\psi))\leq 0$. Hence
  $\Ker(\psi)=0$ and     $ L^{\hskip -2pt\vee \vee}$ also has multiplicity 1. By  \cite[Proposition~2.12]{LS}
 $L^{\hskip -2pt\vee\vee} $ is therefore a line module and, in particular, is cyclic.
  Since $L^{\hskip -2pt\vee\vee} $ only differs from $L$ by a finite dimensional module, and has a prime annihilator \cite[Theorem~6.3]{LS},  it follows that $L^{\hskip -2pt\vee\vee} $ is a $Q$-module, say $L^{\hskip -2pt\vee\vee} =Q/J'$. Moreover,  $J'\supseteq  JQ_{\geq m}$ for some $m$. Since neither  $Q/J'$ nor $Q/J$ has a non-zero,  finite dimensional submodule, it follows that $J=J'$.  
  By \cite[Proposition~7.2]{SV}  this also ensures that $J$ is MCM,  and so $J$ is the 
  right ideal we seek.

   Finally, the assertion that  $\End_{Q'}(J) \cong Q$   follows from Lemma~\ref{mo-lemma} and Remark~\ref{mo-remark}.  
   \end{proof}

\begin{lemma}\label{lem:translation2}
Keep the set-up from Lemma~\ref{lem:translation} and let $J$ be the $(Q',Q)$-bimodule constructed there. Then 
 $J=J_1Q=Q'J_1$.
\end{lemma}

\begin{proof}  For $J_Q$, the result follows from \cite[Lemma~5.6(2)]{RSSblowdown}, but a little more work is needed on the left.

Let $B := Q/Qg= B(E, \mc{N}, \sigma)$ and $B' := Q'/Q'g = B(E, \mc{N}',\sigma)$ for the appropriate   invertible sheaves 
$\mc{N}, \mc{N}'$ of  degree 4 on $E$.  
By \cite[Lemma~5.6(2) and Notation~2.2]{RSSblowdown}, $\overline{J}$ is a \emph{saturated} right $B$-module, in the sense that it has no finite dimensional extensions. 
Moreover it is a torsion-free $B$-module as $J$ is $g$-divisible.  As $g$ acts centrally,  
  $\overline{J}$ 
 is a therefore a torsion-free left $B'$-module. We first prove that $\overline{J}$  is also saturated on the left.
By \cite[(4.6.6) and Remark~5.8(4)]{Lev1992},  $X :=\Hom_{B'}(\Hom_{B'}(\overline{J},B'),B')$ is the    maximal essential extension
 of $\overline{J}$ by finite dimensional left $B'$-modules. Moreover,  $X/\overline{J}$ is finite dimensional, say with 
 $KX\subset \overline{J}$ for $K :=B'_{\geq m}$.  Clearly $X\subset Q_{\gr}(B')=Q_{\gr}(B)$ and  so if $y\in B$ then
 $KXy\subseteq \overline{J}y\subseteq \overline{J}$. Hence $Xy+\overline{J}$ is a finite dimensional extension of 
 $\overline{J}$  and so $Xy\subseteq X$. Thus $X$ is a right $B$-module which, as $\overline{J}_B$ is saturated, 
 implies that  $X=\overline{J}$; in other words, $\overline{J}$  is saturated as a left $B'$-module. Therefore, by
\cite[Theorem~1.3]{AV}, $\overline{J}=\bigoplus_{n\geq 0}H^0(E, \mc{N}'_n\otimes \mc{O}(\mathbf{p})^{\sigma^{n}})$,
 for some divisor $\mathbf{p}$.   Since $\overline{J}_Q$ is a line ideal, $\dim \overline{J}_1=\dim Q_1-2=2$,  and so 
$\deg \mathbf{p}=-2$,  by Riemann-Roch. Therefore, 
 \cite[Lemma~3.1]{R-Sklyanin} implies that $\overline{J}$  is generated in degree 1  as a left $B'$-module.
By the graded Nakayama lemma, $J$ is  also generated in degree 1 as  a left $Q'$-module.
\end{proof}

We immediately obtain:

\begin{corollary}\label{cor:translation}
Let $Q := \QVB(r)$ be a Van den Bergh quadric,
 and let $r' = r+n\sigma$ for some $n \in \ZZ$.
Then there is a graded rank one torsionfree right $Q$-module $M = M_{r',r}$ so that $Q':=   \QVB(r') \cong \End_Q(M)$,  while 
$\End_{Q'}(M)=Q$ and both ${}_{Q'}M$ and $M_Q$ are finitely generated.  
\end{corollary}
 
\begin{proof}   Take $M$ to be the appropriate product of the bimodules $J$ constructed by  Lemma~\ref{lem:translation}.
Again, Lemma~\ref{mo-lemma} and  Remark~\ref{mo-remark} ensures that $\End_{Q'}(M)=Q$.   \end{proof}

\begin{proposition}\label{prop:3.2}  Let $T:=\QVB^{(2)}$ be the 2-Veronese of a Van den Bergh quadric. If  $T^\circ$ is simple, then either $T^\circ$ is hereditary, or  $T^\circ$ has an $A_1$ singularity, as in Definition~\ref{sing-defn}.

If $T=\Skl^{(3)}$ or $T = \Skl'^{(4)}$, then $T^\circ$ is simple and  hereditary.
\end{proposition}

\begin{proof}    
By  \cite[Theorem~7.3]{ATV2} localised Sklyanin elliptic algebras are simple (this uses our standing hypothesis that $|\sigma| = \infty$).
  If $T = \Skl^{(3)}$ then $T^\circ$ is  hereditary  by \cite[Proposition~2.18]{Ajitabh}.
 The same proof   works for $T = \Skl'^{(4)}$.    
  If $T = \QVB^{(2)}$, then the result is clear by Lemma~\ref{lem:Smith-VdB}, 
 unless $T^\circ$ has infinite global dimension.  By Lemma~\ref{lem:Smith-VdB} this happens only if $T = Q^{(2)}$, where $Q = \QVB(\omega-\sigma)$ for some $\omega \in E_2$.

So, fix such a $Q := \QVB(\omega-\sigma)$, set $Q' := \QVB(\omega)$ and let $J$ be the right ideal of $Q$   constructed in Lemma~\ref{lem:translation}; thus $J$ is a $(Q',Q)$-bimodule.
Let $M := J^\circ$.
Since $\End_{(Q')^\circ}(M) = \End_{Q'}(J)^\circ = Q^\circ$ is simple by Lemma~\ref{lem:Smith-VdB}(3), 
$M$ is a projective left $(Q')^\circ$-module by the Dual Basis Lemma.
As $\gldim (Q')^\circ=2$ and $\gldim Q^\circ=\infty$, certainly $(Q')^\circ$ and $Q^\circ$ are not Morita equivalent. Thus ${}_{(Q')^\circ}M$ is not a generator; equivalently  $MM^*\not=(Q')^\circ$  and so  $Q'/JJ^*$ is infinite-dimensional.
We claim that  $Q'/JJ^* \cong \kk[g]$.

By Lemma~\ref{lem:translation2}, ${}_{Q'} J $ and $J_Q$ are generated by $J_1$, which has dimension 2.
In particular,  $J[1]$ 
 is in the set $\mb M$ of right MCM $Q$-modules defined in \cite[5.3]{SV}.  
Thus by  \cite[Lemma~7.1]{SV},  $ \dim (J^*)_0  =2$, say with basis $\{a,b\}$. 
If $Ja=Jb$, then $Jab^{-1}=J$ and so $ab^{-1}\in \End_{Q'}(J)_0=Q_0=\kk$, which is impossible. Thus $Ja\not=Jb$.  Since $J=Q'J_1$   it follows that      $\dim (J J^*)_1\geq 3$; equivalently  $\codim(JJ^*)_1\leq 1$.  We note also that $J^*_{<0}=0$ as $J$ is not cyclic.

By  \cite[Lemma~5.6S(2)]{RSSblowdown},  $\overline{J}$ is saturated and so, as in the proof of Lemma~\ref{lem:translation2},   we may write  
  $B := Q/Qg = B(E, \mc N, \sigma)$,   and $\overline{J} = \bigoplus_{n \geq 0} H^0(E, \mc{O}(\mathbf{q})\otimes\mc{N}_n)$,  for some 
   invertible sheaf $\mc{N}$ and divisor $\mathbf{q}$ with $\deg\mathbf{q}=-2$. By \cite[Lemma~6.14]{RSSlong}
\begin{equation}\label{transequ}
 \overline{J^*} \subseteq (\overline J)^* = \bigoplus_{n \geq 0} H^0(E, \mc{N}_n\otimes \mc{O}(-\mathbf{q})^{\sigma^n}).
 \end{equation}
By \cite[Lemma~5.6]{RSSblowdown} and \cite[Lemma~2.12]{RSSlong}, $J$ and hence $J^*$ are $g$-divisible.
Thus $ \dim J^*_0 = \dim \overline{J^*}_0 = 2$  by the last paragraph, while $\dim (\overline J)^*_0=2$ by \eqref{transequ} and Riemann-Roch. Thus,  $ \overline{J^*}_0 = (\overline J)^*_0$. By \cite[Lemma~3.1]{R-Sklyanin} and \eqref{transequ}, $(\overline{J})^* $ is generated in degree 
zero and so $ \overline{J^*} = (\overline J)^*$.
By \cite[Lemma~3.1]{R-Sklyanin}, again, $(\overline{J J^*})_n = B_n$ for $n \geq 2$.  
Since $(J J^*)_{n+2} \supseteq g (J J^*)_n$, it follows that
 $\codim(JJ^*)_{n+2} \leq \codim (JJ^*)_{n}$ for all $n \in \NN$.
 This codimension is bounded by 1 in degrees 0 and 1, so  $\codim (J J^* )_n \leq 1$ for all $n\geq 0$.  
Conversely, as  $\dim Q'/JJ^* = \infty$, clearly $(J J^*)_n \neq Q'_n$ holds for all $n\geq 0$.
 Thus $P = Q'/JJ^*$ is a point module, from which it follows that  $P$ is $g$-torsionfree since  $P^\circ\not=0$. Thus $P \cong \kk[g]$, as claimed.
   
  Since $P =  \End_Q(J)/JJ^* \cong \kk[g]$, the stable endomorphism ring of $M = J^\circ$ is  
 \[\underline{\End}_{Q^\circ}(M) = (Q')^\circ/MM^* = \kk.\] 
We are now ready to apply   Proposition~\ref{prop:Simon} to $A=Q^\circ$.  Bt the second paragraph of the proof,  $\gldim Q^\circ=\infty$ while   $(Q')^\circ = \End_{Q^\circ}(M)$  has $\gldim((Q')^\circ)=2$.  Moreover,  as $Q^\circ$ is simple, $M$ is automatically a generator. Finally,  $\id(Q^\circ)=1$ and    $M$ is a MCM $Q^\circ$-module by Lemma ~\ref{lem:2.1}. By  Proposition~\ref{prop:Simon} 
and the last display,  $Q^\circ $  therefore has an $A_1$ singularity.  
\end{proof}
 
We next note that these elliptic algebras never have line modules.

\begin{lemma}\label{lem:nolines}
Let $A$ be a cg noetherian algebra of finite global dimension such that $\hilb A = (1-t)^{-n}$ for some $n\geq 2$.
Then for $d >1$, the Veronese ring $A^{(d)}$ has no  linear modules other than point modules.
\end{lemma}

\begin{proof}
Let $L$ be a finitely generated graded $A^{(d)}$-module with Hilbert series $(1-t)^{-p}$.
Let $\wt{L} := L \otimes_{A^{(d)} }A$, and note that $L=\wt{L}^{(d)}$.
Let $P_{\bullet} \to \wt{L} \to 0$ be a finite graded free resolution of $\wt{L}$.
Thus, each $P_j \cong \bigoplus_{i=1}^{\ell_j} A[-a_{ij}]$ for some natural numbers 
 $\ell_j, a_{ij}$ and so  $\hilb \wt{L} = F(t)/(1-t)^n$ for 
 $F = \sum_j (-1)^j \sum_i t^{a_{ij}}$.
Since $p = \GK(\wt{L}) = \GK L$, we can cancel  common factors of $(1-t)$ from this equation  to obtain  
a polynomial  $G\in \mathbb{Z}[t]$,
with $G(1) \neq 0$, such that   $\hilb \wt{L} = G(t)/(1-t)^p$.
Rewriting $G$ as an integer polynomial in $1-t$ gives 
\[ \hilb \wt{L} = H(t) + \sum_{k=1}^p c_k/(1-t)^k,\]
where $H \in \ZZ[t]$ and the $c_k \in \ZZ$.
Thus there exists  $f(t)\in \mathbb{Q}[t]$ with leading term $c_p t^{p-1}/(p-1)!$ with 
$\dim \wt{L}_n = f(n)$ for $n \gg 0$. Thus $\dim L_n = \dim \wt{L}_{dn} = f(dn)$ 
 for $n \gg 0$. But  
 $\dim L_n = \binom{n+p-1}{p-1}$ by hypothesis.  Thus, taking leading terms and multiplying by $(p-1)!$  gives 
$c_p (dn)^{p-1}  = n^{p-1}$ for $ n \gg 0$.
As $d> 1$ and $c_p \in \ZZ$ this is only possible if $p=1$.
\end{proof}

\begin{corollary}\label{cor:nolines}
If  $T=\Skl^{(3)}$ or $T = \Skl'^{(4)}$ or $T=\QVB^{(2)}$, then $T$ has no line modules.
\end{corollary}
\begin{proof} 
For $T= \Skl^{(3)}$ the result is immediate from Lemma~\ref{lem:nolines}. For the other cases    use
 the fact that $\QVB$ (and hence  $T = \Skl'^{(2)}$)  is, by definition,  a factor of the Sklyanin algebra  $A=\SK$
 (see Example~\ref{eg:quadric}). 
   \end{proof}

 Putting together the results of  this section gives our main theorem.

\begin{theorem}\label{thm:3.6}  Let $T=\QVB^{(2)}$ or
$T=\Skl^{(3)} $
or $T = \Skl'^{(4)}$.  Then $T$ is a minimal  elliptic surface.   
 
Indeed,  suppose that $T\subseteq R\subset T_{(g)}$ is a finitely graded 
 overring such that    either  $R$ is noetherian or    $\widehat{R}$ is   finitely graded.  
   Then $R=T$.
\end{theorem}

\begin{proof}  Suppose first that $T^\circ $ is simple.  By  Proposition~\ref{prop:3.2} and  Corollary~\ref{cor:nolines}, $T \subseteq R$ satisfies   the hypotheses of   Theorem~\ref{thm:3.5} and so the   result follows from that theorem. 

Now  suppose that $T^\circ$ is not simple.  
Then Proposition~\ref{prop:3.2} and  Lemma~\ref{lem:Smith-VdB} imply that $T = Q^{(2)}$, where  
  $Q := Q_{VdB}(\omega + k\sigma)$ for
 some $\omega \in E_2$ and $k \in \NN$.
Let $\Tinf := (Q_{VdB}(\omega - \sigma))^{(2)}$. Let $M$ denote the 
   $((Q_{VdB}(\omega - \sigma),\, Q)$-bimodule  constructed by 
 Corollary~\ref{cor:translation}  and set $N=M^{(2)}.$ Clearly, $N $ is a $(\Tinf, T)$-bimodule  that is finitely generated on both side 
 and contained in $Q_{gr}(T)$. It follows from Lemma~\ref{mo-lemma} that $K=N^{**}$
   is a $(\Tinf, T)$-bimodule  contained in $ Q_{gr}(T)=Q_{gr}(\Tinf)$ 
   such that 
 $\End_T(K)=\Tinf$ and, conversely, $\End_{\Tinf}(K) = T$. Suppose that the theorem is false, say for the overring 
 $T\subsetneqq R\subset T_{(g)}$.
If $R$ is noetherian, then  Lemma~\ref{lem:1.7prime} implies that $\wh{R}$ is a finitely generated $R$-module and hence is finitely graded. Thus,  in all cases,  we can replace  $R$ by $\wh{R}$ and assume  that $R$  is finitely graded with $R =\widehat{R}$.  
   We aim for a contradiction.

By Lemma~\ref{mo-lemma},  $K^*=\Hom_T(K,T)= \Hom_{\Tinf}(K,\Tinf)$.   Proposition~\ref{prop:Simon}
Set $\mathfrak{m}=KK^*$; this  is an ideal of $\Tinf$ which,  since $\Tinf^\circ$ is simple, 
  contains  $g^n$ for some $n\geq 1$.  Let $U :=\Tinf + KRK^*$, which is certainly a finitely graded 
    ring with $\Tinf\subseteq U\subseteq (\Tinf)_{(g)} = T_{(g)}$.  
If  $v\in \widehat{U}$, then   $g^mv\in U$ for some $m\geq 1$ and so $g^m(K^* v K ) \subseteq K^*UK\subseteq R=\widehat{R}$. 
Thus $K^*vK\subseteq R $ and so $\mathfrak{m} v \mathfrak{m}=KK^*vKK^*\subseteq U$. Thus 
$g^{2n}\widehat{U}\subseteq \mathfrak{m} \widehat{U} \mathfrak{m}\subseteq  U$. Thus $\widehat{U}$ is finitely graded. 
By Theorem~\ref{thm:3.5} this forces $\wh{U}= \Tinf$ and so \[(K^*K)R(K^*K)\subseteq K^*\Tinf K = K^*K\subseteq T.\]
Therefore, $T$ and $R$ are equivalent orders which, by  \cite[Theorem~6.7]{R-Sklyanin}, implies that $T=R$.
 \end{proof}

 \begin{remark}\label{rem:notasgoodasIdlike}
 Suppose that  $T$ is one of the algebras in the Theorem, and that $T \subsetneqq R=\wh{R}  \subseteq T_{(g)}$ is a $g$-divisible graded overring.  Then we can further apply Theorem~\ref{thm:3.5} to conclude that 
 $\dim R_n = \infty$ for some $n\geq 0$.  If $T^\circ $ is simple one can even show that  $\dim R_0 = \infty$ and we conjecture that  this holds without the assumption of simplicity.
\end{remark}

The conclusion of Theorem~\ref{thm:3.6}  also extends to    Sklyanin algebras and Van den Bergh's quadrics themselves.

\begin{corollary}\label{cor:3.56} 
  Let $S=\Skl$ or $S = \Skl'$  or $S=\QVB$.  Then $S$ is a minimal   surface in the  sense that  if $S\subseteq U\subset S_{(g)}$ for some cg noetherian ring  $U$ then $S=U$. 
 \end{corollary}
 
\begin{proof}  Fix such a ring $U$ and pick  $d$  so that $T=S^{(d)}$ is an elliptic algebra.  Then
  $T\subseteq U^{(d)}\subset T_{(g)}$.  By \cite[Lemma~4.10]{AS} $U^{(d)}$ is noetherian and cg and so, by Theorem~\ref{thm:3.6}, 
  $U^{(d)} = T$.   
If $U\not=S$, pick $v\in U_{nd+r}\smallsetminus S_{nd+r}$ for some $n$ and $0<r<d$. Then 
$vS_1^{d-r}\subseteq U^{(d)}=S^{(d)}\subset S$. Since $S$ is generated in degree 1 it follows that 
$vS_{\geq d}\subseteq S$ and hence that $(vS+S)/S$ is finite dimensional. By the CM property for $S$ (see \cite[(4.6.6) and Remark~5.8(4)]{Lev1992}) this forces $v\in S$ and hence $U=S$. 
\end{proof}

%%%%%%%%%%%%%

\section{Overrings not contained in \texorpdfstring{$T_{(g)}$}{LG}}\label{section-overrings}

 Let $T$ be one of the algebras from Theorem~\ref{thm:3.6}, or indeed any minimal elliptic surface; thus by definition there is no noetherian cg ring $T\subsetneqq U\subset T_{(g)}$. There do, however, exist cg noetherian overrings of $T$ inside $Q_{gr}(T)$. Indeed, for any integer
  $n\geq 2$ one has
 $T\subseteq R:=\kk\ang{T_n g^{1-n}}$. This is of course a rather ``cheap'' counterexample since
 after a change of grading,  $
 \kk\ang{T_n g^{1-n}}\cong T^{(n)}$ under the homomorphism $xg^{1-n}\mapsto x$ for all $x\in T_n$. As we show in this section the 
 possible cg noetherian overrings of $T$ are tightly constrained, and are all quite similar to the ``cheap'' example above.  We do, however, 
 need the   technical assumption that the overring is a maximal order.

\begin{theorem}\label{thm:arbitraryoverring}
\begin{enumerate}
\item
 Let $T$ be a minimal elliptic surface and let $T \subseteq R \subset Q_{gr}(T)$ be a cg noetherian overring
  of $T$ that is a maximal 
order in $Q_{gr}(T)$. Then for each $n$ there is $\ell \geq n$ so that $R_n = T_\ell g^{n-\ell}$, and 
there exist  integers $M \geq N \geq 1$   so that $R^{(N)} = \kk \ang{T_M g^{N-M}} \cong T^{(M)}$.  
\item   Similarly, let $S=\QVB$ or $S=\Skl$ or $S = \Skl'$ and let $d = \deg g$.  
 Let $S \subseteq A \subset Q_{gr}(S)$ be a cg noetherian overring of $S$
  that is a maximal order in $Q_{gr}(S)$.
Then for all $n$, there is $\ell \in \ZZ$ so that $A_n = g^{\ell} S_{n-d\ell}$, and there are integers $N, M \geq 1$ so that $A^{(N)} \cong S^{(M)}$.
\end{enumerate}
   \end{theorem}
 
\begin{proof}
(1)
For $j \in \NN$, define 
\begin{equation}\label{arb-defn}
\phi_j: Q_{gr}(T) \to (Q_{gr}(T))^{(j+1)}\quad  \text{by $\phi_j(f) = f g^{jn}$ for all $f \in Q_{gr}(T)_n$.}\end{equation}  
Since $g$ is central, $\phi_j$ is a $\kk$-algebra homomorphism. 
Since $R$ is cg, and hence a finitely generated algebra, 
 there  exists $j\geq 0$ with $\phi_j(R) \subset T_{(g)}$.  
Let   $U:=\phi_j(R)$.
Define  $h:\mathbb{N}\to \mathbb{N}\cup\{\infty\}$ by 
\[\text{$h(n) :=   \max \{i : U_n \subseteq g^{i} T_{(g)} \}$, \  with  $ h(n) := \infty$ if $ U_n = 0 $.}\] 

We begin by following the proof of  \cite[Proposition~9.1]{RSSlong}, although as our notation is slightly different we give most of the details. 

Let   $U$ be generated in degrees $\leq r$, so 
$ U_n = \sum_{i=1}^r U_i U_{n-i}$ for all $n > r$.  As in   \cite{RSSlong}, the fact that $gT_{(g)}$ is a completely prime ideal 
implies that  $h(n) + h(m) \geq h(n+m)$ for all $n, m \in \NN$ and, moreover, that 
\beq\label{666} 
h(n) = \min \{ h(n-i) + h(i) : 1 \leq i \leq r\} \quad \mbox{for $n>r$,}
\eeq
with the obvious convention if any of these numbers equals $\infty$.
Now choose $1 \leq K \leq r$ so that $\lambda = h(K)/K$ is minimal.
Applying induction to \eqref{666}, we have $h(n) \geq \lambda n$ for all $n \in \NN$, and this forces 
\beq\label{h-mult} h(nK) = n h(K) \quad \mbox{ for all $n \in \NN$.}
\eeq

Since $U = \phi_j(R) \subseteq (T_{(g)})^{(j+1)}$ and $U_K \neq 0$, the number $N := K/(j+1)$ is an integer.
Let $D := h(K)$, and note that  $D < K$, since if $D=K$ then 
$U_K \subseteq g^K T_{(g)}$, and so $R_N \subseteq g^N T_{(g)}$, contradicting 
$T_N \subseteq R_N$.  Let $M = K-D$.
By \eqref{h-mult}, $U_{nK} \subseteq g^{nD} T_{(g)}$ for all $n \in \NN$.
Thus the function
$U_{nK} \to (T_{(g)})_{nM}$ given by 
$x \mapsto x g^{-nD}$ is well-defined, and  induces
 an injective vector space  homomorphism $\theta : U^{(K)} \to T_{(g)}^{(M)}$ with $\theta(U^{(K)})\not\subseteq \kk+gT_{(g)}$.  
It is routine to see that $\theta$ is an algebra homomorphism.  

Let $V:= \theta(U^{(K)})$ and  $Z := \wh{V[g]}\subset T_{(g)}$, recalling Notation~\ref{not:2.31}.   
Since $T_1 \subset R$, both $g^j T_1 $ and $g^{K-1} T_1$ are contained in $U$.  
Thus $g^{M-1} T_1 = \theta(g^{K-1}T_1) \subset V$, and in particular $g^{M }\in V$.
By Lemma~\ref{lem:1.7prime}, $Z$ and $Z^{(M)}$ are  finitely generated left and right $V$-modules and are thus cg noetherian.
Further, as $g^{M-1} T_1 \subseteq V$, we have  $Z \supseteq T$ which, as $T$ is a minimal elliptic surface, forces 
$T=Z$.  
Thus $V=\theta ( U^{(K)}) \subseteq T$ and so $U^{(K)} \subseteq T$.

   We claim that in fact $U\subseteq T$. To see this,    take $U_n\not=0$.  Since $R\supseteq T$, clearly $U_n\supseteq g^aT_b$ for some $a,b\geq 1$. Therefore,
 $g^{a(K-1)}U_n T_{b(K-1)} = U_n(g^aT_b)^{K-1}\subseteq U_{nK}\subseteq T$. As $T$ is $g$-divisible, it follows that 
 $U_n T_c\subseteq T$ for $c={b(K-1)}$.  Hence $U_nT_{\geq c} =U_nT_cT\subseteq T$ and so $(U_nT+T)/T$ is a finite dimensional  extension of $T$. As in the proof of Corollary~\ref{cor:3.56},  the CM property therefore forces  $U_n\subseteq T$. Hence $U\subseteq T$, as claimed.

It remains to get a detailed understanding of the graded pieces $U_n$ of $U$. To this end, define a graded subspace $W$ of $T$ by 
\[ W_n := \begin{cases} g^{h(n)} T_{n-h(n)} & \mbox{if }U_n \neq 0 \\
 0 & \mbox{otherwise.} 
 \end{cases} \]
 Let $n, m$ be  such that $W_n, W_m \neq 0$.  
 Using the equation $h(n+m) \leq h(n)+h(m)$ we have 
  \[ W_n W_m = g^{h(n)+h(m)} T_{n-h(n)}T_{m-h(m)}\subseteq  g^{h(n+m)}  T_{m+n-h(n+m)} = W_{n+m}.\]
 Thus $W$ is a ring.
 The  definition of $h$ and $g$-divisibility of $T$ force 
   $U \subseteq W$ and we may clearly extend $\theta$ to   an isomorphism 
   $\theta:  W^{(K)} \stackrel{\cong}{\to} T^{(M)}$.
  
  We next want to understand the structure of $W$ over its Veronese subalgebras. 
  By \cite[Lemma~4.10]{AS}, $U$ is a finitely generated left and right module over $U^{(K)}$.
For $1 \leq i \leq K-1$, let $U(i) = \bigoplus_n U_{nK+i}$.  This is finitely generated as a right $U^{(K)}$-module, say by $\bigoplus_{n=0}^{n_i} U_{nK+i}$.  Using \eqref{666}, it follows  for $n \geq n_i$ that 
\beq\label{eq:h}
\begin{split}
h(nK+i) & =\  \min_{n' \leq n_i} (h(n'K+i) + h((n-n')K)) \\ &\ =\  (n-n_i)h(K) +  \min_{n' \leq n_i} (h(n'K+i)+h((n_i-n')K))   
\ =\  (n-n_i) h(K) +  h(n_i K +i).
 \end{split}
\eeq
Consequently,  $W_{nK+i} = W_{n_i K+i} W_{(n-n_i)K} = W_{(n-n_i)K} W_{n_i K+i} $ for $1 \leq i \leq K-1$ and $n \geq n_i$, and so $W$ is a finitely generated   left and right $W^{(K)}$-module.  
 Now  $\theta (W^{(K)}) = T^{(M)}$ is a  finitely generated  module over $V=\theta(U^{(K)})$   on both sides, and so $W^{(K)}$ is a finitely generated  $U^{(K)}$-module  on both sides.
 Thus $W$ is  finitely generated as  a $U^{(K)}$-module and hence as  a $U$-module on both sides.
 Since $T \subseteq R$, it is clear that $Q_{\gr}(U) = Q_{\gr}(W) = (Q_{\gr}(T))^{(j+1)}$.   
 \footnote{As an aside for later use, we remark that so far we have used only that $R$ is  cg noetherian  and that $T$ is a minimal elliptic surface.} 
 As $R$ and therefore $U$ are maximal orders, $U=W$.
 It follows that $V = \theta(W^{(K)}) = T^{(M)}$.
  
 Finally, for $n \in \NN$, we have 
 \[R_n \ = \  \phi_j^{-1}(W_{(j+1)n}) \ = \   g^{ h((j+1)n)-nj} T_{(j+1)n-h((j+1)n)}.\]
Pick $m\in \NN$. Then,   as $(j+1) N = K$ and $h(Km)=h(K)m$, we have 
 \[ R_{Nm} \ = \  g^{h(Km)-Njm} T_{Km-h(Km)} \ = \    g^{(N-M)m} T_{Mm}\]
 and so $R^{(N)} = \kk \ang{g^{N-M} T_{M}} \cong T^{(M)}$.
 As $R_N \supseteq T_N$ we have $M\geq N$ and the result is proved.

 (2)
 Our notation in this part of the proof is that Veronese rings are not regraded;  so $(A^{(K)})_{Kn} = A_{Kn}$ for all $n\in \NN$
 and $g\in T_d=S_d$. A number of the steps of the proof exactly parallel those from Part (1), in which case the proof is left to the reader.

 For $j \in \NN$, extend $\phi_j$ to a map $\psi_j:  Q_{\gr}(S) \to Q_{\gr} (S^{(jd+1)})$ 
 by $\psi_j(f) = f g^{jn}$ for all $f \in Q_{\gr}(S)_n$.  As in Part~(1), each $\psi_j$ is an injective graded algebra homomorphism, and  $\psi_j(A) \subset S_{(g)}$ for some $j$.  
 Set   $X = \psi_j(A)$.  
  Define $h:  \NN \to \NN \cup \{ \infty\}$  by
  \[\text{$h(n) := \max \{ i: X_n \subseteq g^iS_{(g)} \}$, \  with  $ h(n) = \infty$ if $ X_n = 0 $.}\] 
   As in Part~(1),  $h(n+m) \leq h(n)+h(m)$ for all $n, m \in \NN$ and $h(n)$ satisfies the analogue of \eqref{666}.
  
Set $R:=A^{(d)}\supseteq T:=S^{(d)}$.   Then $R$ and  $U := \psi_j(A^{(d)}) \subset T_{(g)}$  have the same properties  as their counterparts in Part~(1) of the proof. 
 In particular,   there is an integer $K $, which we may take to be a multiple of $ d$, so
  that $h(nK) = nh(K)$ for all $n \in \NN$.
Further, there is a graded ring homomorphism $\theta:  U^{(K)} \to T_{(g)}$ with $\theta(x) = g^{-h(\deg x)} x$ for all homogeneous 
$x \in U^{(K)}$ and, moreover,   $V = \theta(U^{(K)}) \not \subseteq \kk + g T_{(g)}$.
Set $Z = \wh{V[g]}$.   Using the fact that $T$ is a minimal elliptic surface by Theorem~\ref{thm:3.6},  the arguments of Part~(1) show that  
$Z=T$  is a finitely generated $V$-module on both sides.  Moreover   $U \subseteq T$.
 
 We next claim that $X\subseteq S$. Certainly there exists $L$ such that $X^{(L)}=U^{(L)}\subseteq T^{(L)}$
  (in fact $L=(dj+1)d$ will work).  As $A\supseteq S$, if    $X_n\not=0$ then  $X_n \supseteq g^a S_b$,
   for some $a,b\geq 1$. Now the proof from Part (1) that $U\subseteq T$ can be used essentially unchanged to show that $X\subseteq S$.

 Now define $Y = \bigoplus_{n\geq 0} Y_n$ where 
 \[ Y_n := \begin{cases} g^{h(n)} S_{n-dh(n)} & \mbox{if }X_n \neq 0 \\
 0 & \mbox{else.} 
 \end{cases} \]
Clearly $Y_n \supseteq X_n$ for all $n$ by $g$-divisibility of $S$.  
Using the inequality $h(n+m) \leq h(n)+h(m)$ gives
\[Y_n Y_m = g^{h(n)+h(m)} S_{n+m-d(h(n)+h(m))}   \subseteq g^{h(n+m)} S_{n+m - dh(n+m)} = Y_{n+m},\]
and so  $Y$ is a $\kk$-algebra.
  
 Recall that $T=Z$ is a finitely generated module over $V= \theta(U^{(K)}). $
 As in Part~(1),  each $U(i) = \bigoplus_n U_{nK+i}$ is finitely generated as a right $U^{(K)}$-module, say by $\bigoplus_{n=0}^{n_i} U_{nK+i}$.  Equation  \eqref{eq:h} then follows formally and this now ensures that 
 $Y$ is a finitely generated  $Y^{(K)}$-module on both sides.  
It follows that $Y^{(K)}$ is a finitely generated  $U^{(K)}$-module on both sides 
 and thus $Y$ is a finitely generated  $X$-module on both sides.
 As in Part~(1), $Q_{\gr}(Y) = Q_{\gr} (X)$ and as $X \cong A$ is a maximal order, $Y=X$.   This proves the first assertion of Part (2). 
 The final  sentence follows as in Part~(1).
  \end{proof}

 One of the significant consequences of  \cite{RSSlong} is that graded maximal orders contained in  $T$ are automatically noetherian. 
We conjecture that in the main theorems of this paper,  Theorems~\ref{thm:3.5} and \ref{thm:arbitraryoverring} as well as 
 Corollary~\ref{cor:3.56},  the same is true for overrings. More precisely:

\begin{conjecture} \label{conj:mo=noetherian}  Theorem~\ref{thm:arbitraryoverring}(1) holds even if $R$, respectively $A$, is not assumed to be noetherian. In particular, and in the notation of the theorem, finitely graded  
 maximal orders $T\subseteq R\subset Q_{\gr}(T)$ are automatically noetherian. Similar comments hold  for overrings of $S$.
\end{conjecture}

We next give a couple of examples that show that one cannot easily improve on Theorem~\ref{thm:arbitraryoverring}.  
As usual, given a subset $V$ in a $k$-algebra $A$ we write $k\langle V\rangle $ for the $k$-algebra generated by $V$.
Following the discussion at the beginning of the section, and by  analogy with \cite[Proposition~9.1]{RSSlong}, one might 
hope that any cg  noetherian overring $T\subseteq U\subset Q_{\gr}(T)$ would have the form 
$U=\kk\langle T_{n+1}g^{-n}\rangle$ for some $n$. As the next example shows, this is not the case.

\begin{example}\label{eg:counter-overring1}  Let $T:=\Skl^{(3)}$, the 3-Veronese of a quadratic Sklyanin algebra and set
\[U:=\kk\langle T_1,g^{-1}T_3\rangle\ = \  \kk + T_1 + T_3g^{-1} + T_4 g^{-1} + T_6 g^{-2} + T_7 g^{-2} + ...\]
Then $U$ is a noetherian cg maximal order with $U\supsetneq T$ that cannot be 
written as  $U=\kk\langle T_{n+1}g^{-n}\rangle$ 
for any~$n$. However, up to a change of grading, $U^{(2)}\cong T^{(3)}$.
\end{example}

\begin{proof} 
Note that $\phi_1(U) = R = \kk\langle gT_1,gT_3\rangle$, in the notation from the proof of Theorem~\ref{thm:arbitraryoverring}.  
Since $g$ is central and 
$T$ is generated in degree one it is easy to see that $(gT_1)^2=g^2T_2\subseteq gT_3$ and hence that $R^{(4)}=\kk\langle gT_3\rangle \cong T^{(3)}$.
Since $T^{(3)}$ is noetherian, so  are $R^{(4)}$ and hence $R$ by \cite[Lemma~4.10(3)]{AS}. It is an easy exercise to see that $U$ is not generated by any set $T_{n+1}g^{-n}$.

It remains to prove that $R$ is a maximal order.  Throughout the proof we keep
 the grading from $T$; thus $R$ is generated in degrees 2 and 4, while $\deg(g)=1$.  
First, by \cite[Proposition~3.37]{Hipwood-thesis}, $R^{(4)}\cong T^{(3)}$ is a maximal order.  
So, suppose that $R\subseteq A$ is an equivalent order.  Then $A^{(4)}$ is equivalent to $R^{(4)}$ 
 by \cite[Lemma~3.32]{Hipwood-thesis}  
 and hence $A^{(4)}=R^{(4)}$. We next show that $A\subseteq T$.
Let $a\in A\smallsetminus T$ be homogeneous. 
Since $A^{(4)}=R^{(4)}$,  clearly $\deg(a)=4e+2$ for some $e\in \NN$. 
 Since $g^2\in R$ we have $g^2a\in R^{(4)}\subseteq \kk+gT$ and so $a\in g^{-1}T$. 
 As $a\not\in T$, then $a=g^{-1}x$ for some $x\in T\smallsetminus gT$. Notice that $g$ does not divide $x^2$ since $gT$ is a
  completely prime ideal.  Hence $a^2=g^{-2}x^2 \not\in T$; a contradiction. Thus $A\subseteq T$. 
   
   Now let $a \in A_{4e+2}$ for some $e \in \NN$.
 Write $a=g^uv$ with $u$ as large as possible. Writing $a$ as a sum of terms we may also assume that $v=v_1v_2$ with
 $ v_1 , v_2 \in T$ and $\deg v_1 = 1$.   
 Then $g^{2u} v^2 \in A_{8e+4} = R_{8e+4} = g^{2e+1} T_{6e+3}$.
 As $gT$ is completely prime, $g$ does not divide $v^2$ and so $2u \geq 2e+1$, hence $u > e$.
 Thus we can rewrite $a=g^{e+1}v_1w_2=(gv_1)(g^ew_2)$ where, now, $ w_2 \in T_{3e}$.
 Thus, $g^e w_2\in g^e T_{3e} = R_{4e}$ and $a\in R^{(4)}\langle gT_1\rangle = R$, as required.\end{proof}

As the next example shows, if one merely assumes that $U$ is a noetherian cg   overring of $T$ in
 Theorem~\ref{thm:arbitraryoverring},  then more complicated examples can arise.

\begin{example}\label{eg:counter-overring} Let $T:=\Skl^{(3)}$, the 3-Veronese of a quadratic Sklyanin algebra with factor $B:=T/gT$. 
Set $R:=(gT)^{(2)}+T^{(4)}$ and let $U:=\phi_1^{-1}(R)$, where $\phi_1$ is defined by \eqref{arb-defn}.
 Then $U$ and $R$ are noetherian cg   
 rings such that $\wh{R}^{(2)}=T^{(2)}$  is a finitely generated $R$-module on both sides. Similarly, 
$T\subseteq U\subseteq V:= \phi_1^{-1}(T^{(2)}) $, with $V$ a finitely generated $U$-module on both sides.
However, both 
 $\wh{R}^{(2)}/R$ and $V/U$ are infinite dimensional, so $U$  has noetherian overrings that are substantially larger than itself.
\end{example}
 
\begin{proof} Clearly $\phi_1(T)=\kk\langle gT_1\rangle\subseteq \kk+(gT)^{(2)}\subseteq R\subseteq T^{(2)}=\phi_1(V)$. Thus
 $T\subseteq U\subseteq V$
and all the assertions about $U$ follow from the corresponding assertions about $R$.

Since $g^2T^{(2)}\subset R\subseteq T^{(2)} = \wh{T}^{(2)}$, certainly $\wh{R}^{(2)}=T^{(2)}$. On the other hand as
 $T^{(4)}\subset R\subseteq T^{(2)}$ and $T^{(2)}$ is a (left and right)  noetherian  $T^{(4)}$-module, so is $R$. Hence 
 $R$ is a noetherian ring and $\wh{R}$ is a noetherian $R$-module on both sides.  It remains to prove that   $\dim_{\kk}T^{(2)}/R=\infty$. 
  Since $(gT)^{(2)}$ is an ideal of $R$, it suffices to prove that  $T^{(2)}/(gT)^{(2)} \cong B^{(2)}$ is not a finite dimensional extension 
  of $R/(gT)^{(2)}$. Since 
\[ R/(gT)^{(2)} = \frac{(gT)^{(2)}+T^{(4)} }{(gT)^{(2)} } \cong \frac{ T^{(4)} }{ (gT)^{(2)}\cap T^{(4)} } = \frac{ T^{(4)} }{ (gT)^{(4)} }= B^{(4)},\]
the assertion follows. \end{proof}

\begin{remark}\label{rem:further rings} 
Let $S := \Skl$ be a quadratic Sklyanin algebra, and let $T := S^{(3)}$.  
For the purposes of this remark, define a \emph{minimal model} to be  a cg noetherian algebra $T'$ containing $g$ and birational to $T$ with the property that if  $R'$ is cg noetherian  
with $T'\subsetneqq R'\subset T'_{(g)}$, then $T'=R'$ (thus we are not assuming that $T'$ is elliptic).
The ultimate aim in the present project is, of course,  to classify all the minimal models birational to $T$ (or $S$)
and then  to prove that any  finitely graded  
 maximal order $R$ birational to $S$ can be obtained from such a minimal model by blowing up
(including virtual blowing up as in \cite{RSSlong}) finitely many points on the elliptic curve~$E$.

  However, unlike the commutative situation, we expect there to be more minimal models than just the noncommutative projective plane and quadrics.   
 More precisely, reflexive right ideals $P$ of $S$ have been classified through formal moduli spaces \cite{NSt}, with a discrete invariant 
 $c(P )$ analogous to a second Chern class and a continuous one deforming a Hilbert scheme of points. By analogy with work on the
  Weyl algebra (see for example \cite{BW}), we hope  that 
$$  \End_S(P ) \cong \End_S(P')\iff c(P ) = c(P') \qquad\text{for reflexive right ideals}\  P, P'.$$
 Analogous  results should hold if $S$ is replaced by $T$ or $\QVB$ or their Veronese rings.
 
The  expectation is that  the corresponding endomorphism rings $\End(P )$ will then  give all minimal models birational to $S$.
For elliptic algebras a stronger conjecture will be given in Conjecture~\ref{final-conj}. \end{remark}

%%%%%%%%%%

  \section{General overrings in the locally hereditary case}\label{OVERRING}

  The arguments of  Section~\ref{MINIMAL}   can also be used to  obtain information on  the structure of  arbitrary cg noetherian 
  overrings of  non-minimal elliptic surfaces $T$ provided that  $T $ is {\em locally hereditary}\label{loc-her-ring}
    in the sense that $\gldim T^\circ=1$.  
  The main result of this section is an analogue of Theorem~\ref{ithm:B}; that is, the classical result that any birational morphism of smooth projective surfaces is a 
  composition of finitely many monoidal transformations. 
  
  We   recall that  for elliptic algebras $T$ there is a good analogue of Castelnuovo's theorem on contracting  rational curves of self-intersection $(-1)$.  In order 
  to state this we note that,  if  $T$ is an elliptic algebra such that $\rqgr T$ is smooth in the sense of Definition~\ref{defn:smooth},  then  there 
  is a well-defined intersection product  \cite{MS}, on $\rqgr T$.   
 This is given by  $(M \cdot N) := \sum_{i=0}^\infty (-1)^{i+1} \dim \Ext^i_{\rqgr T}(M, N)$, for $M, N \in \rgr T$.
  The noncommutative version of Castelnuovo's theorem is as follows.
  
  \begin{theorem}\label{thm:Castelnuovo}
  {\rm (\cite[Theorems~1.4, 1.5, and 8.1, Lemma~8.2]{RSSblowdown})}
  Let $T$ be an elliptic algebra so that $\rqgr T$ is smooth, and let $L$ be a line module so that $L$ has 
  self-intersection $(L\cdot L)=-1$.
  Then there is an elliptic algebra $\wt{T}$ with $T \subset \wt{T} \subset T_{(g)}$ and so that $\wt{T}/T \cong \bigoplus_{i \geq 1} L[-i]$ 
  as right $T$-modules.
  Further, $\wt{T}$ is the maximal submodule of $Q_{gr}(T)$ so that $\wt{T}/T$ is isomorphic to a direct sum of shifts of $L$, and $\rqgr \wt{T}$ is smooth.
  \end{theorem}
  
  We refer to the construction of $\wt{T}$ from $T$ given in Theorem~\ref{thm:Castelnuovo} as {\em blowing down} or {\em contracting} the line $L$.
  
We now state the main result of this section.  

  \begin{theorem}\label{thm:manchester}
  Let $T$ be an elliptic algebra of degree $\geq 3$ such that $T^\circ$  is  hereditary and let $T\subseteq R\subset T_{(g)}$ be 
  any $g$-divisible finitely graded   overring. Then $R$ is obtained from $T$ by successively blowing down finitely many line modules 
  $L$ of self-intersection $ (L\cdot L)= -1$.  In particular, $R$ is elliptic. 
  \end{theorem}
  
  We immediately note a simple corollary obtained by combining the theorem with Lemma~\ref{lem:1.7prime}:
    
  \begin{corollary}\label{cor:manchester}
   Let $T$ be an elliptic algebra of degree $\geq 3$ such that $T^\circ$  is  hereditary and let $T\subseteq R\subset T_{(g)}$ be any noetherian cg overring. Then there is  an  extension $R\subseteq R'\subset T_{(g)}$, finitely generated as a left and  
   right $R$-module, such that   $R'$ is obtained from $T$ by successively blowing down finitely many line 
   modules $L$ of self-intersection $(-1)$.\qed
  \end{corollary}

  The proof of Theorem~\ref{thm:manchester} will take the whole section and we assume throughout  that  the hypotheses 
  of the theorem are satisfied. To begin, we may assume that $T\not=R$ and, 
 by  Proposition~\ref{prop:linesonly},  pick a critical module $L$ so that some shift of $L$ is contained in $R/T$ and so that $\Ss:=L^\circ$ is a  simple submodule  of  $R^\circ/T^\circ$. Without loss of generality, we may shift $L$ so that $\min\{n: L_n\not=0\}=0$.   If  no such $L$ is  a   line module (by Proposition~\ref{prop:linesonly} this is equivalent to saying that $L$ has multiplicity $d(L)>1$), then the conditions of Hypothesis~\ref{hyp:2.0} are automatically satisfied. In this case   Proposition~\ref{prop:2.3} applies and leads to a contradiction. Thus, we can and will  assume that $L$ is a line module. 
 
 \begin{remark}\label{line-remark}  We note that by   \cite[Lemmas~5.2 and 5.4]{RSSblowdown},   $L$ is CM with $j(M)=1$.
 \end{remark}

The heart of the proof will be to prove the following fact. 

\begin{proposition}\label{prop:heart} In the situation above,  
   $(L\cdot L) =-1$.  
\end{proposition} 

Before proving the proposition, we will show that this quickly implies  the theorem.

\begin{proof}[Proof of Theorem~\ref{thm:manchester}] 
   Let $T \subset M \subseteq R$ so that $M/T \cong L[-i]$ for some $i$.

Applying  Theorem~\ref{thm:Castelnuovo}, we can blow
down $T$ at $L$ to obtain a second elliptic algebra $U := \wt{T} \supsetneqq T$ such that  $U/T\cong\bigoplus_{n\geq 1}L[-n]$. 
Therefore $U^\circ/T^\circ \cong\bigoplus_{n\geq 1}\Ss_n$ with $\Ss_n\cong L^\circ$ for all $n$. 
However, by \cite[Theorem~5]{Goodearl} every  overring  of $T^\circ$ is obtained by a torsion-theoretic localisation at some set $X$ of simple modules. 
In particular, the overring of $T^\circ $ generated by $M^\circ$ is such an overring and hence must equal $U^\circ$. 
Consequently, $U^\circ\subseteq R^\circ$.
Since $U$ is $g$-divisible, $U=\Phi(U^\circ) \subseteq \Phi(R^\circ)=R$.  

By \cite[Proposition~1.6]{Kuzmanovich},  $U^\circ$ is hereditary.   
Since  $\infty > \dim R_1\geq \dim U_1>\dim T_1$,   we may
now  induct on $\dim U_1$ to conclude that $R$ is obtained from $U$ (and hence $T$) by blowing down a finite number of line modules of self-intersection $(-1)$. This completes the proof of the theorem.
\end{proof}

It remains to prove Proposition~\ref{prop:heart}, for which we need several lemmas.  Note that, by 
\cite[Corollary~6.6 and Lemma~5.5]{RSSblowdown}, 
\begin{equation}\label{zero}
(L\cdot L)\ = -1 \ \iff \  \Ext^1_{T}(L,\, L)=0.
\end{equation}
  So the proof of  Proposition~\ref{prop:heart} amounts to describing this Ext group.  We note that Goodearl's result 
  also applies to the rings $T^\circ\subseteq R^\circ$ and implies, in particular, the following fact: Suppose that  
  $T^\circ\subseteq N^\circ\subseteq \Fr(T^\circ)$ is a finitely generated module extension such that $N^\circ/T^\circ$ 
  has a composition series  with factors consisting entirely of copies of $\Ss=L^\circ$. Then $N\subseteq R^\circ$. 
  
  Applying $\Phi$ this gives:
  
  \begin{lemma}\label{lem:zero.5}
Suppose that $P$  is a $T$-module with $T\subseteq P\subseteq T_{(g)}$  and assume that $P/T$ has a finite composition series with all factors isomorphic to shifts $L[r]$ of $L$. Then $P\subseteq R$.\qed\end{lemma}

We note also that the possibilities for $\Ext^1_T(L, L)$ are quite limited.

\begin{lemma}\label{lem:two}
If $T$ is any elliptic algebra and $L$ is any right 
$T$-line module, then as a $g$-module $\Ext^1_T(L,L)$ is isomorphic to one of:       
\[ 0, \quad \kk[g], \quad g^{-1} \kk[g], \quad g^{-1} \kk[g] \oplus \kk[g].\]
In particular, $\Ext^1_T(L,L) $ is $g$-torsionfree.
\end{lemma}

\begin{proof}
 Let $\overline{L}=L/Lg$, and note that,  by \cite[Lemma~4.7]{RSSblowdown},  there is an  exact sequence
\[ \xymatrix{0 \ar[r] & \Ext^1_T(L,L[-1])  \ar[r]^{\cdot g} & \Ext^1_T(L,L) \ar[r]^{\delta} & \Ext^1_{T/gT}(\bbar{L}, \bbar{L}).} \]
This immediately shows that  $\Ext^1_T(L,L)$ is $g$-torsionfree.
By \cite[Proposition~3.6(2)]{RSSblowdown}, $  \Ext^1_{T/gT}(\bbar{L}, \bbar{L}) = \kk\oplus \kk[1]$.  
 Thus, $\im \delta $ is a graded subspace of $ \kk \oplus \kk[1]$. 
  Since $\Ext^1_T(L,L)$ is left bounded and $g$-torsionfree, 
    this gives the four claimed possibilities.   \end{proof}

\begin{lemma}\label{lem:one}
$\Ext^1_T(L,L)_{-1} = 0$.
\end{lemma}

\begin{proof}
Suppose the contrary:  so
$  \Ext^1_T(L, L[-1])_0= \Ext^1_T(L,L)_{-1} \neq 0 $ and 
  there is a nonsplit extension
\[ 0 \to L[-1] \to M \to L \to 0.\]
Since $L$ is CM by Remark~\ref{line-remark},   
 this induces  an exact sequence
\[ 0 \to \Ext^1_T(L, T) \to \Ext^1_T(M, T) \to \Ext^1_T(L[-1], T) \to 0.\]
By Lemma~\ref{lem:exths} $\Ext^1_T(L,T)_1=\kk$, say given by the nonsplit extension 
$[Y]: 0 \to T \to Y \to L[-1] \to 0$.   By the displayed equation this lifts to a (necessarily nonsplit) extension,
 say
$0 \to T \to X \to M \to 0$.

We claim that $Y$ is   (Goldie) torsionfree. Indeed, if $Y$ has a nonzero  torsion submodule $N$,  then 
$N$ would be canonically isomorphic to a submodule of $L[-1]$ and hence $\Ext^1_T(L[-1]/N,T)\not=0$. As $L[-1]$ is 2-critical, 
$\GKdim(L[-1]/N)<2$ and so this contradicts the fact that $T$ is CM \cite[Proposition~2.4]{RSSshort}. This proves the claim;
in particular $Y^\circ$ is also torsionfree.

By Lemma~\ref{lem:two},  $\Ext^1_T(L, L)$   is $g$-torsionfree
and so  the extension $0 \to \Ss \to M^\circ \to \Ss \to 0$ is nonsplit.  
Thus by Lemma~\ref{general lemma}, $X^\circ $ is (Goldie) torsionfree.  
As $L$ is $g$-torsionfree, so are $M$ and $X$. Thus, if $X$ had nonzero Goldie torsion, then 
so  would $X[g^{-1}]$ and $X^\circ$, giving a contradiction. Therefore, $X$  must also be torsionfree.   
Finally,  by Lemma~\ref{lem:zero.5}, this implies that $X \subseteq R$. Since $\dim X_0 = 2$ but  $R_0=\kk$, this is a contradiction.
\end{proof}

Combining the last two lemmas gives:   

\begin{corollary}\label{cor:two}
Either $\Ext^1_T(L,L)=0$ or $\Ext^1_T(L,L) = \kk[g]$.
\qed
\end{corollary}

We now come to the key step in the proof of  Proposition~\ref{prop:heart}.

\begin{proposition}\label{prop:three}
In Corollary~\ref{cor:two}, suppose that $\Ext^1_T(L,L) \neq 0$.
Then for all $n \in \ZZ_{\geq 1}$, there is a graded CM $T$-module $L{(n)}$, where:  
\begin{enumerate}
\item $L{(1)}  = L$;
\item there is a nonsplit extension
$ 0 \to L{(n)} \to L{(n+1)} \to L \to 0;$
\item $L(n)^\circ$ is essential in $L(n+1)^\circ$.    
\end{enumerate}
\end{proposition}

\begin{proof}  
We will show by induction that $L{(n)}$ exists with the claimed properties and    that $\Ext^1_T(L, L{(n)}) \cong  \Hom_T(L, L{(n)}) \cong \kk[g]$.   The base case $n =1$ holds by Remark~\ref{line-remark} and Corollary~\ref{cor:two}.  

Suppose that $N=L{(n)}$ has been constructed with the given properties.   Since  $\Ext^1_T(L, N) =\kk[g]$, 
  there is a nonsplit exact sequence
$ 0 \to N \to M \to L \to 0$
and so  certainly $M = L{(n+1)}$ exists, and is CM by induction.
By Lemma~\ref{lem:two}, $\Ext^1_T(L, N)$ is $g$-torsionfree, so  
$M^\circ\not\cong N^\circ \oplus \Ss$.
As $L^\circ=\Ss$ is simple,  $N^\circ$ is essential in $M^\circ$.  

Consider the localised exact sequence  
\[\begin{matrix}
&{ \xymatrix{0 \ar[r]& \Hom_{T^\circ}(\Ss, N^\circ) \ar[r]& \Hom_{T^\circ}(\Ss, M^\circ) \ar[r]^{\alpha} & \Hom_{T^\circ}(\Ss, \Ss) \ar[r]^{\ \quad \quad \beta} &
}} \\ &\qquad 
{  \xymatrix{\ar[r]& \Ext^1_{T^\circ}(\Ss, N^\circ) \ar[r] & \Ext^1_{T^\circ}(\Ss, M^\circ) \ar[r] & \Ext^1_{T^\circ}(\Ss, \Ss) \ar[r] & 0.}}
\end{matrix}\]
As noted above, $M^\circ$ is   a nonsplit extension, 
and so \[\Hom_{T^\circ}(\Ss, M^\circ) \cong \Hom_{T^\circ}(\Ss, N^\circ) = \kk.\]  Thus, $\alpha=0$.
Thus $\beta$ is injective, and since $\Ext^1_{T^\circ}(\Ss, N^\circ) = \kk$ by induction,   $\beta$ is even an isomorphism.
It follows that $\Ext^1_{T^\circ}(\Ss, M^\circ)\cong \Ext^1_{T^\circ}(\Ss, \Ss) = \kk$.

Consider now the long exact sequence
\[ \begin{matrix}
&{\xymatrix{ 0 \ar[r] & \Hom_{T}(L, N) \ar[r] & \Hom_{T}(L, M)\ar[r]^{\wh{\alpha}} & \Hom_{T}(L,L) \ar[r]^{\quad\quad \wh{\beta}}&}} \\
&{\xymatrix{ \ar[r]& \Ext^1_{T}(L,N) \ar[r]& \Ext^1_{T}(L,M)\ar[r] & \Ext^1_{T}(L,L)  }}\end{matrix}\]
 Since $\alpha = 0$, the image of $\wh{\alpha}$ must be a $g$-torsion submodule of $\Ext^1_T(L,L) \cong \kk[g]$,
 and so $\wh{\alpha}  =0$.
 Thus $\Hom_T(L,M) \cong \Hom_T(L,N) \cong \kk[g]$ by induction, again.
Further,  $\wh{\beta}$ is an injective map between the graded spaces $\Hom_T(L,L) \cong \kk[g]$ and $\Ext^1_T(L,N) \cong \kk[g]$ and so is an isomorphism.
 Thus $\Ext^1_T(L,M)\hookrightarrow \Ext^1_T(L,L) \cong \kk[g]$ and so $\Ext^1_T(L, M) \cong g^m \kk[g]$ for some $m \geq 0$.
 
 Finally, write $\bbar{L}:=L/Lg$ and  consider   the exact sequence
 \[   \Ext^1_T(L,M)[-1]\  \buildrel{\cdot g}\over{\longrightarrow}\
  \Ext^1_T(L,M) \  \buildrel{\gamma}\over{\longrightarrow}\
 \Ext^1_{T/gT}(\bbar{L}, \bbar{M}),\]
 given by \cite[Lemma~4.7]{RSSblowdown}.
 Since $\bbar{M}$ has a filtration whose factors are $n+1$ copies of the point module $\bbar{L}$, 
 \cite[Proposition~3.6(2)]{RSSblowdown} implies 
 that $  \im \gamma \subseteq (\kk[1]\oplus \kk)^{(n+1)}$.  
Thus, since $\Ext^1_T(L,M)$ is  left bounded and  $g$-torsionfree the map $\cdot g$ is an injection and the graded version of Nakayama's Lemma implies that generators of $\text{Im}(\gamma)$ pull back to generators of  $\Ext^1_T(L,M)$. 
Hence   $m = 0$, completing the proof of the induction step.
\end{proof}

\begin{proof}[Proof of Proposition~\ref{prop:heart}]
Suppose that $\Ext^1_T(L,L) \neq 0$.  Then modules $L{(n)}$ as in Proposition~\ref{prop:three} exist for all $n$.
We claim that there is a torsionfree extension 
$0 \to T \to X{(n)} \to L{(n)}[-1] \to 0$
for all $n$.

As in the proof of Lemma~\ref{lem:one}, certainly $X(1)$ exists. By induction, assume that $X(n)$ exists and let $\xi_n $ be the corresponding element of $ \Ext^1_T(L(n), T)$.
As $L(n)$ is CM we have  the exact sequence \[0 \to \Ext^1_T(L, T) \to \Ext^1_T(L(n+1),T) \to \Ext^1_T(L(n),T) \to 0. \] Let $\xi_{n+1}$ be a preimage of $\xi_n$ in $\Ext^1_T(L(n+1),T)$ and let $0 \to T \to X(n+1) \to L(n+1) \to 0$ be the corresponding extension.  
Under localisation this gives the extension $0 \to T^\circ \to X(n+1)^\circ \to L(n+1)^\circ \to 0$.
By Lemma~\ref{general lemma} and Proposition~\ref{prop:three}(3), $X(n+1)^\circ$ is torsionfree, and as in the final paragraph of the proof of Lemma~\ref{lem:one} $X(n+1)$ must be torsionfree, as claimed.

 Thus  $X{(n)} \subseteq R$ for all $n$, by Lemma~\ref{lem:zero.5}.  
But $\dim X{(n)}_1 \geq n$,   whence $\dim R_1 = \infty$. This gives the required  contradiction and 
 completes the proof of the proposition and hence that of Theorem~\ref{thm:manchester}.
\end{proof}

We also note that Theorem~\ref{thm:manchester}  provides the following variant of Theorem~\ref{thm:3.5} for elliptic algebras of degree at least $3$.

\begin{theorem}  \label{thm111}
Let $T$ be an elliptic algebra of degree $\geq 3$ such that $T^\circ$    is hereditary and $T$ has no line modules of self-intersection $(-1)$. 
Then $T$ is a minimal elliptic surface. \end{theorem}

\begin{proof}  Suppose there exists a cg noetherian ring $R$ with  $T\subsetneq R\subset T_{(g)}$. By 
 Lemma~\ref{lem:1.7prime} we can replace $R$ by $\wh{R}$ and assume that $R$ is also $g$-divisible. Now Theorem~\ref{thm:manchester} implies that $R=T$.
 \end{proof}

We end by making the following conjecture.     This may be compared with Remark~\ref{rem:further rings} where the algebras are not assumed to be elliptic.

\begin{conjecture}\label{final-conj}
If $T, T'$ are two elliptic algebras with $Q_{gr}(T) = Q_{gr}(T')$, then they are related, up to isomorphism, by a finite series of blowdowns and blowups.  \end{conjecture}

  %%%%%%%%%%%%%%%%% 
\section{Gelfand-Kirillov dimension}\label{GKDIM}

In this final  section, we consider the Gelfand-Kirillov (GK) dimension of overrings of elliptic algebras.  See  
\cite{KL} for the basic theory of GK-dimension.   
Let $T$ be a minimal elliptic surface; by definition $\GKdim T  = 3$.  
Theorem~\ref{thm:3.6} and Corollary~\ref{cor:3.56} can be viewed as saying that any noetherian graded algebra  lying strictly between
 $T$ and $Q_{gr}(T)$ must be significantly larger than $T$.  
In this section we prove that this is true in the sense of GK-dimension as well:   any proper noetherian graded   
  overring of $T$ contained in $T_{(g)}$ has GK-dimension $\geq 4$.  
We also show that, when  $T=\Skl^{(3)}$ or $T= \Skl'^{(4)}$ or $T=\QVB^{(2)}$,   any proper overring of 
$T^\circ$ has GK-dimension $\geq 3$.

We first give some elementary computations on linear systems on elliptic curves.
Let $E$ be an elliptic curve and let $K = \kk(E)$.  If $f \in K$, we write 
\beq\label{div} (f) = (f)_0 - (f)_\infty \eeq 
for the divisor of $f$, where both $(f)_0$ and $ (f)_\infty$ are effective and of minimal degree.
If $D$ is a divisor on $E$, we write $\abs{D} = H^0(E, \sO(D)) $, which we 
 identify with $\{ f \in K : (f) + D \geq 0\}$. 

If  $y \in K$  then 
\beq\label{note}
y \abs{D} = \{ f \in K : y^{-1}f \in \abs{D}\} = \{ f : (f)  + D \geq (y)\} = \abs{D-(y)} \subseteq \abs{D+(y)_\infty}.
\eeq

We need the following elementary lemmas.

 \begin{lemma}\label{lem:GK3}
 Let $D, D'$ be divisors on $E$ such that $\deg \inf\{D, D'\} >0$.  Then $\abs{D}+ \abs{D'} = \abs{ \sup \{D, D'\}}$.
 \end{lemma}
 \begin{proof}
 Certainly $\abs{D}, \abs{D'} \subseteq  \abs{ \sup \{D, D'\}}$.  For the other inclusion, we count dimensions.
 We have:
\[
 \abs{D} \cap \abs{D'} = \{ f \in K : (f) \geq -D, -D'\} = 
  \{ f : (f) \geq \sup\{ -D, -D'\} \}
  = \abs{\inf\{D, D'\}}.
  \]
 So
 \begin{align*}
 \dim (\abs{D} + \abs{D'}) & = \dim\abs{D} + \dim \abs{D'} - \dim(\abs{D} \cap \abs{D'}) \\
 & = \deg D + \deg D' -\deg \inf\{ D, D'\} \quad \mbox{by Riemann-Roch and  hypothesis on $\inf\{D, D'\}$} \\
 & = \deg \sup \{D, D'\}  = \dim \abs{\sup\{D, D'\}}.
 \end{align*}
The lemma follows.
 \end{proof}

\begin{lemma}\label{lem:GK1}
Let $x,y \in K \ssm \kk$  and let $a = \deg (x)_0 = \deg(x)_\infty$ and $b = \deg (y)_0 = \deg (y)_\infty$.
If $(x)$ and $(y)$ have disjoint supports and  $\deg D> a+b$ then
\[ x \abs{D} + y\abs{D} = \abs{D + (x)_\infty + (y)_\infty}.\]
\end{lemma}
\begin{proof}
By \eqref{note} we have $x \abs{D} = \abs{D - (x)}$, $y\abs{D} = \abs{D-(y)}$.
We have $\inf \{ D -(x), D-(y) \} = D -(x)_0 -(y)_0$ by hypothesis and so
$\deg \inf \{ D -(x), D-(y) \} = \deg D - a-b > 0$.
Thus by Lemma~\ref{lem:GK3} we have
\[
x \abs{D} + y \abs{D} 
 = \abs{D - (x)} + \abs{D-(y)} 
 = \abs{\sup\{ D-(x), D-(y)\}}  
= \abs{D + (x)_\infty  + (y)_\infty},
\]
again using our assumption on the supports of $(x)$ and $(y)$.
\end{proof}

In the next result, let  $\sigma$ be an infinite order translation automorphism of $E$.  
We work inside the Ore extension $K[t; \sigma]$.  

\begin{lemma}\label{lem:GK2}
Let $y \in K \ssm \kk$ be such that all points in the support of $(y)$ have disjoint $\sigma$-orbits, and let $d := \deg(y)_0$.  
Let $D$ be a divisor with $\deg D > 2d$.  
Let $V := \abs{D} \cdot t$ and let $W := \kk + \kk  y + V$.
If $n \geq m \geq 1$, then 
$W^n \cap K t^m = \abs{F(n,m)} t^m$,
where
\[ F(n,m) := 
D + \sigma^{-1}D + \dots + \sigma^{-m+1} D+ (n-m)\Bigl( (y)_\infty + (y^\sigma)_\infty + \dots + (y^{\sigma^m})_\infty \Bigr).\]
\end{lemma}
\begin{proof}
Certainly  $W^n \cap K t^m $ is spanned by 
\[\{ y^{i_0} V (y^\sigma)^{i_1} V^\sigma \cdots V^{\sigma^{m-1}} (y^{\sigma^m})^{i_m} t^m : \sum i_j \leq n-m, i_j \geq 0\ \forall j \}.\]
Each of these is contained $\abs{F(n,m)} t^m$, giving one inclusion. 

We show the other inclusion by induction on $n \geq m$.  
Certainly $W^m \cap K t^m = V^m = \abs{ F(m,m)} t^m$, 
 so assume that  $n> m$.   By induction,
$ W^{n-1} \cap K t^m = \abs{F} t^m$, where  $F= F(n-1,m)$.

Choose $0 \leq j \leq m-1$.  Then
\begin{align*}
W^n \cap K t^m & \supseteq( y^{\sigma^j} W^{n-1} + y^{\sigma^{j+1}}W^{n-1})\cap K t^m 
=  (y^{\sigma^j} \abs{F} + y^{\sigma^{j+1}} \abs{F}) t^m \\
 & = \abs{F+(y^{\sigma^j})_\infty +(y^{\sigma^{j+1}})_\infty} t^m, \quad \text{by Lemma~\ref{lem:GK1}}, 
 \end{align*}
 where we have used that $\deg F \geq \deg D > 2d$.
 
 So 
 \begin{align*} W^n \cap K t^m & \supseteq \sum_{j=0}^{m-1} \abs{F+(y^{\sigma^j})_\infty +(y^{\sigma^{j+1}})_\infty} t^m 
 = \abs{\sup_j F+(y^{\sigma^j})_\infty +(y^{\sigma^{j+1}})_\infty} t^m \quad \text{ using Lemma~\ref{lem:GK3}}\\
 & = \abs{ F + \sum_{j=0}^m (y^{\sigma^j})_\infty}t^m \quad \text{ by assumption on $y$}\\
& = \abs{F(n,m)} t^m, 
 \end{align*}
 as needed.  
 \end{proof}
 
We now give  a result on the GK-dimension of overrings of TCRs of  elliptic curves.

\begin{proposition}\label{prop:GK4}
Let $B := B(E, \sL, \sigma)$ where $E$ is an elliptic curve, $\sigma$ has infinite order, and $\sL$ is an ample invertible sheaf on $E$,  and 
let $B \subsetneqq C \subseteq Q_{\gr}(B)$ where $C$ is $\mathbb{Z}$-graded with $\dim C_0 > 1$.
Then $\GKdim C \geq 3$.
\end{proposition}

\begin{proof}
We write $K := \kk(E)$ and $C = \bigoplus C_i t^i \subseteq K[t, t^{-1}; \sigma]$.  
By hypothesis there exists $y \in C_0 \ssm \kk$.
Let  $d := \deg (y)_0 > 0$.

Choose a positive integer $c$ so that all points in the support of $(y)$ have disjoint $\sigma^c$-orbits and so that $\deg \sL_c > 2 d$.
Let $C' := (B \ang{y})^{(c)}$.   
It suffices to prove that $\GKdim C' \geq 3$.

We note that $C'= \kk\ang{W}$ where $W := \kk + \kk y + H^0(E, \sL_c)t^c$.
By Lemma~\ref{lem:GK2} we have
\begin{align*}
\dim W^n & \geq \sum_{m=1}^n \dim (W^n \cap Kt^{mc}) 
\geq \sum_{m=1}^n (n-m)m = n \sum_{m=1}^n m - \sum_{m=1}^n m^2 \\
& = n \binom{n+1}{2} - \frac{n(n+1)(2n+1)}{6} = \frac{n^3}{6} + O(n^2),
\end{align*}
and so $\GKdim C' \geq 3$.
(A little more work will show that $\GKdim C' = 3$ but this is all that is required.)
\end{proof}

Applying this to minimal elliptic surfaces,  we obtain:

\begin{corollary}\label{cor:GK5}
Let $T$ be a minimal elliptic surface and let $T \subsetneqq R \subset T_{(g)}$ be a noetherian graded overring of $T$.
Then $\GKdim R \geq 4$.  In particular, by Theorem~\ref{thm:3.6}, this holds for
$T=\QVB^{(2)}$ or
$ \Skl^{(3)} $
or $  \Skl'^{(4)}$.  
\end{corollary}

\begin{proof}   
By definition, $R$   cannot be cg.   
Note that one cannot have $R_{\leq 0}=\kk$, as this would contradict   Remark~\ref{rem:2.33}(2). 
Thus  Remark~\ref{rem:2.33}(1) implies that $R_0\supsetneq \kk$.

So, suppose that   ${R}_0\ni x_1^{-1}x_2g^{m}$, for some $m\geq 0$ and  $x_i\in T\smallsetminus gT$. Then
$\wh{R}_{\leq 0}\ni z=x_1^{-1}x_2$ with $z\in T_{(g)}\smallsetminus gT_{(g)}$. Thus $\wh{R}+gT_{(g)}\supseteq T\langle z\rangle$ and hence
 $\overline{\wh{R}}\supseteq \overline{T}\langle \overline z\rangle$.  Therefore,  by Proposition~\ref{prop:GK4},  $\GKdim \wh{R}/g\wh{R}\geq 3$.

Note that $R^\circ=(\wh{R})^\circ$. Therefore, by \cite[Lemmas~2.1 and 2.2]{RSS},    there is a   filtered isomorphism 
$\theta{\hskip -2pt}:{\hskip -0.5pt} R^\circ  \buildrel{\sim}\over{\longrightarrow} {R}/(g-1){R}$, with $\gr(R^\circ) = \wh{R}/g\wh{R}$.  Thus,  
 $\GKdim \gr(R^\circ)  \geq 3$.
     By \cite[Lemma~6.5]{KL}, $\GKdim R^\circ \geq \GKdim \gr(R^\circ) $. Therefore, by \cite[Proposition~3.15]{KL}, 
     $\GKdim R \geq \GKdim R^\circ + 1$. 
     Putting this together gives $\GKdim R \geq 4$.
\end{proof}

 We also have:   

\begin{theorem}\label{thm:GK6}
Let $T = \Skl^{(3)}$ or $T= (\Skl')^{(4)}$ or $T = \QVB^{(2)}$. 
If  $A$ is an algebra with $T^\circ \subsetneqq A \subseteq Q(T^\circ)$ 
then $\GKdim A \geq 3 = \GKdim T^\circ + 1$.   
\end{theorem}

\begin{proof}
We first establish the result if $T^\circ$ is simple.  
Using Notation~\ref{not:2.31} and \cite[p.2099]{RSSlong}, $A=(\Phi A)^\circ$ and so    $T \subsetneqq \Phi A \subseteq T_{(g)}$.  
Since $\Phi A=\wh{\Phi A}$,  Remark~\ref{rem:notasgoodasIdlike} implies that    $\dim (\overline{\Phi A})_0 > 1$.
Thus by Proposition~\ref{prop:GK4}, $\GKdim \overline{\Phi A} \geq 3$.
 
 Now suppose that $\GKdim A:=\alpha<3$.  Since $\Phi(A) \subseteq A[g,g^{-1}]$, it follows from \cite[Lemma~3.1 and Proposition~3.5]{KL} that 
  $\beta:=\GKdim \Phi(A)\leq \GKdim A[g,g^{-1}]\leq \alpha+1$. Now, as $gT_{(g)}\cap \Phi(A) = g \Phi(A)$ is a nonzero ideal of the domain 
  $\Phi A$, it follows from  \cite[Proposition~3.15]{KL} that $\GKdim \overline{\Phi A}\leq \beta-1\leq \alpha<3$. This contradicts the 
  first paragraph of the proof. 
  
  Now suppose that $T^\circ$ is not simple.  As in the proof of Theorem~\ref{thm:3.6}, there are a simple elliptic algebra
  $\Tinf$ and a 
  $(\Tinf, T)$-bimodule $K$, which is finitely generated on both sides by Corollary~\ref{cor:translation}, so that 
 $\End_T(K)=\Tinf$ and $\End_{\Tinf}(K) = T$.
 Set $\Kk := K^\circ$; thus  $\Ll:=\Kk^* \Kk \subset T^\circ$ and $\Kk \Kk^* =\Tinf^\circ$ as $\Tinf^\circ$ is simple.
 Let $C := \Tinf^\circ + \Kk A \Kk^* $. If $C= \Tinf^\circ$, then $\Ll A\Ll \subseteq  \Kk^*\Tinf^\circ\Kk\subseteq T^\circ$.  
 As $T^\circ$ is a maximal order by Remark~\ref{mo-remark},   
  this implies that $A\subseteq T^\circ$, a contradiction. Thus
  $C\supsetneqq \Tinf^\circ$ and so, by      the first part of the proof, $\GKdim C  \geq 3$.
   
 Let $A' := T^\circ + \Kk^* C \Kk \subseteq A$.  
 Now,
$\Kk A' = \Kk + \Kk\Kk^* C \Kk = C \Kk, $
 and so 
   $C \Kk$ is a $(C, A')$-bimodule which is finitely generated, and clearly torsionfree, on both sides. 
  Thus by \cite[Corollary~5.4]{KL}, 
  \[ 3\leq \GKdim C = \GKdim {}_C (C \Kk) = \GKdim (C \Kk)_{A'} = \GKdim A' \leq \GKdim A,\]
  giving the result.
     \end{proof}
     
     We conjecture that Theorem~\ref{thm:GK6} holds if $T$ is any minimal elliptic surface.  
  
  \begin{remark}\label{rem:ML}
  Theorem~\ref{thm:GK6} is reminiscent of a striking result of Makar-Limanov  \cite{ML}  on the localised Weyl algebra.  
  If $\chrr \kk = 0$ and $A := \kk(x)[\partial,\partial^{-1}]$, for $\partial:=\frac{d}{dx}$ he shows  that if $B$ is a ring with 
  $A \subsetneqq  B \subseteq \Fr(A)$ then $\GKdim B > \GKdim A$; in fact $\GKdim B = \infty$.
  \end{remark}

  %%%%%%%%%%%%

  %%%%%%%%%%%%%%%%%
 
 \appendix
\section{Commutative algebras}\label{commutative}
 
We end this paper with a few comments on the commutative analogues of the results in this paper; thereby justifying some of the comments from the introduction by noting that elliptic algebras are noncommutative versions of anticanonical (homogeneous coordinate) rings of del Pezzo surfaces, and by exploring some of the properties of these rings. This result will not be used in the body of the paper, so can be skipped on first reading.

%A {\em del Pezzo surface} is a smooth projective surface $X$ so that $\omega_X^{-1}$ is ample.

 \begin{lemma}\label{lem:dP}
 Let $T$ be a cg commutative domain that is generated by $T_1$ and so that there is $g \in T_1$ with $T/gT$  isomorphic to the homogeneous coordinate ring $B := B(E, \mc L)$ of an elliptic curve $E$ with respect to an ample line bundle $\mc L$ on $E$.
 Suppose also that $X := \Proj T$ is nonsingular.
 Then $X$ is a del Pezzo surface of degree $\geq 3$, and $T$ is isomorphic to the anticanonical coordinate ring of $X$.
 \end{lemma}
 
 \begin{proof}  We need to prove that $\omega_X^{-1}$ is ample..
 We use $C \cdot C'$ to denote the intersection product on $X$.
 From the setup, $g$ defines the elliptic curve $E \subset X$, and $\mc O_X(E) \cong \mc O_X(1)$ with $\mc L \cong \mc O_X(E) |_E$.  
 Since $B(E, \mc L) $ is generated in degree 1, $d: = \deg \mc L = E^2$, where $E^2:= E\cdot E \geq 3$.
 Letting $K = K_X$ be the canonical divisor on $X$, by adjunction  \cite[Proposition~V.1.5]{Ha} we have 
 \beq \label{adjunct}
 0 =  E \cdot (E+K).
 \eeq
 
  For a sheaf $\mc M$ on $X$, let $h^i(X, \mc M) = \dim H^i(X, \mc M)$.
 Then we have
 \begin{align*}
 1 + d \frac{n(n+1)}{2} =& \sum_{k=0}^n \dim B_k = \dim T_n \qquad \qquad\qquad \mbox{as $B = T/gT$} \\
 =& h^0(X, \mc O_X(nE)) = \chi(\mc O_X(nE)) \qquad \mbox{for $n \gg 0$} \\
   =& \chi(\mc O_X) + \frac{1}{2} nE \cdot (nE-K) \qquad\qquad \mbox{by Riemann-Roch} \\
  =& \chi(\mc O_X) + \frac{n(n+1)}{2} E^2 \qquad\qquad\qquad \mbox{by \eqref{adjunct}.}
  \end{align*}
  Since $E^2 = d$, $\chi(\mc O_X) = 1$.
  
  Further, note that %for $n \geq 1$, 
  $K$ is not effective, as $E$ is ample and 
  $K \cdot E = - E^2 < 0$.
  In particular, $h^2(X, \mc O_X) = h^0(X, K) = 0$, and as $1 = \chi(\mc O_X) = h^0(X, \mc O_X)$, we have $h^1(X, \mc O_X) = 0$.
  
  Consider the exact sequence
  \[
  \xymatrix@R=5pt{
  H^1(X,\mc O_X) \ar[r] & H^1(X, \mc O_E) \ar[r]   & H^2(X, \mc O_X(-E)) \ar[r] & H^2(X, \mc O_X) .    }
 \]
Since the outside terms are zero, it follows that  that $\kk \cong H^2(X, \mc O_X(-E)) \cong H^0(X, \mc O_X(K+E))$.
Thus $K+E$ is effective.
By \eqref{adjunct} and the fact that $E$ is ample, 
$K+E = 0$, so $\omega_X^{-1} \cong \mc O_X(E)$.
In particular, $\omega_X^{-1}$ is ample.
Thus, by definition,  $X$ is del Pezzo, while $T$ and $B(X, \omega_X^{-1}) = B(X, \mc O_X(E))$ are  equal in large degree.  
But the Hilbert series of  $B(X, \omega_X^{-1})$ is the same as that of $T$, which was calculated above, by \cite[Corollary~III.3.2.5]{Kollar}.
Thus $T = B(X, \omega_X^{-1})$ is the anticanonical coordinate ring of $X$.
 \end{proof}

 \begin{remark}\label{rem:dP2}
 Line modules of self-intersection $(-1)$ play a crucial role in the noncommutative theory we are developing.
 We remark that in the   context of Lemma~\ref{lem:dP}, all line modules correspond to lines of self-intersection $(-1)$.
  Indeed, let $X$ be a smooth del Pezzo surface of degree $\geq 3$, and let $T := B(X, \omega_X^{-1})$ be the anticanonical ring.  
 Let $L$ be a line module over $T$.  
 Then there is a curve $C$ on $X$ so that, for $n$ sufficiently large, $L_{n} = H^0(C, \omega_X^{-n}|_C)$, by Serre's Theorem \cite[Exercise~II.5.9]{Ha}.
 From the Hilbert series of $L$ and Riemann-Roch on $C$, we obtain that $-K_X \cdot C = 1$ and that $\chi(\mc O_C) = 1$.  Thus by \cite[Lemma~III.3.6.1]{Kollar} $C$ is a smooth rational curve of self-intersection $-1$.
 \end{remark}
 
 It may seem counterintuitive that blowing up a point corresponds to constructing a subalgebra, but it can be quite natural in the commutative case, as the following example illustrates.

  \begin{remark}\label{rem:dP3}
 Let $X$ be the blowup of $\PP^2$ at $p$.
 Then the anticanonical ring of $X$ is the subalgebra of $\kk[x,y,z]^{(3)}$ generated by 3-forms vanishing at $p$.
 \end{remark}
 \begin{proof}
 To see this, let $\pi:  X\to \PP^2$ be the blowdown morphism and let $L$ be the exceptional line.
 By \cite[Proposition~V.3.3]{Ha}, $-K_X = \pi^*(-K_{\PP^2})-L$.
 Thus
 \begin{align*}
 H^0(X, \omega_X^{-1}) = H^0(X, -K_X) & = \{ f \in H^0(X, \pi^*(-K_{\PP^2})) : f|_L \equiv 0\} \\
  & = \{ f \in H^0(\PP^2, -K_{\PP^2}) : f(p) =  0\}.
  \end{align*}
  As the anticanonical ring of a degree 8 del Pezzo surface is generated in degree 1, this is sufficient.    \end{proof}
%%%%%%%%%

  \section*{Index of Notation}\label{index}
\begin{multicols}{2}
{\small  \baselineskip 14pt

Auslander-Gorenstein algebra \hfill\pageref{def:gor}

 associated elliptic curve \hfill\pageref{degree-defn}

birational algebras  \hfill\pageref{cg-defn}

Cohen-Macaulay (CM) ring   \hfill\pageref{CM-defn}

CM and MCM modules    
\hfill\pageref{CM-defn} 

connected graded (cg) algebra   \hfill\pageref{cg-defn}

critical and pure modules \hfill\pageref{critical-defn}

 degree of an elliptic algebra \hfill\pageref{degree-defn}

$E^{11}(M)=\Ext^1_T(\Ext^1_T(M,T),T)$  \hfill\pageref{E11-notation}

elliptic algebra $T$ \hfill\pageref{ellipticalg-defn}

finitely graded algebra \hfill\pageref{rem:2.331}

$\GKdim M$ Gelfand-Kirillov dimension  \hfill\pageref{GK-defn}

$g$-torsion and Goldie torsion modules   \hfill\pageref{linear-defn}

$g$-divisible module, $g$-divisible  hull $\widehat{X}$  \hfill\pageref{hull-defn}
 
Hilbert series $\hilb (M)$  \hfill\pageref{Hilbert-defn}

$j(M)$  homological grade of a module  \hfill\pageref{grade-defn}

linear, line and point modules \hfill\pageref{linear-defn}

locally hereditary algebra  \hfill\pageref{loc-her-ring}

locally simple algebra \hfill\pageref{locally simple defn}

$M^\circ$, $T^{\circ}$  localisations \hfill\pageref{Tcirc-defn}

maximal order  \hfill\pageref{maxorder-defn}

minimal elliptic surface  \hfill\pageref{min-model defn}

multiplicity of $M$  \hfill\pageref{multiplicity-defn}

$ \Phi  \Mm :=\bigoplus_{n\in \mathbb{Z}} (\Phi \Mm)_n$,  \hfill\pageref{not:2.31}

quadric elliptic algebra  $T=\QVB^{(2)} $    \hfill\pageref{VDB-defn}

$\rgr R$, $\rqgr R$  (quotient) module category  \hfill\pageref{rqgr-defn}

$Q_{gr}(R)$ graded quotient ring \hfill\pageref{rqgr-defn}

$\Tinf$ special quadric elliptic algebra \hfill\pageref{thm:3.6}

singularity category, $A_1$ singularity   \hfill\pageref{sing-defn}

Sklyanin  algebra, $S= \Skl(a,b, c) $   \hfill\pageref{eg:sklyanin}

Sklyanin  elliptic algebra,  $T=S^{(3)} $  \hfill\pageref{eg:sklyanin}

smooth $\rqgr T$  \hfill\pageref{defn:smooth}

TCR  twisted coordinate ring $B(X,\mathcal{L},\theta)$  \hfill\pageref{TCR-defn}
 
   Van den Bergh quadric $\QVB=\QVB(r)$  \hfill\pageref{VDB-defn}

} \end{multicols}

\bibliographystyle{amsalpha}

\begin{thebibliography}{ATV2}

\bibitem[Aj]{Ajitabh}
K. \ Ajitabh, 
Residue complex for Sklyanin algebras of dimension three,
{\em Adv. Math.} {\bf 144} (1999), no. 2, 137--160. 


\bibitem[Ar]{Ar} M.\ Artin,  Some problems on three-dimensional graded
domains, in:  \emph{Representation  Theory and Algebraic Geometry (Waltham, MA, 1995)},
1-19, London Math.\ Soc.\ Lecture  Note Ser., no.\ 238, Cambridge Univ.
Press, Cambridge, 1997.   

\bibitem[ASc]{ASc} 
M.\ Artin and W.\ F.\ Schelter, Graded algebras of global dimension $3$, {\em Adv. in Math.} {\bf 66} (1987), no. 2, 171--216.


\bibitem[AS]{AS}  M.\ Artin and J.\ T.\ Stafford,  Noncommutative graded domains with quadratic
  growth,  \emph{Invent.\   Math.} \textbf{122} (1995), 231-276.


\bibitem[ATV1]{ATV1990}
M.~Artin, J.~Tate, and M.~{Van den Bergh}, Some algebras associated to
  automorphisms of elliptic curves, \emph{The Grothendieck Festschrift Vol. I},
  Progr. Math., \textbf{Vol.~86}, Birkh{\"a}user Boston, Boston, 1990, pp.~33-85.

\bibitem[ATV2]{ATV2}
M.~Artin, J.~Tate, and M.~{Van den Bergh},  
 Modules over regular algebras of dimension 3,  \emph{Invent.\ Math.}
  \textbf{106} (1991),  335-388.
  
 \bibitem[AV]{AV} M.~Artin  and M.~{Van den Bergh},
Twisted homogeneous coordinate rings \emph{J.\ Algebra} \textbf{133} (1990),  249-271. 

\bibitem[AB]{AB}  M.\ Auslander and M.\ Bridger, Stable module theory, \emph{Mem.\ Amer.\ Math.\ Soc.} \textbf{No. 94},
1969.

         
          
\bibitem[BW]{BW}  Y.    Berest, and G.   Wilson,  Ideal classes of the Weyl algebra and noncommutative projective geometry,
  \emph{Int.\ Math.\ Res.\ Not.}  \textbf{26} (2002), 1347-1396. 
      
       
          
\bibitem[Bj]{Bj} J.-E. Bj\"{o}rk, Filtered Noetherian rings, in  \emph{Noetherian Rings and Their Applications}, (ed: L. W. Small) Math. Surveys and Monographs, \textbf{No. 24}, pp. 59-98, Amer. Math. Soc, Providence, RI, 1987.


\bibitem[Bu]{Buchweitz} R.-O.~Buchweitz, 
 Maximal Cohen-Macaulay modules and Tate-cohomology over Gorenstein rings, 
 unpublished manuscript (1986), see {\tt http://hdl.handle.net/1807/16682}.
 
 \bibitem[BH]{Hipwood-thesis} L.\ D.\ Bush Hipwood, Maximal orders in the Sklyanin algebra, Ph.D. thesis, University of Manchester, 2018. See also arXiv:1812.04137.
 
%\bibitem[Co]{Co} J. H. Cozzens, Maximal orders and reflexive modules, Trans. Amer. Math. Soc. 219 (1976), 323-336. 
 
  
\bibitem[Cr]{Simonthesis}
S.~Crawford, Singularities of noncommutative surfaces,  Ph.D. thesis, University of Edinburgh, 2018.   
Available at https://era.ed.ac.uk/handle/1842/31543.



\bibitem[DL]{DL}   K.\ De Laet,  
On the center of 3-dimensional and 4-dimensional Sklyanin algebras, 
\emph{J. Algebra} \textbf{487} (2017), 244?268. 

\bibitem[Go]{Goodearl}  K.~R.~Goodearl,  Localization and splitting in hereditary noetherian prime rings,
 \emph{Pacific J.\ Math.}\ \textbf{53} (1974), 137-151.
 
 
  \bibitem[Ha]{Ha}  R.~Hartshorne, \emph{ Algebraic geometry}, Springer-Verlag, Berlin, 2006.

\bibitem[Sh]{Shafarevich}  V.\ A.\ Iskovskikh and I.\ R.\ Shafarevich, Algebraic surfaces.  {\em Algebraic geometry, II}, 127--262, Encyclopaedia Math. Sci., 35, Springer, Berlin, 1996.

  
\bibitem[KRS]{KRS}  D.~S.~Keeler, D.~Rogalski  and J.~T.~Stafford, 
 Naive noncommutative blowing up, \emph{Duke Math.\ J.}  \textbf{126} (2003),  491-546.

\bibitem[Ko]{Kollar}J.~Koll\'ar, {\em Rational curves on algebraic varieties}. Ergebnisse der Mathematik und ihrer Grenzgebiete. 3. Folge. A Series of Modern Surveys in Mathematics, {\bf 32}, Springer-Verlag, Berlin, 1996.

\bibitem[KL]{KL} G.\ R.\ Krause and T.\ H.\ Lenagan, \emph{Growth of Algebras and Gelfand-Kirillov Dimension}, 
Research Notes in Math.\  \textbf{Vol.\ 116}, Pitman Boston, 1985.


 \bibitem[Ku]{Kuzmanovich}
 J.\ Kuzmanovich, Localizations of Dedekind prime rings, {\em J. Algebra} {\bf 21} (1972), 378-393.
 
\bibitem[Le]{Lev1992}
T.\ Levasseur, Some properties of noncommutative regular graded
  rings, \emph{Glasgow Math.\ J.} \textbf{34} (1992),  277-300.  
  
\bibitem[LS]{LS}  T.\ Levasseur and S.\ P.\ Smith, Modules over the 4-dimensional Sklyanin algebra, 
\emph{Bull.\ Soc.\ Math.\ France} \textbf{121} (1993), 35-90.
  
  
  \bibitem[ML]{ML}
  L.\ Makar-Limanov, On subalgebras of the first Weyl skewfield, 
\emph{Comm.\  Algebra} \textbf{19}  (1991), 1971-1982.
  
\bibitem[MS]{MS}
 I.~ Mori and S.~P.~ Smith,
 The Grothendieck group of a quantum projective space bundle,
 {\em  K-Theory} {\bf 37} (2006), 263-289. 
  
  \bibitem[NS]{NSt}  T. Nevins and J.   T.   Stafford,
 Sklyanin algebras and Hilbert schemes of 
points,   \emph{Adv.\ Math.} \textbf{210}  (2007),  405-478.

\bibitem[Rg]{R-Sklyanin}
D.~Rogalski,  Blowup subalgebras of the {S}klyanin algebra,
   \emph{Adv.\ Math.} \textbf{226} (2011), 1433-1473.
 

\bibitem[RSS1]{RSS}
D.~Rogalski, S.~J. Sierra and J.~T. Stafford, Algebras in which every
  subalgebra is noetherian, \emph{Proc.\ Amer.\ Math.\ Soc.}
 \textbf{142} (2014), 2983-2990.




\bibitem[RSS2]{RSSlong}
D.~Rogalski, S.~J. Sierra and J.~T. Stafford,  Classifying orders in the
  {S}klyanin algebra,  \emph{Algebra and Number Theory} \textbf{9} (2015), 2055-2119.
 

\bibitem[RSS3]{RSSshort}
D.~Rogalski, S.~J. Sierra and J.~T. Stafford,   Noncommutative blowups of elliptic algebras, \emph{Algebra and Rep.\ Theory} \textbf{18} (2015), 491-529.

\bibitem[RSS4]{RSSblowdown}
D.~Rogalski, S.~J. Sierra and J.~T. Stafford,   Ring-theoretic blowing down: I, \emph{J.\ Noncomm.\ Geom.},  \textbf{11}
(2017), 1465-1520.
 
\bibitem[RS]{RS}   D.~Rogalski  and J.~T. Stafford,   A class of noncommutative projective  surfaces,
 \emph{Proc.\ London Math.\ Soc.},  \textbf{99} (2009), 100-144.  

\bibitem[Sc]{Sco} A.\ H.\  Schofield, Stratiform simple Artinian rings, \emph{Proc.\ London Math.\ Soc.} \textbf{53} (1986), 267-287.

\bibitem[Si]{S-surfclass} S.\ J.\ Sierra, Classifying birationally commutative projective surfaces,
{\em Proc. London Math. Soc.} \textbf{103} (2011),   139-196.

\bibitem[SW]{SW} L.\ W.\ Small and R.\ B.\ Warfield, Jr., Prime affine algebras of Gelfand-Kirillov dimension one
\emph{J.\ Algebra} \textbf{91} (1984), 386-389.

\bibitem[SSf]{SSf}
S.~P.~Smith and J.~T.~Stafford, Regularity of the 4-dimensional
{S}klyanin
  algebra, \emph{Compositio Math}. \textbf{83} (1992), 259--289.
  
\bibitem[SSn]{SS} S.\ P.\ Smith and J.\ M.\ Staniszkis, Irreducible representations of the 4-dimensional Sklyanin algebra 
at points of infinite order, \emph{J.\ Algebra} \textbf{160} (1993),  57-86.



\bibitem[SmV]{SV}
S.~P.~Smith and M~Van~den {B}ergh,  Non-commutative quadric surfaces,
 \emph{J.\  Noncomm.\ Geom.}  \textbf{7} (2013),  817-856.

\bibitem[StV]{StV}
J.~T.~Stafford  and M~Van~den {B}ergh,  
Noncommutative projective curves and surfaces,
\emph{Bull. Amer. Math. Soc.} \textbf{38} (2001), 171-216.


\bibitem[VB1]{VdB1}
M.~{Van den Bergh},  A translation principle for the four-dimensional
  {S}klyanin algebras, \emph{J.\  Algebra}  \textbf{184} (1996),   435-490.
  

\bibitem[VB2]{VdB-blowups}  M.~{Van den Bergh},   Blowing up of non-commutative smooth surfaces,
  \emph{Mem.\ Amer.\ Math.\ Soc.} \textbf{154}  (2001), no.~734.
  



\bibitem[VB3]{VdB3}
M.~{Van den Bergh},   Noncommutative quadrics, 
    \emph{Int.\ Math.\ Res.\ Not.\ IMRN}   (2011), \textbf{no. 17},    398-4026. 


   
\end{thebibliography}

\def\cprime{$'$}
\providecommand{\bysame}{\leavevmode\hbox to3em{\hrulefill}\thinspace}
\providecommand{\MR}{\relax\ifhmode\unskip\space\fi MR } 
\providecommand{\MRhref}[2]{%
  \href{http://www.ams.org/mathscinet-getitem?mr=#1}{#2}
}
\providecommand{\href}[2]{#2}

\end{document}